\documentclass[a4paper, 11pt]{scrartcl}
\pdfoutput=1

\title{Comparison of some Entropy Conservative Numerical Fluxes for the Euler
       Equations}
\author{Hendrik Ranocha}
\date{28th August 2017}
\titlehead{\includegraphics[width=0.5\textwidth]{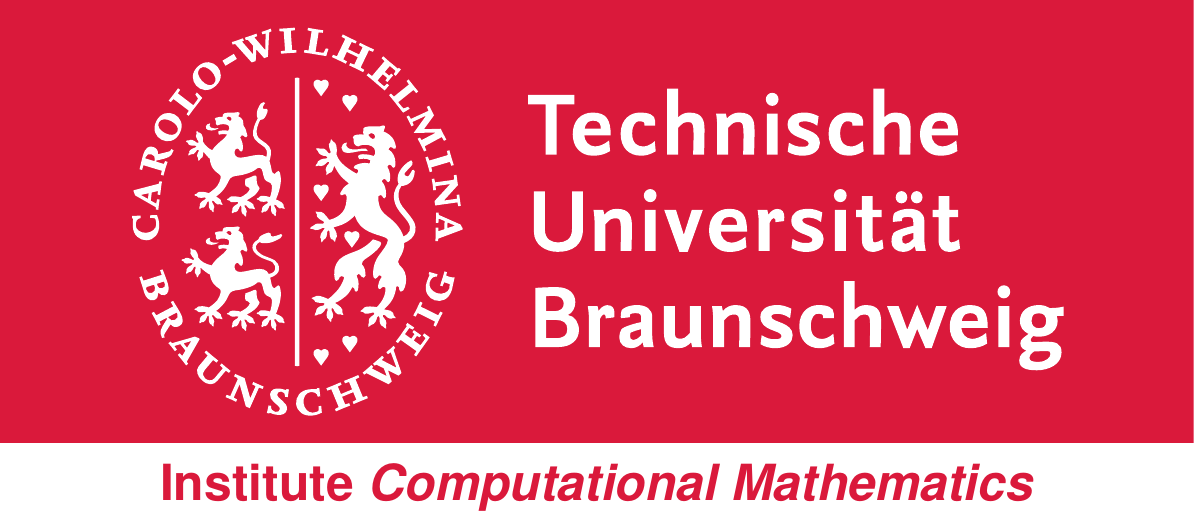}}

\usepackage[utf8]{luainputenc}
\usepackage[english]{babel}
\usepackage{csquotes}

\usepackage[a4paper, top=0.51in, bottom=0.79in, left=0.999in, right=0.99in]{geometry}

\usepackage[plainpages=false,pdfpagelabels,hidelinks,unicode]{hyperref}
\makeatletter
\hypersetup{pdfauthor={\@author}}
\hypersetup{pdftitle={\@title}}
\makeatother

\usepackage[%
  backend=biber,
  style=numeric-comp,
  giveninits=true, uniquename=init, 
  natbib=true,
  url=true,
  doi=true,
  isbn=false,
  backref=false,
  maxnames=99,
  ]{biblatex}
\usepackage{amsmath}
\allowdisplaybreaks
\usepackage{amssymb}
\usepackage{commath}
\usepackage{mathtools}
\usepackage{bbm}

\usepackage{siunitx}

\usepackage{amsthm}
\theoremstyle{plain}
  \newtheorem{theorem}{Theorem}[section]
  \newtheorem{lemma}{Lemma}[section]
  \newtheorem{corollary}{Corollary}[section]
\theoremstyle{definition}
  \newtheorem{remark}{Remark}[section]
  \newtheorem{definition}{Definition}[section]
  \newtheorem{procedure}{Procedure}[section]

\usepackage{color}
\usepackage{graphicx}
\usepackage[small]{caption}
\usepackage{subcaption}

\ifx\useTikzForPlotting\undefined
\else
  \usepackage{pgfplots}
  \pgfplotsset{compat=1.11}
  \usetikzlibrary{external}
\fi

\begingroup\expandafter\expandafter\expandafter\endgroup
\expandafter\ifx\csname pdfsuppresswarningpagegroup\endcsname\relax
\else
\fi

\usepackage{booktabs}
\usepackage{rotating}

\usepackage{multicol}
\usepackage{enumitem}

\usepackage{calc}
\usepackage{xparse}

\renewcommand{\vec}[1]{\underline{#1}}
\NewDocumentCommand{\mat}{mo}{%
  \IfValueTF{#2}{%
    \underline{\underline{#1}}{#2}
  }{%
    \underline{\underline{#1}}\,
  }%
}
\newcommand{\diag}[1]{\operatorname{diag}\left(#1\right)}

\newcommand{\scp}[2]{\left\langle{#1,\, #2}\right\rangle}

\newcommand{\inv}[1]{{#1}^{-1}}
\newcommand{\vect}[1]{\begin{pmatrix} #1 \end{pmatrix}}

\newcommand{\I}{\operatorname{I}}
\newcommand{\fnum}{f^{\mathrm{num}}}
\newcommand{\fnumj}{f^{\mathrm{num},j}}
\newcommand{\fnumx}{f^{\mathrm{num},x}}
\newcommand{\fnumy}{f^{\mathrm{num},y}}
\newcommand{\fnumxy}{f^{\mathrm{num},x/y}}
\newcommand{\vecfnum}{\vec{f}^{\mathrm{num}}}

\newcommand{\fvol}{f^{\mathrm{vol}}}

\newcommand{\Fnum}{F^{\mathrm{num}}}
\newcommand{\pnum}{p^{\mathrm{num}}}

\newcommand{\VOL}{\vec{\mathrm{VOL}}}
\newcommand{\SURF}{\vec{\mathrm{SURF}}}

\renewcommand{\epsilon}{\varepsilon}
\renewcommand{\phi}{\varphi}
\renewcommand{\rho}{\varrho}
\newcommand{\N}{\mathbb{N}}
\newcommand{\R}{\mathbb{R}}

\newsavebox{\DelimiterBox}
\newlength{\DelimiterHeight}
\newlength{\DelimiterDepth}
\newsavebox{\ArgumentBox}
\newlength{\ArgumentHeight}
\newlength{\ArgumentDepth}
\newlength{\ResizedDelimiterHeight}
\newlength{\ResizedDelimiterDepth}

\newcommand{\mean}[1]{%
  \savebox{\ArgumentBox}{$\displaystyle #1$}%
  \settoheight{\ArgumentHeight}{\usebox{\ArgumentBox}}%
  \settodepth{\ArgumentDepth}{\usebox{\ArgumentBox}}%
  \savebox{\DelimiterBox}{$\{\!\!\{$}%
  \settoheight{\DelimiterHeight}{\usebox{\DelimiterBox}}%
  \settodepth{\DelimiterDepth}{\usebox{\DelimiterBox}}%
  \setlength{\ResizedDelimiterHeight}{%
    \maxof{1.2\ArgumentHeight}{\DelimiterHeight}%
  }
  \setlength{\ResizedDelimiterDepth}{%
    \maxof{1.2\ArgumentDepth}{\DelimiterDepth}%
  }
  \raisebox{-\ResizedDelimiterDepth}{%
    \resizebox{\width}{\ResizedDelimiterHeight+\ResizedDelimiterDepth}{%
      \raisebox{\DelimiterDepth}{$\{\!\!\{$}%
    }%
  }
  #1
  \raisebox{-\ResizedDelimiterDepth}{%
    \resizebox{\width}{\ResizedDelimiterHeight+\ResizedDelimiterDepth}{%
      \raisebox{\DelimiterDepth}{$\}\!\!\}$}%
    }%
  }
}

\newcommand{\jump}[1]{%
  \savebox{\ArgumentBox}{$\displaystyle #1$}%
  \settoheight{\ArgumentHeight}{\usebox{\ArgumentBox}}%
  \settodepth{\ArgumentDepth}{\usebox{\ArgumentBox}}%
  \savebox{\DelimiterBox}{$[\![$}%
  \settoheight{\DelimiterHeight}{\usebox{\DelimiterBox}}%
  \settodepth{\DelimiterDepth}{\usebox{\DelimiterBox}}%
  \setlength{\ResizedDelimiterHeight}{%
    \maxof{1.2\ArgumentHeight}{\DelimiterHeight}%
  }
  \setlength{\ResizedDelimiterDepth}{%
    \maxof{1.2\ArgumentDepth}{\DelimiterDepth}%
  }
  \raisebox{-\ResizedDelimiterDepth}{%
    \resizebox{\width}{\ResizedDelimiterHeight+\ResizedDelimiterDepth}{%
      \raisebox{\DelimiterDepth}{$[\![$}%
    }%
  }
  #1
  \raisebox{-\ResizedDelimiterDepth}{%
    \resizebox{\width}{\ResizedDelimiterHeight+\ResizedDelimiterDepth}{%
      \raisebox{\DelimiterDepth}{$]\!]$}%
    }%
  }
}

\newcommand{\logmean}[1]{\mean{#1}_\mathrm{log}}
\newcommand{\geomean}[1]{\mean{#1}_\mathrm{geo}}

\begin{document}

\maketitle

\begin{abstract}
  
Entropy conservation and stability of numerical methods in gas dynamics have received
much interest. Entropy conservative numerical fluxes can be used as ingredients
in two kinds of schemes: Firstly, as building blocks in the subcell flux differencing
form of Fisher and Carpenter (2013) and secondly (enhanced by dissipation) as
numerical surface fluxes in finite volume like schemes.

The purpose of this article is threefold. Firstly, the flux differencing theory is
extended, guaranteeing high-order for general symmetric and consistent numerical fluxes
and investigating entropy stability in a generalised framework of summation-by-parts
operators applicable to multiple dimensions and simplex elements.
Secondly, a general procedure to construct affordable entropy conservative fluxes
is described explicitly and used to derive several new fluxes.
Finally, robustness properties of entropy stable numerical fluxes are investigated
and positivity preservation is proven for several entropy conservative fluxes
enhanced with local Lax-Friedrichs type dissipation operators.
All these theoretical investigations are supplemented with numerical experiments.

\end{abstract}

\section{Introduction}
\label{sec:introduction}

During the last decades, there has been an enduring and increasing interest in
entropy conservation and stability of numerical methods for conservation laws.
It is a topic that still needs further research and this article shall contribute 
to it.

As an ingredient, entropy conservative numerical fluxes can be used in two kinds 
of application: They can be used as volume fluxes in the
flux differencing framework of Fisher and Carpenter \cite{fisher2013high} and 
-- enhanced with additional dissipation operators -- as numerical fluxes in a 
finite volume framework.

In this article, the theory of the flux differencing form by Fisher and Carpenter 
\cite{fisher2013high} is extended. Up to now, high order has only been proven for
the special entropy conservative flux of Tadmor \cite{tadmor1987numerical} but
has been observed for a variety of other numerical fluxes. Here, high order
of accuracy is proven in general for consistent and symmetric numerical fluxes
(\autoref{thm:order}). 
Secondly, for the first time, a formulation of generalised summation-by-parts
operators that can be used in multiple dimensions and on simplex elements is used
to investigate entropy conservation and stability (\autoref{thm:entropy}).

Afterwards, the construction of affordable entropy conservative fluxes is briefly
reviewed, a general procedure (Procedure~\ref{pro:affordable-EC-fluxes}) for their 
derivation is distilled and several new entropy conservative numerical fluxes
are constructed.
Nevertheless, entropy conservation or stability alone are not sufficient. In order
to be robust, numerical schemes for gas dynamics also have to satisfy the physical
constraints given by positivity (non-negativity) of the density and internal
energy / pressure. Thus, additional dissipation / limiting will be necessary in
general, possibly reducing the order of the scheme. Furthermore, general results
about convergence are still unknown. 

However, the aim of this article is not the investigation of convergence but the 
analysis of some entropy conservative and stable schemes.
Therefore, the numerical fluxes are enhanced with several dissipation operators.
Positivity preservation is investigated and most entropy conservative fluxes 
enhanced with local Lax-Friedrichs type dissipation operators are proven to 
preserve non-negativity of the density under a non-vanishing CFL condition
(\autoref{thm:positivity-general}).

This article is organised as follows. At first, some well-known properties of
the Euler equations are summed up in section~\ref{sec:Euler} in order to fix the
notation and for further reference. Afterwards, the extension of the flux 
differencing theory of Fisher and Carpenter \cite{fisher2013high} is presented in 
section~\ref{sec:flux-diff}.
Thereafter, several entropy conservative numerical fluxes are constructed
in sections~\ref{sec:fluxes} and \ref{sec:reversed-fluxes}. To get numerical
surface fluxes usable in finite volume methods, the addition of dissipation is
discussed in section~\ref{sec:surface-fluxes}, especially with regard to positivity
preservation. After that, the methods are tested in section~\ref{sec:numerical-tests}.
Finally, the results are summed up in section~\ref{sec:summary}, conclusions are
drawn and some remaining open problems are formulated.

\section{Euler Equations}
\label{sec:Euler}

In this section, some well known properties of the Euler equations in two space
dimensions are given in order to fix the notation and refer to them later.
The Euler equations are
\begin{equation}
\label{eq:Euler}
\begin{aligned}
  \partial_t
  \underbrace{
  \begin{pmatrix}
    \rho
    \\
    \rho v_x
    \\
    \rho v_y
    \\
    \rho e
  \end{pmatrix}
  }_{= u}
  + \,\partial_x
  \underbrace{
  \begin{pmatrix}
    \rho v_x
    \\
    \rho v_x^2 + p
    \\
    \rho v_x v_y
    \\
    (\rho e + p) v_x
  \end{pmatrix}
  }_{= f_x(u)}
  + \,\partial_y
  \underbrace{
  \begin{pmatrix}
    \rho v_y
    \\
    \rho v_x v_y
    \\
    \rho v_y^2 + p
    \\
    (\rho e + p) v_y
  \end{pmatrix}
  }_{= f_y(u)}
  =
  0,
\end{aligned}
\end{equation}
where $\rho$ is the density of the gas, $v = (v_x,v_y)$ its speed, $\rho v$ the
momentum, $e$ the specific total energy, and $p$ the pressure.
The total energy $\rho e$ can be decomposed into the internal energy $\rho \epsilon$
and the kinetic energy $\frac{1}{2} \rho v^2$, i.e.
$\rho e = \rho \epsilon + \frac{1}{2} \rho v^2$.
For a perfect gas,
\begin{equation}
\label{eq:p}
  p
  = \rho R T
  = (\gamma-1) \rho \epsilon
  = (\gamma-1) \left( \rho e - \frac{1}{2} \rho v^2 \right),
\end{equation}
where $R$ is the gas constant, $T$ the (absolute) temperature, and $\gamma$ the
ratio of specific heats. For air, $\gamma = 1.4$ will be used, unless stated
otherwise.

The (mathematical) entropy (scaled by a constant for convenience, as chosen inter
alia by \cite{ismail2009affordable, chandrashekar2013kinetic}) used is
\begin{equation}
\label{eq:U}
  U = - \frac{\rho s}{\gamma - 1},
\end{equation}
where the (physical) specific entropy is given by
$s = \log \frac{p}{\rho^\gamma} = \log p - \gamma \log \rho$.
With the associated entropy flux $F = U v = - \frac{\rho s}{\gamma - 1} v$,
smooth solutions fulfil $\partial_t U + \partial_x F_x + \partial_y F_y = 0$,
and the entropy inequality
\begin{equation}
\label{eq:Euler-entropy}
  \partial_t U + \partial_x F_x + \partial_y F_y \leq 0
\end{equation}
will be used as an additional admissibility criterion for weak solutions.

For $\rho, p > 0$, the entropy $U(u)$ is strictly convex, and the entropy variables
\begin{equation}
\label{eq:w}
  w = U'(u)
  =
  \left(
    \frac{\gamma}{\gamma-1} - \frac{s}{\gamma-1} - \frac{\rho v^2}{2 p},
    \frac{\rho v_x}{p},
    \frac{\rho v_y}{p},
    -\frac{\rho}{p}
  \right)^T
\end{equation}
can be used interchangeably with the conservative variables $u$. The flux potentials
$\psi_x = \rho v_x$, $\psi_y = \rho v_y$ 
fulfil $\psi_{x/y}'(w) = f_{x/y}\left( u(w) \right)$ and
$F_{x/y} = w \cdot f_{x/y} - \psi_{x/y}$.

\section{Summation-by-Parts Operators and Flux Differencing}
\label{sec:flux-diff}

In this section, summation-by-parts operators are briefly presented in order to
fix the notation. Afterwards, the flux differencing framework of Fisher and Carpenter 
\cite{fisher2013highJCP, fisher2013high} is described and new results about
properties of the resulting semidiscretisations are given.

\subsection{Summation-by-Parts Operators}
\label{subsec:SBP}

Summation-by-parts (SBP) operators are composed of discrete derivative operators
and quadrature rules. These differential and integral operators are compatible,
i.e. they satisfy a discrete analogue of the fundamental theorem of calculus or
the divergence theorem in one or several space dimensions, respectively. Since
the multidimensional framework can be described nearly as briefly as in one
space dimension, multiple dimensions are considered here \cite{hicken2016multidimensional,
ranocha2016sbp}.

Numerical solutions have to be represented as a vector $\vec{u}$ in a finite 
dimensional vector space. In the following, nodal bases are considered, i.e.
the components $\vec{u}_i = u(\xi_i)$ are nodal values at pairwise different
points $\xi_i$. Furthermore, nonlinear operations are performed pointwise on
these nodes as in classical finite difference methods. As an example, the square
of the numerical solution represented as $\vec{u}$ is given as $\vec{u^2}$, where
$(\vec{u^2})_i = (\vec{u}_i)^2$, and the flux $\vec{f}$ is given by the components
$\vec{f}_i = f(\vec{u}_i)$.

\begin{definition}
\label{def:SBP}
  An SBP operator on a $d$ dimensional element $\Omega$ with order of accuracy
  $p \in \N$ consists of the following components.
  \begin{itemize}
    \item
    Derivative operators $\mat{D_j}$, $j \in \set{1, \dots, d}$, approximating the
    partial derivative in the $j$-th coordinate direction. These are required to
    be exact for polynomials of degree $\leq p$.
    
    \item
    A mass matrix $\mat{M}$, approximating the $L_2$ scalar product on $\Omega$ via
    \begin{equation}
      \vec{u}^T \mat{M} \vec{v}
      =
      \scp{\vec{u}}{\vec{v}}_M
      \approx
      \scp{u}{v}_{L_2(\Omega)}
      =
      \int_\Omega u v,
    \end{equation}
    where $u,v$ are functions on $\Omega$ and $\vec{u}, \vec{v}$ their approximations
    in the SBP basis (also known as projections on the grid).
    
    \item
    A restriction operator $\mat{R}$ performing interpolation of functions on the
    volume $\Omega$ to the boundary $\partial \Omega$ of $\Omega$.
    
    \item
    A boundary mass matrix $\mat{B}$ approximating the $L_2$ scalar product on
    $\partial \Omega$ via
    \begin{equation}
      \vec{u_B}^T \mat{B} \vec{v_B}
      =
      \scp{\vec{u_B}}{\vec{v_B}}_B
      \approx
      \scp{u_B}{v_B}_{L_2(\partial \Omega)}
      =
      \int_{\partial \Omega} u_B v_B,
    \end{equation}
    where $u_B,v_B$ are functions on $\partial \Omega$ and $\vec{u_B}, \vec{v_B}$
    their approximations in the SBP basis (also known as projections on the grid).
    
    \item
    Multiplication operators $\mat{N_j}$, $j \in \set{1, \dots, d}$, performing
    multiplication of functions on the boundary $\partial \Omega$ with the $j$-th
    component $n_j$ of the outer unit normal. Thus, if $\vec{u}$ is the approximation
    of a function $u|_{\Omega}$ in the SBP basis, $\mat{R} \vec{u}$ is the
    approximation of $u|_{\partial\Omega}$ on the boundary and 
    $\mat{N_j} \mat{R} \vec{u}$ is the approximation of $n_j \, u|_{\partial\Omega}$,
    where $n_j$ is the $j$-th component of the outer unit normal at $\partial\Omega$.
    
    \item
    The restriction and boundary operators approximate
    $\vec{u}^T \mat{R}[^T] \mat{B} \mat{N_j} \mat{R} \vec{v}
    \approx \int_{\partial \Omega} u v \, n_j$,
    where $n_j$ is the $j$-th component of the outer unit normal $n$, and this
    approximation has to be exact for polynomials of degree $\leq p$.
    
    \item
    Finally, the SBP property
    \begin{equation}
    \label{eq:SBP}
      \mat{M} \mat{D_j} + \mat{D_j}[^T] \mat{M} = \mat{R}[^T] \mat{B} \mat{N_j} \mat{R}
    \end{equation}
    has to be fulfilled, mimicking the divergence theorem on a discrete level
    \begin{equation}\small
      \int_\Omega u \, (\partial_j v)
      + \int_\Omega (\partial_j u) \, v
      \approx
      \vec{u}^T \mat{M} \mat{D_j} \vec{v}
      + \vec{u}^T \mat{D_j}[^T] \mat{M} \vec{v}
      =
      \vec{u}^T \mat{R}[^T] \mat{B} \mat{N_j} \mat{R} \vec{v}
      \approx
      \int_{\partial \Omega} u v \, n_j.
    \end{equation}
  \end{itemize}
\end{definition}
In one space dimension, the index of the derivative and multiplication operators
$\mat{D_1}$, $\mat{N_1}$ will be dropped. Furthermore, the boundary matrix
is the $2 \times 2$ identity matrix $\mat{B} = \diag{1,1}$ and multiplication
with the outer normal is given by $\mat{N} = \diag{-1,1}$.

\begin{remark}
  In Definition~\ref{def:SBP}, the order of accuracy is enforced in the usual 
  sense (of Taylor expansions) by requirements of exactness for polynomials up 
  to some specific degree $p$.
\end{remark}

\begin{remark}
  Multi-dimensional SBP operators can be constructed via tensor products of
  SBP operators in one dimension. However, genuinely multidimensional SBP
  operators on simplices that are not formed as tensor products of lower
  dimensional operators can be constructed as well \cite{hicken2016multidimensional}.
  However, coordinate directions $j$ are still used there, as in many numerical
  schemes known to the author.
\end{remark}

\begin{remark}
\label{rem:only-1dim}
  Since coordinate directions $j$ are used, multidimensional semidiscretisations
  can be obtained via summing up the terms for each space dimension. Therefore,
  only one space dimension is considered in the following.
\end{remark}

\subsection{Flux Differencing Form}
\label{subsec:flux-diff}

In the flux differencing form of Fisher and Carpenter \cite{fisher2013highJCP,
fisher2013high}, (two-point) numerical fluxes and SBP operators are used to 
create high-order semidiscretisations of hyperbolic conservation laws.

\begin{definition}
\label{def:fnum}
  A numerical flux $\fnum$ is a Lipschitz continuous mapping 
  $(u_-,u_+) \mapsto \fnum(u_-,u_+)$ that is consistent with the flux $f$
  of the conservation law \eqref{eq:Euler}, i.e. $\fnum(u,u) = f(u)$.
\end{definition}

\begin{definition}
\label{def:fnum-EC}
  A numerical flux $\fnum$ is entropy conservative (in the sense of Tadmor
  \cite{tadmor1987numerical, tadmor2003entropy}), if
  $(w_i - w_k) \cdot \fnum(u_i, u_k) - (\psi_i - \psi_k) = 0$. Here, $u_{i/k}$
  are conserved variables, $w_{i/k} = w(u_{i/k})$ the corresponding entropy variables
  and $\psi_{i/k} = \psi(u{i/k})$ the flux potentials as in \autoref{sec:Euler}.
  
  A numerical flux $\fnum$ is entropy stable, (in the sense of Tadmor
  \cite{tadmor1987numerical, tadmor2003entropy}), if
  $(w_i - w_k) \cdot \fnum(u_i, u_k) - (\psi_i - \psi_k) \leq 0$.
\end{definition}

\begin{definition}
  A numerical flux is symmetric, if $\fnum(u_i, u_k) = \fnum(u_k, u_i)$.
\end{definition}

A general semidiscretisation of a conservation law $\partial_t u + \operatorname{div} f = 0$
using SBP operators can be written on one element as
\begin{equation}
\label{eq:semidiscretisation}
  \partial_t \vec{u} = -\VOL - \SURF,
\end{equation}
where $\VOL$ are volume terms in the interior of the element and $\SURF$ are 
surface terms coupling the elements.

Numerical fluxes are used in two different ways in semidiscretisations applying
the flux differencing form of Fisher and Carpenter \cite{fisher2013highJCP,
fisher2013high}. Firstly, numerical fluxes $\fnum$ are applied at element
boundaries in order to couple neighbouring elements as in finite volume and 
discontinuous Galerkin methods. Secondly, numerical fluxes are used in the interior
of each element in order to form a discretisation of the divergence of the flux $f$.
In order to distinguish these different applications of fluxes, the second kind
of fluxes will be denoted $\fvol$, since they form the volume terms of the
semidiscretisation.

In the end, a semidiscretisation using the flux differencing form can be written
as \eqref{eq:semidiscretisation}, where the volume and surface terms are given by
\begin{align}
  \label{eq:volume-terms}
  \VOL_i
  &=
  \sum_k 2 \mat{D}[_{i,k}] \fvol( \vec{u}_i, \vec{u}_k ),
  \\
  \label{eq:surface-terms}
  \SURF
  &=
  \mat{M}[^{-1}] \mat{R}[^T] \mat{B} \mat{N} \left(
    \vecfnum - \mat{R} \vec{f}
  \right).
\end{align}
Here, $\vecfnum$ contains the numerical fluxes at the interfaces between elements.
In \eqref{eq:volume-terms}, the sum $\sum_k$ contains contributions from all
points of the nodal basis used to represent the numerical solutions. Heuristically,
the factor $2$ appears in the volume terms \eqref{eq:volume-terms}, since the
volume flux $\fvol$ can be interpreted as a mean value, containing an additional
factor $\frac{1}{2}$. It is justified essentially by Lemma~\ref{lem:expansion-of-fnum}
below.

\begin{remark}
\label{rem:dropping-dimensions}
  As mentioned in Remark~\ref{rem:only-1dim}, semidiscretisations in multiple
  space dimensions contain additional sums over each coordinate direction;
  the fluxes, derivative and multiplication operators have to be indexed by
  space dimension.
\end{remark}

\begin{remark}
\label{rem:FV}
  A first order finite volume method can be obtained in this setting as
  follows. The nodal basis uses only one node inside each element (e.g. the
  midpoint). Thus, the solution is constant in each element and the derivative
  matrix $\mat{D}$ is zero. Moreover, the $1 \times 1$ mass matrix is the length
  $\Delta x$ of the element. Thus, the volume terms \eqref{eq:volume-terms} vanish 
  and the surface terms \eqref{eq:surface-terms} become
  $\SURF = \frac{1}{\Delta x} \left( \fnum_R - \fnum_L \right)$,
  where $\fnum_{L/R}$ is the numerical flux at the left / right boundary of the
  element.
\end{remark}

\subsection{Order of Accuracy}
\label{subsec:order}

Fisher and Carpenter \cite[Theorem 3.1]{fisher2013high} considered diagonal-norm
SBP operators including the boundary nodes and showed that the volume terms
\eqref{eq:volume-terms} of the semidiscretisation \eqref{eq:semidiscretisation}
are approximations to $\partial_x f$ of the same order of accuracy as the SBP 
derivative operators $\mat{D}$, if the two-point flux $\fvol$ used is the entropy 
conservative one proposed by Tadmor~\cite{tadmor1987numerical},
\begin{equation}
\label{eq:fvol-tadmor}
  \fvol(u_i, u_k)
  =
  \int_0^1 f\Big( u\big( w(u_i) + t (w(u_k) - w(u_i)) \big) \Big) \dif t.
\end{equation}
Here, the following generalisation of Theorem 3.1 of \cite{fisher2013high}
will be proven.
\begin{theorem}
\label{thm:order}
  If the numerical flux $\fvol$ is smooth, consistent with the flux~$f$, and 
  symmetric, the volume terms \eqref{eq:volume-terms} are an approximation to 
  $\partial_x f$ of the same order of accuracy as the SBP derivative matrix 
  $\mat{D}$.
\end{theorem}
\begin{remark}
\label{rem:order-multidim}
  An analogous result holds in multiple space dimensions, if the contributions
  of the coordinate directions are summed up as mentioned in Remark~\ref{rem:only-1dim}.
\end{remark}
In order to prove \autoref{thm:order}, the following Lemma will be used.
\begin{lemma}
\label{lem:expansion-of-fnum}
  If the numerical flux $\fvol$ is smooth, consistent with the flux $f$, and symmetric,
  a power series expansion of the $m$-th component $\fvol_m$ can be written as
  \begin{equation}
  \label{eq:flux-diff-condition-order-fvol}
    \fvol_m(w_i, w_k)
    =
    f_m(w_i)
    + \frac{1}{2} f_m'(w_i) \cdot (w_k - w_i)
    + \sum_{\abs\alpha \geq 2} c_\alpha \left(w_k - w_i \right)^\alpha,
  \end{equation}
  where multi-index notation is used, $c_\alpha$ are scalar coefficients, and 
  $w$ denotes any variable, e.g. conservative variables, primitive variables, 
  or entropy variables.
\end{lemma}
Here, multi-index notation \cite[Appendix A.3]{evans2010partial} is used, i.e. 
the multi-index $\alpha = (\alpha_1, \dots, \alpha_n) \in \N_0^n$ has length 
$\abs{\alpha} = \alpha_1 + \dots + \alpha_n$ and for $x \in \R^n$, 
$x^\alpha := x_1^{\alpha_1} \cdot \dots \cdot x_n^{\alpha_n}$. The last term
in \eqref{eq:flux-diff-condition-order-fvol} is a sum over multi-indices $\alpha$
of length $\abs{\alpha} \geq 2$.

\begin{proof}[Proof of Lemma~\ref{lem:expansion-of-fnum}]
  A general power series expansion of the mapping $w_k \mapsto \fvol_m(w_i, w_k)$ 
  around $w_i$ is
  \begin{equation}
    \fvol_m(w_i, w_i)
    + \frac{\partial \fvol_m(w_i,w_k)}{\partial w_k} \Bigg|_{w_k=w_i} \cdot (w_k - w_i)
    + \sum_{\abs\alpha \geq 2} c_\alpha \left(w_k - w_i \right)^\alpha.
  \end{equation}
  Since the numerical flux $\fvol$ is consistent, i.e. $\fvol_m(w_i, w_i) = f_m(w_i)$,
  it suffices to prove $\partial_{w_k}\fvol_m(w_i,w_k) \big|_{w_k=w_i} = 
  \frac{1}{2} f_m'(w_i)$.
  
  Denoting the partial derivative with respect to the $l$-th component of $w_k$ 
  as $\partial_{w_{k,l}} \fvol_m(w_i, w_k)$,  
  \begin{equation}
  \begin{aligned}
    &
    \frac{\partial \fvol_m(w_i,w_k)}{\partial w_{k,l}} \Bigg|_{w_k = w_i = w}
    =
    \lim_{\delta \to 0} \frac{ \fvol_m(w, w + \delta e_l) - \fvol_m(w, w)}{\delta}
    \\
    =&
    \lim_{\delta \to 0} \frac{ \fvol_m(w + \delta e_l, w) - \fvol_m(w, w)}{\delta}
    =
    \frac{\partial \fvol_m(w_i,w_k)}{\partial w_{i,l}} \Bigg|_{{w_k = w_i = w}}.
  \end{aligned}
  \end{equation}
  Due to this symmetry, the $l$-th component of directional derivative of 
  the flux $\fvol_m(w_i,w_k)$ at $w_i=w_k=w$ in direction $\frac{1}{\sqrt{2}} (1,1)^T$ is
  given by
  \begin{equation}
  \begin{aligned}
    &
    \frac{2}{\sqrt{2}} 
    \frac{\partial \fvol_m(w_i,w_k)}{\partial w_{k,l}} \Bigg|_{w_k = w_i = w}
    =
    \lim_{\delta \to 0} 
      \frac{\fvol_m(w + \frac{\delta}{\sqrt{2}} e_l, w + \frac{\delta}{\sqrt{2}}  e_l)
            - \fvol_m(w,w)}{\delta}
    \\
    &=
    \lim_{\delta \to 0} \frac{f_m(w + \frac{\delta}{\sqrt{2}} e_l) - f_m(w)}{\delta}
    =
    \frac{1}{\sqrt{2}} \lim_{\delta \to 0} \frac{f_m(w + \delta e_l) - f_m(w)}{\delta}
    =
    \frac{1}{\sqrt{2}} \frac{\partial f_m(w)}{\partial w_l},
  \end{aligned}
  \end{equation}
  where the consistency $\fvol(w,w) = f(w)$ has been used. This proves the
  desired equality
  $\partial_{w_k}\fvol_m(w_i,w_k) \big|_{w_k=w_i} = \frac{1}{2} f_m'(w_i)$.
  \qed
\end{proof}

\begin{proof}[Proof of \autoref{thm:order}]
  It suffices to consider a single component $m$ of the flux. In order to simplify
  the notation, this index is dropped in the following.  
  Using Lemma~\ref{lem:expansion-of-fnum}, the volume terms \eqref{eq:volume-terms}
  at $x_i$ can be rewritten as
  \begin{equation}
  \label{eq:flux-diff-expanded}
  \begin{aligned}
    \sum_{k} 2 \mat{D}[_{i,k}] \fvol(w_i, w_k)
    =&
    \sum_{k} 2 \mat{D}[_{i,k}] f(w_i)
    + \sum_{k} \mat{D}[_{i,k}] f'(w_i) \cdot (w_k - w_i)
    \\&
    + \sum_{k} \mat{D}[_{i,k}] \sum_{\abs\alpha \geq 2} c_\alpha \left(w_k - w_i \right)^\alpha.
  \end{aligned}
  \end{equation}
  Since the derivative is exact for constants, i.e. $\mat{D} \vec{1} = 0$, the
  first sum on the right hand side of \eqref{eq:flux-diff-expanded} vanishes.
  By the same reason, the second sum can be rewritten as
  \begin{equation}
    \sum_k \mat{D}[_{i,k}] f'(w_i) \cdot (w_k - w_i)
    =
    f'(w_i) \cdot \sum_k \mat{D}[_{i,k}] w_k
  \end{equation}
  and is therefore of the desired order of accuracy.
  Finally, the third summand in \eqref{eq:flux-diff-expanded} is a higher order 
  correction to the product rule. Due to the binomial theorem (in multi-index
  notation),
  \begin{equation}
    \sum_{\abs{\alpha} \geq 2} c_\alpha (w_k - w_i)^\alpha
    =
    \sum_{\abs{\alpha} \geq 2} c_\alpha 
    \sum_{\beta \leq \alpha} \binom{\alpha}{\beta} w_k^\beta (-w_i)^{\alpha-\beta}
  \end{equation}
  for multi-indices $\alpha, \beta \in \N_0^n$. Thus, the third term in
  \eqref{eq:flux-diff-expanded} is
  \begin{equation}
    \sum_{\abs{\alpha} \geq 2} c_\alpha 
    \sum_{\beta \leq \alpha} 
    \binom{\alpha}{\beta}  \sum_k (-w_i)^{\alpha-\beta} \mat{D}[_{i,k}] w_k^\beta,
  \end{equation}
  where $\beta \leq \alpha$ means $\forall j\colon \beta_j \leq \alpha_j$.
  By the product rule, a smooth function $w$ of $x$ satisfies
  \begin{equation}
  \begin{aligned}
    \partial_x w^\beta
    =
    \partial_x \left( w_1^{\beta_1} \dots w_n^{\beta_n} \right)
    =&
    \sum_{j=1}^n \beta_j
    w_1^{\beta_1} \dots w_{j-1}^{\beta_{j-1}} w_j^{\beta_j-1} w_{j+1}^{\beta_{j+1}}
    \dots w_n^{\beta_n} \partial_x w_j
    \\
    =&
    \sum_{j=1}^n \beta_j w^{\beta - e_j} \partial_x w_j,
  \end{aligned}
  \end{equation}
  where $e_j$ is the $j$-th unit vector, $(e_j)_l = \delta_{jl}$. Thus, the third
  sum in \eqref{eq:flux-diff-expanded} is an approximation of the same order of
  accuracy as the derivative matrix $\mat{D}$ to
  \begin{equation}
  \begin{aligned}
    &
    \sum_{\abs{\alpha} \geq 2} c_\alpha 
    \sum_{\beta \leq \alpha} \binom{\alpha}{\beta} (-w_i)^{\alpha-\beta}
    \sum_{j=1}^n \beta_j w_i^{\beta-e_j}
    \sum_k \mat{D}[_{i,k}] w_{k,j}
    \\
    =&
    \sum_{\abs{\alpha} \geq 2} c_\alpha 
    \sum_{j=1}^n
    \underbrace{\sum_{\beta \leq \alpha} \binom{\alpha}{\beta} \beta_j (-\mathbbm{1})^{\alpha-\beta}}
    w_i^{\alpha-e_j}
    \sum_k \mat{D}[_{i,k}] w_{k,j},
  \end{aligned}
  \end{equation}
  where $\mathbbm{1}$ is the vector with components $1$ of the same size as $w_i$ and
  $w_{k,j}$ is the $j$-th component of the vector $w_k$ approximating $w$ at
  $x = x_k$. The sum depending on $\beta$ vanishes, since
  \begin{equation}
  \begin{aligned}
    &
    \partial_{w_j} \left( -\mathbbm{1} + w \right)^\alpha
    =
    \partial_{w_j} \sum_{\beta \leq \alpha} \binom{\alpha}{\beta}
    (-\mathbbm{1})^{\alpha-\beta} w^\beta
    =
    \sum_{\beta \leq \alpha} \binom{\alpha}{\beta}
    (-\mathbbm{1})^{\alpha-\beta} \beta_j w^{\beta-e_j}
    \\
    \stackrel{w=\mathbbm{1}}{\implies}&
    0 = \alpha_j \left( -\mathbbm{1} + \mathbbm{1} \right)^{\alpha-e_j}
    = \sum_{\beta \leq \alpha} \binom{\alpha}{\beta}
    (-\mathbbm{1})^{\alpha-\beta} \beta_j.
  \end{aligned}
  \end{equation}
  Thus, the volume terms \eqref{eq:volume-terms} are an approximation of the same 
  order of accuracy as the derivative matrix $\mat{D}$ to $\partial_x f(w)$ 
  at $x_i$.
  \qed
\end{proof}

\subsection{Entropy Conservation}
\label{subsec:entropy}

Fisher and Carpenter \cite[Theorem 3.2]{fisher2013high} considered diagonal-norm
SBP operators including the boundary nodes in one space dimension and showed that
the semidiscretisation \eqref{eq:semidiscretisation} using the volume terms 
\eqref{eq:volume-terms} and the surface terms \eqref{eq:surface-terms} is semidiscretely
entropy conservative if the volume flux $\fvol$ is consistent, symmetric and
entropy conservative. They proved additional subcell entropy conservation
properties that are not considered here, since its extension to multidimensional
SBP operators on simplices does not seem clear. Instead, only entropy conservation
across elements will be considered.
Here, the following generalisation / variation of Theorem 3.2 of \cite{fisher2013high}
will be proven.
\begin{theorem}
\label{thm:entropy}
  If the numerical (volume) flux $\fvol$ is consistent with $f$, symmetric, and 
  entropy conservative, the nodal mass matrix $\mat{M}$ is diagonal, and the 
  boundary operator $\mat{R}[^T] \mat{B} \mat{N} \mat{R}$ is diagonal, too, the 
  semidiscrete scheme \eqref{eq:semidiscretisation} is entropy conservative / stable 
  across elements, if the numerical (surface) flux $\fnum$ is entropy conservative
  / stable.
\end{theorem}
\begin{proof}[Proof of \autoref{thm:entropy}]
  In the semidiscrete scheme \eqref{eq:semidiscretisation}, the rate of change 
  of the total entropy $\int_\Omega U$ is given as $\od{}{t} \int_\Omega U
  = \vec{w}^T \mat{M} \partial_t \vec{u}$. Multiplying the volume term
  \eqref{eq:volume-terms} with $\vec{w}^T \mat{M}$ results in
  \begin{equation}
    \sum_{i,k} 2 w_i \cdot \left[ \mat{M} \mat{D} \right]_{i,k} \fvol_{i,k}
    =
    \sum_{i,k} w_i \cdot \left[
      \mat{M} \mat{D} + \mat{R}[^T] \mat{B} \mat{N} \mat{R} - \mat{D}[^T] \mat{M}
    \right]_{i,k} \fvol_{i,k},
  \end{equation}
  where $\fvol_{i,k} = \fvol(u_i,u_k)$.
  Since the mass matrix $\mat{M}$ is diagonal,
  \begin{equation}
  \begin{aligned}
    \sum_{i,k} w_i \cdot \left[ \mat{M} \mat{D} - \mat{D}[^T] \mat{M} \right]_{i,k} \fvol_{i,k}
    =&
    \sum_{i,k} \left( M_{ii} D_{ik} - M_{kk} D_{ki} \right) w_i \cdot \fvol_{i,k}
    \\
    =&
    \sum_{i,k} M_{ii} D_{ik} (w_i - w_k) \cdot \fvol_{i,k},
  \end{aligned}
  \end{equation}
  where the indices $i,j$ have been exchanged in the second part of the sum, using
  the symmetry of $\fvol$.
  Then, by entropy conservation $(w_i - w_k) \cdot \fvol_{i,k} = \psi_i - \psi_k$,
  \begin{small}
  \begin{equation}
  \begin{aligned}
    &
    \sum_{i,k} M_{ii} D_{ik} (w_i - w_k) \cdot \fvol_{i,k}
    =
    \sum_{i,k} M_{ii} D_{ik} ( \psi_i - \psi_k )
    =
    - \sum_{i,k} M_{ii} D_{ik} \psi_k
    \\
    =&
    - \sum_{i,k} \left[ \mat{M} \mat{D} \right]_{ik} \psi_k
    =
    - \sum_{i,k} \left[ \mat{R}[^T] \mat{B} \mat{N} \mat{R} - \mat{D}[^T] \mat{M} \right]_{ik}
    \psi_k
    =
    - \sum_{i,k} \left[ \mat{R}[^T] \mat{B} \mat{N} \mat{R} \right]_{ik} \psi_k,
  \end{aligned}
  \end{equation}
  \end{small}%
  since the derivative $\mat{D}$ is exact for constants, i.e. $\mat{D} \vec{1} = 0$.

  The boundary term with the diagonal matrix $\mat{R}[^T] \mat{B} \mat{N} \mat{R}$
  can be written as
  \begin{equation}
  \begin{aligned}
    \sum_{i,k} w_i \cdot \left[ \mat{R}[^T] \mat{B} \mat{N} \mat{R} \right]_{i,k} \fvol_{i,k}
    =&
    \sum_{k} \left[ \mat{R}[^T] \mat{B} \mat{N} \mat{R} \right]_{k,k} w_k \cdot 
    \underbrace{ \fvol_{k,k} }_{=f_k},
  \end{aligned}
  \end{equation}
  since the volume flux $\fvol$ is consistent with the flux $f$. Therefore, the
  total expression becomes
  \begin{equation}
  \begin{aligned}
    \sum_{i,k} 2 w_i \cdot \left[ \mat{M} \mat{D} \right]_{i,k} \fvol_{i,k}
    =&
    \sum_{k} \left[ \mat{R}[^T] \mat{B} \mat{N} \mat{R} \right]_{k,k}
      \underbrace{(w_k \cdot f_k - \psi_k)}_{= F_k}
    \\
    =&
    \sum_{i,k} \left[ \mat{R}[^T] \mat{B} \mat{N} \mat{R} \right]_{i,k} F_k
    =
    \vec{1}^T \mat{R}[^T] \mat{B} \mat{N} \mat{R} \vec{F},
  \end{aligned}
  \end{equation}
  since the entropy flux $F$ is given by $F = w \cdot f - \psi$.
  
  The surface term \eqref{eq:surface-terms} multiplied with $\vec{w}^T \mat{M}$ is
  $\vec{w}^T \mat{R}[^T] \mat{B} \mat{N} \left( \vecfnum - \mat{R} \vec{f} \right)$.
  Thus, the semidiscrete rate of change of the entropy $U$ in one element is
  \begin{equation}
    \vec{w}^T \mat{M} \partial_t \vec{u}
    =
    - \vec{1}^T \mat{R}[^T] \mat{B} \mat{N} \mat{R} \vec{F}
    - \vec{w}^T \mat{R}[^T] \mat{B} \mat{N} \left( \vecfnum - \mat{R} \vec{f} \right).
  \end{equation}
  Since $\mat{R}[^T] \mat{B} \mat{N} \mat{R}$ is diagonal, 
  $\vec{1}^T \mat{R}[^T] \mat{B} \mat{N} \mat{R} \vec{F} = 
  \vec{w}^T \mat{R}[^T] \mat{B} \mat{N} \mat{R} \vec{f}
  - \vec{1}^T \mat{R}[^T] \mat{B} \mat{N} \mat{R} \vec{\psi}$. Therefore,
  \begin{equation}
    \vec{w}^T \mat{M} \partial_t \vec{u}
    =
    \vec{1}^T \mat{R}[^T] \mat{B} \mat{N} \mat{R} \vec{\psi}
    - \vec{w}^T \mat{R}[^T] \mat{B} \mat{N} \vecfnum.
  \end{equation}
  Since the numerical flux is defined per boundary, the contribution of one
  boundary between cells with indices $-$,$+$ is given as
  \begin{equation}
    (w_+ - w_-) \cdot \fnum - (\psi_+ - \psi_-),
  \end{equation}
  which vanishes for an entropy conservative flux $\fnum$ and is non-positive
  for an entropy stable flux.
  \qed
\end{proof}

\begin{remark}
\label{rem:diagonal-boundary-operators}
  A multi-dimensional analogue of \autoref{thm:entropy} can be obtained if the 
  contributions of the coordinate directions are summed up as mentioned in 
  Remark~\ref{rem:only-1dim}. However, the assumption of diagonal mass and
  boundary matrices is still crucial.
  To the author's knowledge, there are no known SBP operators on simplices in
  general with diagonal $\mat{R}[^T] \mat{B} \mat{N_j} \mat{R}$.
  In the framework of Hicken et al. \cite{hicken2016multidimensional}, this operator
  is called $\mathrm{E}_j$ and they mention (Remark 4 in section 4.2) that they 
  have not been able to get diagonal operators that are sufficiently accurate.
  However, using tensor products of Lobatto-Legendre nodes in cubes, these
  operators are diagonal.
  Additionally, it can be conjectured that it is possible to get diagonal operators 
  $\mat{R}[^T] \mat{B} \mat{N_j} \mat{R}$ if enough nodes are added at the
  boundaries. However, this would probably reduce the efficiency of the scheme.
\end{remark}

\begin{remark}
\label{rem:symmetry-and-entropy-conservation}
  To sum up, the semidiscretisation \eqref{eq:semidiscretisation} using the flux
  differencing form of the volume terms \eqref{eq:volume-terms} and surface terms
  \eqref{eq:surface-terms} with entropy stable numerical fluxes is entropy stable
  and high order accurate. However, additional dissipation will still be needed
  in general if discontinuities appear. Thus, it should only be considered as an
  entropy stable baseline scheme.
\end{remark}

\section{Entropy Conservative Fluxes}
\label{sec:fluxes}

In the semidiscrete setting of Tadmor \cite{tadmor1987numerical, tadmor2003entropy},
an entropy conservative numerical flux has to fulfil
\begin{equation}
\label{eq:EC}
  \jump{w} \cdot \fnumj - \jump{\psi_j} = 0,
\end{equation}
where $w$ are the entropy variables \eqref{eq:w}, $\fnumj$ is the numerical
flux in space direction $j$, $\psi_j$ is the flux potential in space direction $j$, and
\begin{equation}
  \jump{a} = a_+ - a_-
\end{equation}
denotes the jump of a quantity, cf. Definition~\ref{def:fnum-EC}.
Since the flux $f_j$ is the gradient of the potential $\psi_j$, i.e.
$f_j = \partial_w \psi_j$, the condition \eqref{eq:EC} for an entropy conservative
flux determines $\fnumj$ as an appropriate mean value of $f_j$.
Indeed, the entropy conservative flux proposed by Tadmor \cite[Equation (4.6a)]{tadmor1987numerical}
has the form of an integral mean
\begin{equation}
\label{eq:fnum-EC-Tadmor}
  \fnumj(w_-, w_+)
  =
  \int_{s=0}^1 f_j\left( u \left(w_- + s (w_+ - w_-) \right) \right) \dif s.
\end{equation}
However, this integral mean value is difficult to compute in general.
Tadmor \cite[Theorem 6.1]{tadmor2003entropy} proposed another integral mean based on
a piecewise linear path in phase space to compute an integral mean similar to
\eqref{eq:fnum-EC-Tadmor}. Nevertheless, another approach will be used here.

Following the well-known proverb \emph{``Differentiation is mechanics, integration
is art.''}, the integral mean can be exchanged by some kind of differential mean.
Sadly, there is no differential mean value theorem giving some kind of numerical
flux fulfilling \eqref{eq:EC} directly in general. However, the mean value theorem
can be used for scalar variables. Indeed, if a scalar conservation law is considered,
both the flux potential $\psi$ and the entropy variable $w$ in \eqref{eq:EC} are
scalar. Thus, the entropy conservative flux $\fnumj$ is uniquely determined as
$\fnumj = \jump{\psi_j} / \jump{w}$ for $\jump{w} \neq 0$.

A similar procedure can be used for systems of conservation laws, where the
entropy variables $w$ are vector-valued. Thus, expressing both the entropy
variables and the flux potential in a common set of scalar variables (e.g.
primitive variables), differential mean values can be used for each scalar variable.
There are several mean values that can be used for this task. The simplest one
is the arithmetic mean
\begin{equation}
\label{eq:arithmetic-mean}
  \mean{a} = (a_- + a_+) / 2,
\end{equation}
with corresponding product and chain rule
\begin{equation}
\label{eq:arithmetic-mean-rule}
  \jump{a b} = \mean{a} \jump{b} + \mean{b} \jump{a},
  \quad
  \jump{a^2} = 2 \mean{a} \jump{a}.
\end{equation}
This is enough to get some entropy conservative fluxes for the shallow water
equations, since the entropy variables $w$ and the flux potential $\psi$ can
be expressed as polynomials in both the primitive variables and the entropy
variables \cite{ranocha2017shallow}.
However, this is not true for the Euler equations. Therefore, other means have to
be used. Roe \cite{roe2006affordable} proposed the logarithmic mean
\begin{equation}
\label{eq:logarithmic-mean}
  \logmean{a} = \frac{a_+ - a_-}{\log a_+ - \log a_-},
\end{equation}
described in \cite{ismail2009affordable}, including a numerically stable
implementation. The corresponding chain rule reads as
\begin{equation}
\label{eq:logarithmic-mean-rule}
  \jump{\log a}
  =
  \jump{a} / \logmean{a}.
\end{equation}
As an example, the derivation of the entropy conservative flux of \cite{roe2006affordable}
is carried out in section~\ref{sec:z}. Thereafter, the basic idea is distilled
as Procedure~\ref{pro:affordable-EC-fluxes} in section~\ref{subsec:general-procedure}.
Afterwards, the framework of kinetic energy preserving fluxes of
\cite{jameson2008formulation} is presented and commented in section~\ref{subsec:KEP}.
Finally, the entropy conservative numerical flux of \cite{chandrashekar2013kinetic}
is given and several new fluxes are constructed. 
While the fluxes of\cite{roe2006affordable} and \cite{chandrashekar2013kinetic}
in sections~\ref{sec:z} and \ref{sec:rho-v-beta} are well-known in the literature,
the other ones are new.

\subsection{Using \texorpdfstring{$\sqrt{\frac{\rho}{p}}, \sqrt{\frac{\rho}{p}} v,
                            \sqrt{\rho p}$}{√(ρ/p),√(ρ/p)v,√(ρp)}
                            as Variables}
\label{sec:z}

The entropy conservative flux of \cite{roe2006affordable, ismail2009affordable}
can be derived using the variables
\begin{equation}
  z_1 := \sqrt{\frac{\rho}{p}}, \quad
  z_2 := \sqrt{\frac{\rho}{p}} v_x, \quad
  z_3 := \sqrt{\frac{\rho}{p}} v_y, \quad
  z_5 := \sqrt{\rho p}.
\end{equation}
In these variables, the flux potentials $\psi_{x/y} = \rho v_{x/y}$, the entropy $s$, 
and the entropy variables $w$ \eqref{eq:w} are given by
\begin{gather*}
\label{eq:psi-w-using-z}
\stepcounter{equation}\tag{\theequation}
  \psi_x = z_2 z_5,
  \quad
  \psi_y = z_3 z_5,
  \quad
  s = -(\gamma+1) \log z_1 - (\gamma-1) \log z_5,
  \\
  w
  =
  \left(
    \frac{\gamma}{\gamma-1} - \frac{s}{\gamma-1} - \frac{1}{2} z_2^2 - \frac{1}{2} z_3^2,\,
    z_1 z_2,\,
    z_1 z_3,\,
    -z_1^2
  \right)^T.
\end{gather*}
Thus, the jumps can be expressed using the chain rules / discrete differential mean
value theorems \eqref{eq:arithmetic-mean-rule} and \eqref{eq:logarithmic-mean-rule} as
\begin{small}
\begin{equation}
\begin{aligned}
  \jump{w_1}
  =&
  - \frac{1}{\gamma-1} \jump{s} - \frac{1}{2} \jump{ z_2^2 } - \frac{1}{2} \jump{ z_3^2 }
  =
    \frac{\gamma+1}{\gamma-1} \jump{ \log z_1 }
  + \jump{ \log z_5 }
  - \frac{1}{2} \jump{ z_2^2 }
  - \frac{1}{2} \jump{ z_3^2 }
  \\
  =&
  \frac{\gamma+1}{\gamma-1} \frac{1}{\logmean{z_1}} \jump{z_1}
  + \frac{1}{\logmean{z_5}} \jump{z_5}
  - \mean{z_2} \jump{z_2}
  - \mean{z_3} \jump{z_3},
\end{aligned}
\end{equation}
\begin{gather*}
\stepcounter{equation}\tag{\theequation}
  \jump{w_2}
  =
  \jump{z_1 z_2}
  =
  \mean{z_1} \jump{z_2}
  + \mean{z_2} \jump{z_1},
  \quad
  \jump{w_3}
  =
  \jump{z_1 z_3}
  =
  \mean{z_1} \jump{z_3}
  + \mean{z_3} \jump{z_1},
  \\
  \jump{w_4}
  =
  - \jump{z_1^2}
  =
  - 2 \mean{z_1} \jump{z_1},
  \quad
  \jump{\psi_x}
  =
  \jump{z_2 z_5}
  =
  \mean{z_2} \jump{z_5}
  + \mean{z_5} \jump{z_2},
\end{gather*}
\end{small}%
and the entropy conservation conditions $\jump{w} \cdot \fnumx - \jump{\psi_x} = 0$
\eqref{eq:EC} becomes ($\fnumy$ analogously)
\begin{small}
\begin{equation}
\begin{aligned}
  0
  =&
  \left(
    \frac{\gamma+1}{\gamma-1} \frac{1}{\logmean{z_1}} \fnumx_{\rho}
    + \mean{z_2} \fnumx_{\rho v_x}
    + \mean{z_3} \fnumx_{\rho v_y}
    - 2 \mean{z_1} \fnumx_{\rho e}
  \right) \jump{z_1}
  \\&
  + \left(
    - \mean{z_2} \fnumx_{\rho}
    + \mean{z_1} \fnumx_{\rho v_x}
    - \mean{z_5}
  \right) \jump{z_2}
  + \left(
    - \mean{z_3} \fnumx_{\rho}
    + \mean{z_1} \fnumx_{\rho v_y}
  \right) \jump{z_3}
  \\&
  + \left(
    \frac{1}{\logmean{z_5}} \fnumx_{\rho}
    - \mean{z_2}
  \right) \jump{z_5}.
\end{aligned}
\end{equation}
\end{small}%
Thus, the fluxes ($\fnumy$ analogously)
\begin{small}
\begin{gather*}
\label{eq:Roe-EC}
\stepcounter{equation}\tag{\theequation}
  \fnumx_{\rho}
  =
  \mean{z_2} \logmean{z_5},
  \;
  \fnumx_{\rho v_x}
  =
  \frac{\mean{z_2}}{\mean{z_1}} \fnumx_{\rho}
  + \frac{\mean{z_5}}{\mean{z_1}},
  \;
  \fnumx_{\rho v_y}
  =
  \frac{\mean{z_3}}{\mean{z_1}} \fnumx_{\rho},
  \\
  \fnumx_{\rho e}
  =
  \frac{1}{2} \frac{\gamma+1}{\gamma-1} \frac{1}{\mean{z_1} \logmean{z_1}} \fnumx_{\rho}
  + \frac{1}{2} \frac{\mean{z_2}}{\mean{z_1}} \fnumx_{\rho v_x}
  + \frac{1}{2} \frac{\mean{z_3}}{\mean{z_1}} \fnumx_{\rho v_y},
\end{gather*}
\end{small}%
proposed (in one space dimension) in \cite{roe2006affordable, ismail2009affordable}
can be seen to be consistent and entropy conservative.
However, by this choice of variables $z$, the pressure influences the numerical
density flux. As explained by Derigs et al. \cite{derigs2016novelAveraging}, this
can lead to problems if there are discontinuities in the pressure, see also
Remark~\ref{rem:positivity-density} and the numerical tests in section
\ref{sec:numerical-tests}.

\subsection{General Procedure to Construct Affordable Entropy Conservative Fluxes}
\label{subsec:general-procedure}

The general procedure to construct affordable entropy conservative fluxes that
has been mentioned in the introduction of this section has been exemplified
in the previous section~\ref{sec:z}. Similarly, the affordable, entropy conservative 
numerical flux of \cite{chandrashekar2013kinetic} can be constructed using the 
same general approach that can be described as
\begin{procedure}
\label{pro:affordable-EC-fluxes}
  \begin{enumerate}
    \item 
    Express the flux potentials $\psi$ and the entropy variables $w$
    \eqref{eq:w} using the chosen set of variables.
    
    \item
    Express the jumps of $\psi, w$ as products of some mean values and jumps of the
    chosen variables using some kind of product/chain rule as in the mean
    value theorem.
  \end{enumerate}
\end{procedure}

\subsection{Kinetic Energy Preservation}
\label{subsec:KEP}

Besides entropy conservation / stability (cf. Definition~\ref{def:fnum-EC}), 
kinetic energy preservation has been proposed as a desirable property of numerical
fluxes for the Euler equations \eqref{eq:Euler} and has therefore been used as a 
design criterion \cite{jameson2008formulation, chandrashekar2013kinetic, gassner2016split}.
The kinetic energy $\frac{1}{2} \rho v^2$ satisfies (for smooth solutions)
\begin{equation}
  \partial_t \left( \frac{1}{2} \rho v^2 \right)
  + \partial_x \left( \frac{1}{2} \rho v^2 v_x \right)
  + \partial_y \left( \frac{1}{2} \rho v^2 v_y \right)
  + v_x \partial_x p
  + v_y \partial_y p
  = 0.
\end{equation}
In order to mimic this behaviour discretely in one space dimension, 
Jameson~\cite[Equation (2.23)]{jameson2008formulation} formulated the following 
condition, also used in \cite[Section 3]{chandrashekar2013kinetic} and 
\cite[Equation (3.23)]{gassner2016split}.
\begin{definition}
\label{def:fnum-KEP}
  A numerical flux for the Euler equations \eqref{eq:Euler} is said to be
  kinetic energy preserving, if the momentum flux $\fnum_{\rho v}$ can be written
  as $\fnum_{\rho v} = \mean{v} \fnum_{\rho} + \pnum$, where $\pnum$ is a 
  consistent approximation of the pressure.
\end{definition}
\begin{remark}
\label{rem:fnum-KEP}
  Every consistent numerical flux $\fnum_{\rho v}$ can be written in the form
  required in Definition~\ref{def:fnum-KEP}, if some differences are accepted, 
  i.e. if $\pnum = \fnum_{\rho v} - \mean{v} \fnum_{\rho}$ is accepted as numerical
  approximation of the pressure. Thus, Definition~\ref{def:fnum-KEP} alone does 
  not seem to yield a useful criterion for the construction of numerical fluxes,
  i.e. the structural property ``kinetic  energy preserving'' is not well-defined.
\end{remark}
\begin{remark}
\label{rem:design-criteria}
  In this article, some properties of numerical fluxes proposed in the literature
  as desirable design criteria are used, including entropy conservation / 
  stability (Definition~\ref{def:fnum-EC}) and kinetic energy preservation 
  (Definition~\ref{def:fnum-KEP}).
  It is not the purpose of this article to judge these criteria or attempt to
  use them for convergence proofs. However, as described in Remark~\ref{rem:fnum-KEP},
  the property ``kinetic energy preserving'' should be considered carefully. 
  Furthermore, robustness properties of entropy conservative numerical fluxes
  enhanced with additional dissipation operators are investigated in 
  section~\ref{sec:surface-fluxes}. There, some fluxes fulfilling an additional
  structural property are proven to preserve the non-negativity of the density
  under a non-vanishing CFL condition.
\end{remark}

\subsection{Using \texorpdfstring{$\rho, v, \beta$}{ρ,v,β} as Variables}
\label{sec:rho-v-beta}

Using the inverse of the temperature
\begin{equation}
\label{eq:beta}
  \beta = \frac{1}{2 R T} = \frac{\rho}{2 p},
\end{equation}
\citet{chandrashekar2013kinetic} derived some entropy conservative fluxes.
The flux potential and the entropy variables are
\begin{gather}
\label{eq:psi-w-using-rho-v-beta}
  \psi_x = \rho v_x,
  \quad
  \psi_y= \rho v_y,
  \\
  w
  =
  \left(
    \frac{\gamma}{\gamma-1} - \frac{s}{\gamma-1} - \beta v^2,\,
    2 \beta v_x,\,
    2 \beta v_y,\,
    -2 \beta
  \right)^T,
  \quad
  s = \log \frac{p}{\rho^\gamma}
  = - \log \beta - (\gamma-1) \log \rho - \log 2.
\end{gather}

\subsubsection{Variant 1}

Writing the jumps using the chain rules \eqref{eq:arithmetic-mean-rule} and
\eqref{eq:logarithmic-mean-rule} as
\begin{equation}
\begin{aligned}
  \jump{w_1}
  =&
  - \frac{1}{\gamma-1} \jump{s} - \jump{ \beta v^2 }
  =
  \jump{ \log \rho }
  + \frac{1}{\gamma-1} \jump{ \log \beta }
  - \jump{ \beta v^2 }
  \\
  =&
  \frac{1}{\logmean{\rho}} \jump{\rho}
  + \frac{1}{\gamma-1} \frac{1}{\logmean{\beta}} \jump{\beta}
  - \mean{v_x^2} \jump{\beta}
  - \mean{v_y^2} \jump{\beta}
  - 2 \mean{\beta} \mean{v_x} \jump{v_x}
  - 2 \mean{\beta} \mean{v_y} \jump{v_y},
  \\
  \jump{w_2}
  =&
  2 \jump{\beta v_x}
  =
  2 \mean{\beta} \jump{v_x}
  + 2 \mean{v_x} \jump{\beta},
  \\
  \jump{w_3}
  =&
  2 \jump{\beta v_y}
  =
  2 \mean{\beta} \jump{v_y}
  + 2 \mean{v_y} \jump{\beta},
  \\
  \jump{w_4}
  =&
  - 2 \jump{\beta},
  \\
  \jump{\psi_x}
  =&
  \jump{\rho v_x}
  =
  \mean{\rho} \jump{v_x}
  + \mean{v_x} \jump{\rho},
  \\
  \jump{\psi_y}
  =&
  \jump{\rho v_y}
  =
  \mean{\rho} \jump{v_y}
  + \mean{v_y} \jump{\rho},
\end{aligned}
\end{equation}
the entropy conservation conditions $\jump{w} \cdot f^{\mathrm{num},x/y}
- \jump{\psi_{x/y}} = 0$ \eqref{eq:EC} become
\begin{equation}
\begin{aligned}
  0
  =&
  \left(
    \frac{1}{\logmean{\rho}} \fnumx_{\rho}
    - \mean{v_x}
  \right) \jump{\rho}
  + \left(
    - 2 \mean{\beta} \mean{v_x} \fnumx_{\rho}
    + 2 \mean{\beta} \fnumx_{\rho v_x}
    - \mean{\rho}
  \right) \jump{v_x}
  \\&
  + \left(
    - 2 \mean{\beta} \mean{v_y} \fnumx_{\rho}
    + 2 \mean{\beta} \fnumx_{\rho v_y}
  \right) \jump{v_y}
  + \Bigg(
    \frac{1}{\gamma-1} \frac{1}{\logmean{\beta}} \fnumx_{\rho}
    - \mean{v_x^2} \fnumx_{\rho}
    \\&
    - \mean{v_y^2} \fnumx_{\rho}
    + 2 \mean{v_x} \fnumx_{\rho v_x}
    + 2 \mean{v_y} \fnumx_{\rho v_y}
    - 2 \fnumx_{\rho e}
  \Bigg) \jump{\beta},
  \\
  0
  =&
  \left(
    \frac{1}{\logmean{\rho}} \fnumy_{\rho}
    - \mean{v_y}
  \right) \jump{\rho}
  + \left(
    - 2 \mean{\beta} \mean{v_x} \fnumy_{\rho}
    + 2 \mean{\beta} \fnumy_{\rho v_x}
  \right) \jump{v_x}
  \\&
  + \left(
    - 2 \mean{\beta} \mean{v_y} \fnumy_{\rho}
    + 2 \mean{\beta} \fnumy_{\rho v_y}
    - \mean{\rho}
  \right) \jump{v_y}
  + \Bigg(
    \frac{1}{\gamma-1} \frac{1}{\logmean{\beta}} \fnumy_{\rho}
    - \mean{v_x^2} \fnumy_{\rho}
    \\&
    - \mean{v_y^2} \fnumy_{\rho}
    + 2 \mean{v_x} \fnumy_{\rho v_x}
    + 2 \mean{v_y} \fnumy_{\rho v_y}
    - 2 \fnumy_{\rho e}
  \Bigg) \jump{\beta}.
\end{aligned}
\end{equation}
Thus, the fluxes
\begin{equation}
\label{eq:Chandrashekar-rho-v-beta-EC-KEP}
\begin{aligned}
  \fnumx &
  \begin{dcases}
  \fnumx_{\rho}
  =
  \logmean{\rho} \mean{v_x},
  \\
  \fnumx_{\rho v_x}
  =
  \mean{v_x} \fnumx_{\rho} + \frac{\mean{\rho}}{2 \mean{\beta}},
  \\
  \fnumx_{\rho v_y}
  =
  \mean{v_y} \fnumx_{\rho},
  \\
  \fnumx_{\rho e}
  =
  \frac{1}{2(\gamma-1)} \frac{1}{\logmean{\beta}} \fnumx_{\rho}
  - \frac{ \mean{v_x^2} + \mean{v_y^2} }{2}  \fnumx_{\rho}
  + \mean{v_x} \fnumx_{\rho v_x}
  + \mean{v_y} \fnumx_{\rho v_y},
  \end{dcases}
  \\
  \fnumy &
  \begin{dcases}
  \fnumy_{\rho}
  =
  \logmean{\rho} \mean{v_y},
  \\
  \fnumy_{\rho v_x}
  =
  \mean{v_x} \fnumy_{\rho},
  \\
  \fnumy_{\rho v_y}
  =
  \mean{v_y} \fnumy_{\rho} + \frac{\mean{\rho}}{2 \mean{\beta}},
  \\
  \fnumy_{\rho e}
  =
  \frac{1}{2(\gamma-1)} \frac{1}{\logmean{\beta}} \fnumy_{\rho}
  - \frac{ \mean{v_x^2} + \mean{v_y^2} }{2}  \fnumy_{\rho}
  + \mean{v_x} \fnumy_{\rho v_x}
  + \mean{v_y} \fnumy_{\rho v_y},
  \end{dcases}
\end{aligned}
\end{equation}
proposed by \citet{chandrashekar2013kinetic} can be seen to be entropy conservative.
Since $p = \frac{\rho}{2\beta}$, they are consistent.
Additionally, they are kinetic energy preserving with numerical pressure flux
$\pnum = \frac{\mean{\rho}}{2 \mean{\beta}}$.

\subsubsection{Variant 2}

Choosing another possibility to split the jumps
\begin{equation}
\begin{aligned}
  \jump{ \beta v_{x/y}^2 }
  =&
  \mean{\beta v_{x/y}} \jump{v_{x/y}} + \mean{v_{x/y}} \jump{\beta v_{x/y}}
  \\
  =& 
  \mean{\beta v_{x/y}} \jump{v_{x/y}} + \mean{\beta} \mean{v_{x/y}} \jump{v_{x/y}}
  + \mean{v_{x/y}}^2 \jump{\beta},
\end{aligned}
\end{equation}
the entropy conservation conditions \eqref{eq:EC} can be written as
\begin{equation}
\begin{aligned}
  0
  =&
  \left(
    \frac{1}{\logmean{\rho}} \fnumx_{\rho}
    - \mean{v_x}
  \right) \jump{\rho}
  \\&
  + \left(
    - \mean{\beta v_x} \fnumx_{\rho}
    - \mean{\beta} \mean{v_x} \fnumx_{\rho}
    + 2 \mean{\beta} \fnumx_{\rho v_x}
    - \mean{\rho}
  \right) \jump{v_x}
  \\&
  + \left(
    - \mean{\beta v_y} \fnumx_{\rho}
    - \mean{\beta} \mean{v_y} \fnumx_{\rho}
    + 2 \mean{\beta} \fnumx_{\rho v_y}
  \right) \jump{v_y}
  \\&
  + \Bigg(
    \frac{1}{\gamma-1} \frac{1}{\logmean{\beta}} \fnumx_{\rho}
    - \mean{v_x^2} \fnumx_{\rho}
    - \mean{v_y^2} \fnumx_{\rho}
    \\&\qquad
    + 2 \mean{v_x} \fnumx_{\rho v_x}
    + 2 \mean{v_y} \fnumx_{\rho v_y}
    - 2 \fnumx_{\rho e}
  \Bigg) \jump{\beta},
  \\
  0
  =&
  \left(
    \frac{1}{\logmean{\rho}} \fnumy_{\rho}
    - \mean{v_y}
  \right) \jump{\rho}
  \\&
  + \left(
    - \mean{\beta v_x} \fnumy_{\rho}
    - \mean{\beta} \mean{v_x} \fnumy_{\rho}
    + 2 \mean{\beta} \fnumy_{\rho v_x}
  \right) \jump{v_x}
  \\&
  + \left(
    - \mean{\beta v_y} \fnumy_{\rho}
    - \mean{\beta} \mean{v_y} \fnumy_{\rho}
    + 2 \mean{\beta} \fnumy_{\rho v_y}
    - \mean{\rho}
  \right) \jump{v_y}
  \\&
  + \Bigg(
    \frac{1}{\gamma-1} \frac{1}{\logmean{\beta}} \fnumy_{\rho}
    - \mean{v_x^2} \fnumy_{\rho}
    - \mean{v_y^2} \fnumy_{\rho}
    \\&\qquad
    + 2 \mean{v_x} \fnumy_{\rho v_x}
    + 2 \mean{v_y} \fnumy_{\rho v_y}
    - 2 \fnumy_{\rho e}
  \Bigg) \jump{\beta}.
\end{aligned}
\end{equation}
Thus, the fluxes
\begin{equation}
\label{eq:Chandrashekar-rho-v-beta-EC-notKEP}
\begin{aligned}
  \fnumx&
  \begin{dcases}
  \fnumx_{\rho}
  =
  \logmean{\rho} \mean{v_x},
  \\
  \fnumx_{\rho v_x}
  =
  \frac{ \mean{\beta v_x}
  + \mean{\beta} \mean{v_x} }{ 2 \mean{\beta} } \fnumx_{\rho}
  + \frac{\mean{\rho}}{2 \mean{\beta}},
  \\
  \fnumx_{\rho v_y}
  =
  \frac{ \mean{\beta v_y}
  + \mean{\beta} \mean{v_y} }{ 2 \mean{\beta} } \fnumx_{\rho},
  \\
  \fnumx_{\rho e}
  =
  \frac{1}{2(\gamma-1)} \frac{1}{\logmean{\beta}} \fnumx_{\rho}
  - \frac{  \mean{v_x^2} + \mean{v_y^2} }{2} \fnumx_{\rho}
  + \mean{v_x} \fnumx_{\rho v_x}
  + \mean{v_y} \fnumx_{\rho v_y},
  \end{dcases}
  \\
  \fnumy&
  \begin{dcases}
  \fnumy_{\rho}
  =
  \logmean{\rho} \mean{v_y},
  \\
  \fnumy_{\rho v_x}
  =
  \frac{ \mean{\beta v_x}
  + \mean{\beta} \mean{v_x} }{ 2 \mean{\beta} } \fnumy_{\rho},
  \\
  \fnumy_{\rho v_y}
  =
  \frac{ \mean{\beta v_y}
  + \mean{\beta} \mean{v_y} }{ 2 \mean{\beta} } \fnumy_{\rho}
  + \frac{\mean{\rho}}{2 \mean{\beta}},
  \\
  \fnumy_{\rho e}
  =
  \frac{1}{2(\gamma-1)} \frac{1}{\logmean{\beta}} \fnumy_{\rho}
  - \frac{  \mean{v_x^2} + \mean{v_y^2} }{2} \fnumy_{\rho}
  + \mean{v_x} \fnumy_{\rho v_x}
  + \mean{v_y} \fnumy_{\rho v_y},
  \end{dcases}
\end{aligned}
\end{equation}
proposed by \citet{chandrashekar2013kinetic} can be seen to be entropy conservative.
Since the property ``kinetic energy preserving'' is not well-defined, they could
possibly be considered as kinetic energy preserving, cf. Remark~\ref{rem:fnum-KEP}.

\subsection{Using \texorpdfstring{$\rho, v, \frac{1}{p}$}{ρ,v,1/p} as Variables}
\label{sec:rho-v-1/p}

Using the variables $\rho, v, \frac{1}{p}$, the flux potentials
and the entropy variables \eqref{eq:w} can be written as
\begin{gather}
\label{eq:psi-w-using-rho-v-1/p}
  \psi_x = \rho v_x,
  \quad
  \psi_y = \rho v_y,
  \\
  w
  =
  \left(
    \frac{\gamma}{\gamma-1} - \frac{s}{\gamma-1} - \frac{\rho v^2}{2 p},\,
    \frac{\rho v_x}{p},\,
    \frac{\rho v_y}{p},\,
    -\frac{\rho}{p}\,
  \right)^T,
  \quad
  s = \log \frac{p}{\rho^\gamma} = - \log \frac{1}{p} - \gamma \log \rho.
\end{gather}

One variant to write the jumps is given by setting
\begin{equation}
\begin{aligned}
  \jump{w_1}
  =&
  - \frac{1}{\gamma-1} \jump{s} - \frac{1}{2} \jump{ \frac{\rho v^2}{p} }
  =
  \frac{1}{\gamma-1} \jump{ \log \frac{1}{p}
  + \gamma \log \rho}
  - \frac{1}{2} \jump{ \frac{\rho v^2}{p} }
  \\
  =&
  \frac{1}{\gamma-1} \frac{1}{\logmean{\inv p}} \jump{\inv p}
  + \frac{\gamma}{\gamma-1} \frac{1}{\logmean{\rho}} \jump{\rho}
  - \mean{\frac{\rho}{p}} \mean{v_x} \jump{v_x}
  \\&
  - \mean{\frac{\rho}{p}} \mean{v_y} \jump{v_y}
  - \mean{\rho} \frac{ \mean{v_x^2} + \mean{v_y^2} }{2} \jump{\inv p}
  - \frac{ \mean{v_x^2} + \mean{v_y^2} }{2} \mean{\inv p} \jump{\rho},
  \\
  \jump{w_2}
  =&
  \jump{ \frac{\rho v_x}{p} }
  =
  \mean{\frac{\rho}{p}} \jump{v_x}
  + \mean{\rho} \mean{v_x} \jump{\inv p}
  + \mean{v_x} \mean{\inv p} \jump{\rho},
  \\
  \jump{w_3}
  =&
  \jump{ \frac{\rho v_y}{p} }
  =
  \mean{\frac{\rho}{p}} \jump{v_y}
  + \mean{\rho} \mean{v_y} \jump{\inv p}
  + \mean{v_y} \mean{\inv p} \jump{\rho},
  \\
  \jump{w_4}
  =&
  - \jump{ \frac{\rho}{p} }
  =
  - \mean{\rho} \jump{\inv p}
  - \mean{\inv p} \jump{\rho},
  \\
  \jump{\psi_x}
  =&
  \jump{\rho v_x}
  =
  \mean{\rho} \jump{v_x}
  + \mean{v_x} \jump{\rho},
  \\
  \jump{\psi_y}
  =&
  \jump{\rho v_y}
  =
  \mean{\rho} \jump{v_y}
  + \mean{v_y} \jump{\rho}.
\end{aligned}
\end{equation}
Therefore, the entropy conservation conditions \eqref{eq:EC} can be written as
\begin{equation}
\begin{aligned}
  0
  =&
  \Bigg(
    \frac{\gamma}{\gamma-1} \frac{1}{\logmean{\rho}} \fnumx_{\rho}
    - \frac{ \mean{v_x^2} + \mean{v_y^2} }{2} \mean{\inv p} \fnumx_{\rho}
    + \mean{\inv p} \mean{v_x} \fnumx_{\rho v_x}
    \\&\qquad
    + \mean{\inv p} \mean{v_y} \fnumx_{\rho v_y}
    - \mean{\inv p} \fnumx_{\rho e}
    - \mean{v_x} 
  \Bigg) \jump{\rho}
  \\&
  + \left(
    - \mean{\frac{\rho}{p}} \mean{v_x} \fnumx_{\rho}
    + \mean{\frac{\rho}{p}} \fnumx_{\rho v_x}
    - \mean{\rho}
  \right) \jump{v_x}
  + \left(
    - \mean{\frac{\rho}{p}} \mean{v_y} \fnumx_{\rho}
    + \mean{\frac{\rho}{p}} \fnumx_{\rho v_y}
  \right) \jump{v_y}
  \\&
  + \Bigg(
    \frac{1}{\gamma-1} \frac{1}{\logmean{\inv p}} \fnumx_{\rho}
    - \mean{\rho} \frac{ \mean{v_x^2} + \mean{v_y^2} }{2} \fnumx_{\rho}
    + \mean{\rho} \mean{v_x} \fnumx_{\rho v_x}
    \\&\qquad
    + \mean{\rho} \mean{v_y} \fnumx_{\rho v_y}
    - \mean{\rho} \fnumx_{\rho e}
  \Bigg) \jump{\inv p},
\end{aligned}
\end{equation}
\begin{equation}
\begin{aligned}
  0
  =&
  \Bigg(
    \frac{\gamma}{\gamma-1} \frac{1}{\logmean{\rho}} \fnumy_{\rho}
    - \frac{ \mean{v_x^2} + \mean{v_y^2} }{2} \mean{\inv p} \fnumy_{\rho}
    + \mean{\inv p} \mean{v_x} \fnumy_{\rho v_x}
    \\&\qquad
    + \mean{\inv p} \mean{v_y} \fnumy_{\rho v_y}
    - \mean{\inv p} \fnumy_{\rho e}
    - \mean{v_y} 
  \Bigg) \jump{\rho}
  \\&
  + \left(
    - \mean{\frac{\rho}{p}} \mean{v_x} \fnumy_{\rho}
    + \mean{\frac{\rho}{p}} \fnumy_{\rho v_x}
  \right) \jump{v_x}
  + \left(
    - \mean{\frac{\rho}{p}} \mean{v_y} \fnumy_{\rho}
    + \mean{\frac{\rho}{p}} \fnumy_{\rho v_y}
    - \mean{\rho}
  \right) \jump{v_y}
  \\&
  + \Bigg(
    \frac{1}{\gamma-1} \frac{1}{\logmean{\inv p}} \fnumy_{\rho}
    - \mean{\rho} \frac{ \mean{v_x^2} + \mean{v_y^2} }{2} \fnumy_{\rho}
    + \mean{\rho} \mean{v_x} \fnumy_{\rho v_x}
    \\&\qquad
    + \mean{\rho} \mean{v_y} \fnumy_{\rho v_y}
    - \mean{\rho} \fnumy_{\rho e}
  \Bigg) \jump{\inv p}.
\end{aligned}
\end{equation}
Thus, the fluxes
\begin{equation}
\label{eq:rho-v-1/p-EC-KEP}
\begin{aligned}
  \fnumx&
  \begin{dcases}
  \fnumx_{\rho}
  =
  (\gamma-1) \left(
    \frac{\gamma}{\logmean{\rho}}
    - \frac{\mean{\inv p}}{\logmean{\inv p} \mean{\rho}}
  \right)^{-1} \mean{v_x},
  \\
  \fnumx_{\rho v_x}
  =
  \mean{v_x} \fnumx_{\rho}
  + \frac{\mean{\rho}}{\mean{\rho / p}},
  \\
  \fnumx_{\rho v_y}
  =
  \mean{v_y} \fnumx_{\rho},
  \\
  \fnumx_{\rho e}
  =
  \left(
    \frac{1}{\gamma-1} \frac{1}{\logmean{\inv p} \mean{\rho}}
    + \mean{v_x}^2 + \mean{v_y}^2
    - \frac{  \mean{v_x^2} + \mean{v_y^2} }{2} 
  \right) \fnumx_{\rho}
  + \frac{\mean{\rho}\mean{v_x}}{\mean{\rho / p}},
  \end{dcases}
  \\
  \fnumy&
  \begin{dcases}
  \fnumy_{\rho}
  =
  (\gamma-1) \left(
    \frac{\gamma}{\logmean{\rho}}
    - \frac{\mean{\inv p}}{\logmean{\inv p} \mean{\rho}}
  \right)^{-1} \mean{v_y},
  \\
  \fnumy_{\rho v_x}
  =
  \mean{v_x} \fnumy_{\rho},
  \\
  \fnumy_{\rho v_y}
  =
  \mean{v_y} \fnumy_{\rho}
  + \frac{\mean{\rho}}{\mean{\rho / p}},
  \\
  \fnumy_{\rho e}
  =
  \left(
    \frac{1}{\gamma-1} \frac{1}{\logmean{\inv p} \mean{\rho}}
    + \mean{v_x}^2 + \mean{v_y}^2
    - \frac{  \mean{v_x^2} + \mean{v_y^2} }{2} 
  \right) \fnumy_{\rho}
  + \frac{\mean{\rho}\mean{v_y}}{\mean{\rho / p}},
  \end{dcases}
\end{aligned}
\end{equation}
can be seen to be entropy conservative and consistent.
Additionally, they are kinetic energy preserving with numerical pressure flux
$\pnum = \frac{\mean{\rho}}{\mean{\rho / p}}$.

Again, similarly to the flux \eqref{eq:Roe-EC} of \citet{roe2006affordable,
ismail2009affordable}, the pressure influences the numerical density flux,
leading to some problems as explained by \citet{derigs2016novelAveraging},
see also Remark~\ref{rem:positivity-density} and the numerical tests in section
\ref{sec:numerical-tests}.

\subsection{Using \texorpdfstring{$\rho, v, p$}{ρ,v,p} as Variables}
\label{sec:rho-v-p}

Using the variables $\rho, v, p$, the flux potentials
and the entropy variables \eqref{eq:w} can be written as
\begin{gather}
\label{eq:psi-w-using-rho-v-p}
  \psi_x = \rho v_x,
  \quad
  \psi_y = \rho v_y,
  \\
  w
  =
  \left(
    \frac{\gamma}{\gamma-1} - \frac{s}{\gamma-1} - \frac{\rho v^2}{2 p},\,
    \frac{\rho v_x}{p},\,
    \frac{\rho v_y}{p},\,
    -\frac{\rho}{p}
  \right)^T,
  \quad
  s = \log \frac{p}{\rho^\gamma} = \log p - \gamma \log \rho.
\end{gather}
In order to handle the terms $\frac{1}{p}$, a new mean value has to be used.
Since
\begin{equation}
  \jump{\frac{1}{a}}
  =
  \frac{1}{a_+} - \frac{1}{a_-}
  =
  \frac{a_- - a_+}{a_+ a_-},
\end{equation}
the \emph{geometric mean}
\begin{equation}
\label{eq:geometric-mean}
  \geomean{a} := \sqrt{a_+ a_-},
\end{equation}
fulfils
\begin{equation}
\label{eq:geometric-mean-rule}
  \jump{\frac{1}{a}}
  =
  -\frac{1}{\geomean{a}^2} \jump{a}.
\end{equation}

One variant to write the jumps is
\begin{equation}
\begin{aligned}
  \jump{w_1}
  =&
  - \frac{1}{\gamma-1} \jump{s} - \frac{1}{2} \jump{ \frac{\rho v^2}{p} }
  =
  \frac{1}{\gamma-1} \jump{ - \log p + \gamma \log \rho}
  - \frac{1}{2} \jump{ \frac{\rho v^2}{p} }
  \\
  =&
  - \frac{1}{\gamma-1} \frac{1}{\logmean{p}} \jump{p}
  + \frac{\gamma}{\gamma-1} \frac{1}{\logmean{\rho}} \jump{\rho}
  - \mean{\frac{\rho}{p}} \mean{v_x} \jump{v_x}
  \\&
  - \mean{\frac{\rho}{p}} \mean{v_y} \jump{v_y}
  + \mean{\rho} \frac{ \mean{v_x^2} + \mean{v_y^2} }{2 \geomean{p}^2} \jump{p}
  - \frac{ \mean{v_x^2} + \mean{v_y^2} }{2} \mean{\inv p} \jump{\rho},
  \\
  \jump{w_2}
  =&
  \jump{ \frac{\rho v_x}{p} }
  =
  \mean{\frac{\rho}{p}} \jump{v_x}
  - \frac{ \mean{\rho} \mean{v_x} }{ \geomean{p}^2 } \jump{p}
  + \mean{\inv p} \mean{v_x} \jump{\rho},
  \\
  \jump{w_3}
  =&
  \jump{ \frac{\rho v_y}{p} }
  =
  \mean{\frac{\rho}{p}} \jump{v_y}
  - \frac{ \mean{\rho} \mean{v_y} }{ \geomean{p}^2 } \jump{p}
  + \mean{\inv p} \mean{v_y} \jump{\rho},
  \\
  \jump{w_4}
  =&
  - \jump{ \frac{\rho}{p} }
  =
    \frac{ \mean{\rho} }{ \geomean{p}^2 } \jump{p}
  - \mean{\inv p} \jump{\rho},
  \\
  \jump{\psi}
  =&
  \jump{\rho v}
  =
  \mean{\rho} \jump{v}
  + \mean{v} \jump{\rho}.
\end{aligned}
\end{equation}
Therefore, the entropy conservation conditions \eqref{eq:EC} can be written as
\begin{equation}
\begin{aligned}
  0
  =&
  \Bigg(
    \frac{\gamma}{\gamma-1} \frac{1}{\logmean{\rho}} \fnumx_{\rho}
    - \frac{ \mean{v_x^2} + \mean{v_y^2} }{2} \mean{\inv p} \fnumx_{\rho}
    + \mean{\inv p} \mean{v_x} \fnumx_{\rho v_x}
    \\&\qquad
    + \mean{\inv p} \mean{v_y} \fnumx_{\rho v_y}
    - \mean{\inv p} \fnumx_{\rho e}
    - \mean{v_x}
  \Bigg) \jump{\rho}
  \\&
  + \left(
    - \mean{\frac{\rho}{p}} \mean{v_x} \fnumx_{\rho}
    + \mean{\frac{\rho}{p}} \fnumx_{\rho v_x}
    - \mean{\rho}
  \right) \jump{v_x}
  + \left(
    - \mean{\frac{\rho}{p}} \mean{v_y} \fnumx_{\rho}
    + \mean{\frac{\rho}{p}} \fnumx_{\rho v_y}
  \right) \jump{v_y}
  \\&
  + \Bigg(
    - \frac{1}{\gamma-1} \frac{1}{\logmean{p}} \fnumx_{\rho}
    + \mean{\rho} \frac{ \mean{v_x^2} + \mean{v_y^2} }{2 \geomean{p}^2} \fnumx_{\rho}
    - \frac{ \mean{\rho} \mean{v_x} }{ \geomean{p}^2 } \fnumx_{\rho v_x}
    \\&\qquad
    - \frac{ \mean{\rho} \mean{v_y} }{ \geomean{p}^2 } \fnumx_{\rho v_y}
    + \frac{ \mean{\rho} }{ \geomean{p}^2 } \fnumx_{\rho e}
  \Bigg) \jump{p},
\end{aligned}
\end{equation}
\begin{equation}
\begin{aligned}
  0
  =&
  \Bigg(
    \frac{\gamma}{\gamma-1} \frac{1}{\logmean{\rho}} \fnumy_{\rho}
    - \frac{ \mean{v_x^2} + \mean{v_y^2} }{2} \mean{\inv p} \fnumy_{\rho}
    + \mean{\inv p} \mean{v_x} \fnumy_{\rho v_x}
    \\&\qquad
    + \mean{\inv p} \mean{v_y} \fnumy_{\rho v_y}
    - \mean{\inv p} \fnumy_{\rho e}
    - \mean{v_y}
  \Bigg) \jump{\rho}
  \\&
  + \left(
    - \mean{\frac{\rho}{p}} \mean{v_x} \fnumy_{\rho}
    + \mean{\frac{\rho}{p}} \fnumy_{\rho v_x}
  \right) \jump{v_x}
  + \left(
    - \mean{\frac{\rho}{p}} \mean{v_y} \fnumy_{\rho}
    + \mean{\frac{\rho}{p}} \fnumy_{\rho v_y}
    - \mean{\rho}
  \right) \jump{v_y}
  \\&
  + \Bigg(
    - \frac{1}{\gamma-1} \frac{1}{\logmean{p}} \fnumy_{\rho}
    + \mean{\rho} \frac{ \mean{v_x^2} + \mean{v_y^2} }{2 \geomean{p}^2} \fnumy_{\rho}
    - \frac{ \mean{\rho} \mean{v_x} }{ \geomean{p}^2 } \fnumy_{\rho v_x}
    \\&\qquad
    - \frac{ \mean{\rho} \mean{v_y} }{ \geomean{p}^2 } \fnumy_{\rho v_y}
    + \frac{ \mean{\rho} }{ \geomean{p}^2 } \fnumy_{\rho e}
  \Bigg) \jump{p}.
\end{aligned}
\end{equation}
Thus, the fluxes
\begin{equation}
\label{eq:rho-v-p-EC-KEP}
\begin{aligned}
  \fnumx&
  \begin{dcases}
  \fnumx_{\rho}
  =
  (\gamma-1) \left(
    \frac{\gamma}{\logmean{\rho}}
    - \frac{ \mean{\inv p} \geomean{p}^2 }{ \mean{\rho} \logmean{p} }
  \right)^{-1} \mean{v_x},
  \\
  \fnumx_{\rho v_x}
  =
  \mean{v_x} \fnumx_{\rho}
  + \frac{\mean{\rho}}{\mean{\rho / p}},
  \\
  \fnumx_{\rho v_y}
  =
  \mean{v_y} \fnumx_{\rho},
  \\
  \fnumx_{\rho e}
  =
  \left(
    \frac{1}{\gamma-1} \frac{\geomean{p}^2}{\mean{\rho} \logmean{p}}
    + \mean{v_x}^2 + \mean{v_y}^2
    - \frac{ \mean{v_x^2} + \mean{v_y^2} }{2}
  \right) \fnumx_{\rho}
  + \frac{ \mean{\rho} \mean{v} }{ \mean{\rho / p} },
  \end{dcases}
  \\
  \fnumy&
  \begin{dcases}
  \fnumy_{\rho}
  =
  (\gamma-1) \left(
    \frac{\gamma}{\logmean{\rho}}
    - \frac{ \mean{\inv p} \geomean{p}^2 }{ \mean{\rho} \logmean{p} }
  \right)^{-1} \mean{v_y},
  \\
  \fnumy_{\rho v_x}
  =
  \mean{v_x} \fnumy_{\rho},
  \\
  \fnumy_{\rho v_y}
  =
  \mean{v_y} \fnumy_{\rho}
  + \frac{\mean{\rho}}{\mean{\rho / p}},
  \\
  \fnumy_{\rho e}
  =
  \left(
    \frac{1}{\gamma-1} \frac{\geomean{p}^2}{\mean{\rho} \logmean{p}}
    + \mean{v_x}^2 + \mean{v_y}^2
    - \frac{ \mean{v_x^2} + \mean{v_y^2} }{2}
  \right) \fnumy_{\rho}
  + \frac{ \mean{\rho} \mean{v} }{ \mean{\rho / p} },
  \end{dcases}
\end{aligned}
\end{equation}
can be seen to be entropy conservative and consistent.
Additionally, they are kinetic energy preserving with numerical pressure flux
$\pnum = \frac{\mean{\rho}}{\mean{\rho / p}}$.

As before, the pressure influences the numerical density flux, leading to some
problems as explained by \citet{derigs2016novelAveraging}, see also
Remark~\ref{rem:positivity-density} and the numerical tests in section
\ref{sec:numerical-tests}.

\subsection{Using \texorpdfstring{$\rho, v, T$}{ρ,v,T} as Variables}
\label{sec:rho-v-T}

Using the variables $\rho, v$, and $RT = \frac{p}{\rho}$, the flux potentials
and the entropy variables \eqref{eq:w} can be written as
\begin{gather}
\label{eq:psi-w-using-rho-v-T}
  \psi_x = \rho v_x,
  \quad
  \psi_y = \rho v_y,
  \\
  w
  =
  \left(
    \frac{\gamma}{\gamma-1} - \frac{s}{\gamma-1} - \frac{v^2}{2 RT},
    \frac{v_x}{RT},
    \frac{v_y}{RT},
    -\frac{1}{RT}
  \right)^T,
  \quad
  s = \log \frac{p}{\rho^\gamma} = \log RT - (\gamma-1) \log \rho.
\end{gather}

\subsubsection{Variant 1}

One way to write the jumps is
\begin{equation}
\begin{aligned}
  \jump{w_1}
  =&
  - \frac{1}{\gamma-1} \jump{s} - \frac{1}{2} \jump{ \frac{v^2}{RT} }
  =
  - \frac{1}{\gamma-1} \jump{\log RT}
  + \jump{\log \rho}
  - \frac{1}{2} \jump{ \frac{v^2}{RT} }
  \\
  =&
  - \frac{1}{\gamma-1} \frac{1}{\logmean{RT}} \jump{RT}
  + \frac{1}{\logmean{\rho}} \jump{\rho}
  - \mean{\frac{1}{RT}} \mean{v_x} \jump{v_x}
  - \mean{\frac{1}{RT}} \mean{v_y} \jump{v_y}
  \\&
  + \frac{ \mean{v_x^2} + \mean{v_y^2} }{ 2 \geomean{RT}^2 } \jump{RT},
  \\
  \jump{w_2}
  =&
  \jump{ \frac{v_x}{RT} }
  =
  \mean{\frac{1}{RT}} \jump{v_x}
  - \frac{ \mean{v_x} }{ \geomean{RT}^2 } \jump{RT},
  \\
  \jump{w_3}
  =&
  \jump{ \frac{v_y}{RT} }
  =
  \mean{\frac{1}{RT}} \jump{v_y}
  - \frac{ \mean{v_y} }{ \geomean{RT}^2 } \jump{RT},
  \\
  \jump{w_4}
  =&
  - \jump{ \frac{1}{RT} }
  =
  \frac{1}{ \geomean{RT}^2 } \jump{RT},
  \\
  \jump{\psi_x}
  =&
  \jump{\rho v_x}
  =
  \mean{\rho} \jump{v_x}
  + \mean{v_x} \jump{\rho},
  \\
  \jump{\psi_y}
  =&
  \jump{\rho v_y}
  =
  \mean{\rho} \jump{v_y}
  + \mean{v_y} \jump{\rho}.
\end{aligned}
\end{equation}
Therefore, the entropy conservation conditions \eqref{eq:EC} can be written as
\begin{equation}
\begin{aligned}
  0
  =&
  \left(
    \frac{1}{\logmean{\rho}} \fnumx_{\rho}
    - \mean{v_x}
  \right) \jump{\rho}
  + \left(
    - \mean{\frac{1}{RT}} \mean{v_x} \fnumx_{\rho}
    + \mean{\frac{1}{RT}} \fnumx_{\rho v_x}
    - \mean{\rho}
  \right) \jump{v_x}
  \\&
  + \left(
    - \mean{\frac{1}{RT}} \mean{v_y} \fnumx_{\rho}
    + \mean{\frac{1}{RT}} \fnumx_{\rho v_y}
  \right) \jump{v_y}
  \\&
  + \Bigg(
    - \frac{1}{\gamma-1} \frac{1}{\logmean{RT}} \fnumx_{\rho}
    + \frac{ \mean{v_x^2} + \mean{v_y^2} }{ 2 \geomean{RT}^2 } \fnumx_{\rho}
    - \frac{ \mean{v_x} }{ \geomean{RT}^2 } \fnumx_{\rho v_x}
    \\&\qquad
    - \frac{ \mean{v_y} }{ \geomean{RT}^2 } \fnumx_{\rho v_y}
    + \frac{1}{ \geomean{RT}^2 } \fnumx_{\rho e}
  \Bigg) \jump{RT},
  \\
  0
  =&
  \left(
    \frac{1}{\logmean{\rho}} \fnumy_{\rho}
    - \mean{v_y}
  \right) \jump{\rho}
  + \left(
    - \mean{\frac{1}{RT}} \mean{v_x} \fnumy_{\rho}
    + \mean{\frac{1}{RT}} \fnumy_{\rho v_x}
  \right) \jump{v_x}
  \\&
  + \left(
    - \mean{\frac{1}{RT}} \mean{v_y} \fnumy_{\rho}
    + \mean{\frac{1}{RT}} \fnumy_{\rho v_y}
    - \mean{\rho}
  \right) \jump{v_y}
  \\&
  + \Bigg(
    - \frac{1}{\gamma-1} \frac{1}{\logmean{RT}} \fnumy_{\rho}
    + \frac{ \mean{v_x^2} + \mean{v_y^2} }{ 2 \geomean{RT}^2 } \fnumy_{\rho}
    - \frac{ \mean{v_x} }{ \geomean{RT}^2 } \fnumy_{\rho v_x}
    \\&\qquad
    - \frac{ \mean{v_y} }{ \geomean{RT}^2 } \fnumy_{\rho v_y}
    + \frac{1}{ \geomean{RT}^2 } \fnumy_{\rho e}
  \Bigg) \jump{RT}.
\end{aligned}
\end{equation}
Thus, the fluxes
\begin{equation}
\label{eq:rho-v-T-EC-KEP}
\begin{aligned}
  \fnumx&
  \begin{dcases}
  \fnumx_{\rho}
  =
  \logmean{\rho} \mean{v_x},
  \\
  \fnumx_{\rho v_x}
  =
  \mean{v_x} \fnumx_{\rho}
  + \frac{ \mean{\rho} }{ \mean{1/RT} },
  \\
  \fnumx_{\rho v_y}
  =
  \mean{v_y} \fnumx_{\rho},
  \\
  \fnumx_{\rho e}
  =
  \left(
    \frac{1}{\gamma-1} \frac{ \geomean{RT}^2 }{ \logmean{RT} }
    - \frac{ \mean{v_x^2} + \mean{v_y^2} }{2}
  \right) \fnumx_{\rho}
  + \mean{v_x} \fnumx_{\rho v_x}
  + \mean{v_y} \fnumx_{\rho v_y},
  \end{dcases}
  \\
  \fnumy&
  \begin{dcases}
  \fnumy_{\rho}
  =
  \logmean{\rho} \mean{v_y},
  \\
  \fnumy_{\rho v_x}
  =
  \mean{v_x} \fnumy_{\rho},
  \\
  \fnumy_{\rho v_y}
  =
  \mean{v_y} \fnumy_{\rho}
  + \frac{ \mean{\rho} }{ \mean{1/RT} },
  \\
  \fnumy_{\rho e}
  =
  \left(
    \frac{1}{\gamma-1} \frac{ \geomean{RT}^2 }{ \logmean{RT} }
    - \frac{ \mean{v_x^2} + \mean{v_y^2} }{2}
  \right) \fnumy_{\rho}
  + \mean{v_x} \fnumy_{\rho v_x}
  + \mean{v_y} \fnumy_{\rho v_y},
  \end{dcases}
\end{aligned}
\end{equation}
can be seen to be entropy conservative and consistent.
Additionally, they are kinetic energy preserving with numerical pressure flux
$\pnum = \frac{ \mean{\rho} }{ \mean{1/RT} }$, i.e. the same as for the entropy
conservative and kinetic energy preserving flux
\eqref{eq:Chandrashekar-rho-v-beta-EC-KEP} proposed by
\citet{chandrashekar2013kinetic}. Moreover, the density flux $\fnum_{\rho}$ is
the same. However, the energy fluxes $\fnum_{\rho e}$ are different.

\subsubsection{Variant 2}

Similarly to the derivation of \eqref{eq:Chandrashekar-rho-v-beta-EC-notKEP},
choosing another possibility to split the jump 
\begin{equation}
\begin{aligned}
  \jump{ \frac{v_{x/y}^2}{RT} }
  =&
  \mean{\frac{v_{x/y}}{RT}} \jump{v_{x/y}} + \mean{v_{x/y}} \jump{\frac{v_{x/y}}{RT}}
  \\=& 
  \mean{\frac{v_{x/y}}{RT}} \jump{v_{x/y}}
  + \mean{\frac{1}{RT}} \mean{v_{x/y}} \jump{v_{x/y}}
  - \frac{ \mean{v_{x/y}}^2 }{ \geomean{RT}^2 } \jump{RT},
\end{aligned}
\end{equation}
the entropy conservation conditions \eqref{eq:EC} can be written as
\begin{equation}
\begin{aligned}
  0
  =&
  \left(
    \frac{1}{\logmean{\rho}} \fnumx_{\rho}
    - \mean{v_x}
  \right) \jump{\rho}
  \\&
  + \left(
    - \frac{1}{2} \mean{\frac{v_x}{RT}} \fnumx_{\rho}
    - \frac{1}{2} \mean{\frac{1}{RT}} \mean{v_x} \fnumx_{\rho}
    + \mean{\frac{1}{RT}} \fnumx_{\rho v_x}
    - \mean{\rho}
  \right) \jump{v_x}
  \\&
  + \left(
    - \frac{1}{2} \mean{\frac{v_y}{RT}} \fnumx_{\rho}
    - \frac{1}{2} \mean{\frac{1}{RT}} \mean{v_y} \fnumx_{\rho}
    + \mean{\frac{1}{RT}} \fnumx_{\rho v_y}
  \right) \jump{v_y}
  \\&
  + \Bigg(
    - \frac{1}{\gamma-1} \frac{1}{\logmean{RT}} \fnumx_{\rho}
    + \frac{ \mean{v_x}^2 + \mean{v_y^2} }{ 2 \geomean{RT}^2 } \fnumx_{\rho}
    - \frac{ \mean{v_x} }{ \geomean{RT}^2 } \fnumx_{\rho v_x}
    \\&\qquad
    - \frac{ \mean{v_y} }{ \geomean{RT}^2 } \fnumx_{\rho v_y}
    + \frac{1}{ \geomean{RT}^2 } \fnumx_{\rho e}
  \Bigg) \jump{RT},
\end{aligned}
\end{equation}
\begin{equation}
\begin{aligned}
  0
  =&
  \left(
    \frac{1}{\logmean{\rho}} \fnumy_{\rho}
    - \mean{v_y}
  \right) \jump{\rho}
  \\&
  + \left(
    - \frac{1}{2} \mean{\frac{v_x}{RT}} \fnumy_{\rho}
    - \frac{1}{2} \mean{\frac{1}{RT}} \mean{v_x} \fnumy_{\rho}
    + \mean{\frac{1}{RT}} \fnumy_{\rho v_x}
  \right) \jump{v_x}
  \\&
  + \left(
    - \frac{1}{2} \mean{\frac{v_y}{RT}} \fnumy_{\rho}
    - \frac{1}{2} \mean{\frac{1}{RT}} \mean{v_y} \fnumy_{\rho}
    + \mean{\frac{1}{RT}} \fnumy_{\rho v_y}
    - \mean{\rho}
  \right) \jump{v_y}
  \\&
  + \Bigg(
    - \frac{1}{\gamma-1} \frac{1}{\logmean{RT}} \fnumy_{\rho}
    + \frac{ \mean{v_x}^2 + \mean{v_y^2} }{ 2 \geomean{RT}^2 } \fnumy_{\rho}
    - \frac{ \mean{v_x} }{ \geomean{RT}^2 } \fnumy_{\rho v_x}
    \\&\qquad
    - \frac{ \mean{v_y} }{ \geomean{RT}^2 } \fnumy_{\rho v_y}
    + \frac{1}{ \geomean{RT}^2 } \fnumy_{\rho e}
  \Bigg) \jump{RT}.
\end{aligned}
\end{equation}
Thus, the fluxes
\begin{equation}
\label{eq:rho-v-T-EC-notKEP}
\begin{aligned}
  \fnumx&
  \begin{dcases}
  \fnumx_{\rho}
  =
  \logmean{\rho} \mean{v_x},
  \\
  \fnumx_{\rho v_x}
  =
  \frac{ \mean{v_x/RT} + \mean{1/RT} \mean{v_x} }{ 2 \mean{1/RT} } \fnumx_{\rho}
  + \frac{\mean{\rho}}{\mean{1/RT}},
  \\
  \fnumx_{\rho v_y}
  =
  \frac{ \mean{v_y/RT} + \mean{1/RT} \mean{v_y} }{ 2 \mean{1/RT} } \fnumx_{\rho},
  \\
  \fnumx_{\rho e}
  =
  \left(
    \frac{1}{\gamma-1} \frac{ \geomean{RT}^2 }{ \logmean{RT} }
    - \frac{ \mean{v_x}^2 + \mean{v_y}^2 }{2}
  \right) \fnumx_{\rho}
  + \mean{v_x} \fnumx_{\rho v_x}
  + \mean{v_y} \fnumx_{\rho v_y},
  \end{dcases}
  \\
  \fnumy&
  \begin{dcases}
  \fnumy_{\rho}
  =
  \logmean{\rho} \mean{v_y},
  \\
  \fnumy_{\rho v_x}
  =
  \frac{ \mean{v_x/RT} + \mean{1/RT} \mean{v_x} }{ 2 \mean{1/RT} } \fnumy_{\rho},
  \\
  \fnumy_{\rho v_y}
  =
  \frac{ \mean{v_y/RT} + \mean{1/RT} \mean{v_y} }{ 2 \mean{1/RT} } \fnumy_{\rho}
  + \frac{\mean{\rho}}{\mean{1/RT}},
  \\
  \fnumy_{\rho e}
  =
  \left(
    \frac{1}{\gamma-1} \frac{ \geomean{RT}^2 }{ \logmean{RT} }
    - \frac{ \mean{v_x}^2 + \mean{v_y}^2 }{2}
  \right) \fnumy_{\rho}
  + \mean{v_x} \fnumy_{\rho v_x}
  + \mean{v_y} \fnumy_{\rho v_y},
  \end{dcases}
\end{aligned}
\end{equation}
can be seen to be entropy conservative. Again, since the property ``kinetic energy
preserving'' is not well-defined, they could possibly be considered as kinetic
energy preserving, cf. Remark~\ref{rem:fnum-KEP}.

\subsection{Using \texorpdfstring{$\rho, v, \inv g(\rho/p)$}{ρ,v,g\^-1(ρ/p)} as Variables}
\label{sec:rho-v-invg-rho/p}

As can be seen in the previous subsections, there are many entropy conservative
and kinetic energy preserving numerical fluxes in the sense described at
the beginning of section \ref{sec:fluxes}, obtained using the general
procedure described there. However, they are different and will thus have
advantages or disadvantages compared to each other.
Looking at the entropy variables \eqref{eq:w}
\begin{equation}
  w
  =
  \left(
    \frac{\gamma}{\gamma-1} - \frac{s}{\gamma-1} - \frac{\rho v^2}{2 p},
    \frac{\rho v_x}{p},
    \frac{\rho v_y}{p},
    -\frac{\rho}{p}
  \right)^T,
  \quad
  s = \log \frac{p}{\rho^\gamma},
\end{equation}
it can be seen that the term $\frac{\rho}{p}$ has a crucial role. If the variables
$\rho, v, \chi$ are chosen (where $\chi$ is some third variable), and the
expression $\frac{\rho}{p} = g(\chi)$ depends only on this third variable $\chi$,
a kinetic energy preserving flux can be constructed using a density flux depending
only on $\rho, v$. Indeed, writing the jumps as
\begin{equation}
\begin{aligned}
  \jump{w_1}
  =&
  - \frac{1}{\gamma-1} \jump{s} - \frac{1}{2} \jump{ \frac{\rho}{p} v^2 }
  =
  \jump{ \log \rho }
  + \frac{1}{\gamma-1} \jump{ \log \frac{\rho}{p} }
  - \frac{1}{2} \jump{ \frac{\rho}{p} v^2 }
  \\
  =&
  \frac{1}{\logmean{\rho}} \jump{\rho}
  + \frac{1}{\gamma-1} \jump{ \log \frac{\rho}{p} }
  - \frac{ \mean{v_x^2} + \mean{v_y^2} }{2} \jump{\frac{\rho}{p}}
  - \mean{\frac{\rho}{p}} \mean{v_x} \jump{v_x}
  - \mean{\frac{\rho}{p}} \mean{v_y} \jump{v_y},
  \\
  \jump{w_2}
  =&
  \jump{\frac{\rho}{p} v_x}
  =
  \mean{\frac{\rho}{p}} \jump{v_x}
  + \mean{v_x} \jump{\frac{\rho}{p}},
  \\
  \jump{w_3}
  =&
  \jump{\frac{\rho}{p} v_y}
  =
  \mean{\frac{\rho}{p}} \jump{v_y}
  + \mean{v_y} \jump{\frac{\rho}{p}},
  \\
  \jump{w_4}
  =&
  - \jump{\frac{\rho}{p}},
  \\
  \jump{\psi_x}
  =&
  \jump{\rho v_x}
  =
  \mean{\rho} \jump{v_x}
  + \mean{v_x} \jump{\rho},
  \\
  \jump{\psi_y}
  =&
  \jump{\rho v_y}
  =
  \mean{\rho} \jump{v_y}
  + \mean{v_y} \jump{\rho},
\end{aligned}
\end{equation}
the entropy conservation conditions $\jump{w} \cdot f^{\mathrm{num},{x/y}}
- \jump{\psi_{x/y}} = 0$ \eqref{eq:EC} become
\begin{equation}
\begin{aligned}
  0
  =&
  \left(
    \frac{1}{\logmean{\rho}} \fnumx_{\rho}
    - \mean{v_x}
  \right) \jump{\rho}
  + \left(
    - \mean{\frac{\rho}{p}} \mean{v_x} \fnumx_{\rho}
    + \mean{\frac{\rho}{p}} \fnumx_{\rho v_x}
    - \mean{\rho}
  \right) \jump{v_x}
  \\&
  + \left(
    - \mean{\frac{\rho}{p}} \mean{v_y} \fnumx_{\rho}
    + \mean{\frac{\rho}{p}} \fnumx_{\rho v_y}
  \right) \jump{v_y}
  \\&
  + \Bigg(
    \frac{1}{\gamma-1} \jump{ \log \frac{\rho}{p} } \fnumx_{\rho}
    - \frac{ \mean{v_x^2} + \mean{v_y^2} }{2} \jump{\frac{\rho}{p}} \fnumx_{\rho}
    + \mean{v_x} \jump{\frac{\rho}{p}} \fnumx_{\rho v_x}
    \\&\qquad
    + \mean{v_y} \jump{\frac{\rho}{p}} \fnumx_{\rho v_y}
    - \jump{\frac{\rho}{p}} \fnumx_{\rho e}
  \Bigg),
  \\
  0
  =&
  \left(
    \frac{1}{\logmean{\rho}} \fnumy_{\rho}
    - \mean{v_y}
  \right) \jump{\rho}
  + \left(
    - \mean{\frac{\rho}{p}} \mean{v_x} \fnumy_{\rho}
    + \mean{\frac{\rho}{p}} \fnumy_{\rho v_x}
  \right) \jump{v_x}
  \\&
  + \left(
    - \mean{\frac{\rho}{p}} \mean{v_y} \fnumy_{\rho}
    + \mean{\frac{\rho}{p}} \fnumy_{\rho v_y}
    - \mean{\rho}
  \right) \jump{v_y}
  \\&
  + \Bigg(
    \frac{1}{\gamma-1} \jump{ \log \frac{\rho}{p} } \fnumy_{\rho}
    - \frac{ \mean{v_x^2} + \mean{v_y^2} }{2} \jump{\frac{\rho}{p}} \fnumy_{\rho}
    + \mean{v_x} \jump{\frac{\rho}{p}} \fnumy_{\rho v_x}
    \\&\qquad
    + \mean{v_y} \jump{\frac{\rho}{p}} \fnumy_{\rho v_y}
    - \jump{\frac{\rho}{p}} \fnumy_{\rho e}
  \Bigg).
\end{aligned}
\end{equation}
Thus, the density and momentum fluxes
\begin{equation}
\begin{aligned}
  \fnumx&
  \begin{dcases}
  \fnumx_{\rho}
  =
  \logmean{\rho} \mean{v_x},
  \\
  \fnumx_{\rho v_x}
  =
  \mean{v_x} \fnumx_{\rho}
  + \frac{\mean{\rho}}{\mean{\rho/p}},
  \\
  \fnumx_{\rho v_y}
  =
  \mean{v_y} \fnumx_{\rho},
  \end{dcases}
  \\
  \fnumy&
  \begin{dcases}
  \fnumy_{\rho}
  =
  \logmean{\rho} \mean{v_y},
  \\
  \fnumy_{\rho v_y}
  =
  \mean{v_y} \fnumx_{\rho},
  \\
  \fnumy_{\rho v_y}
  =
  \mean{v_y} \fnumy_{\rho}
  + \frac{\mean{\rho}}{\mean{\rho/p}},
  \end{dcases}
\end{aligned}
\end{equation}
set the two terms to zero. These fluxes are the same as in
\eqref{eq:Chandrashekar-rho-v-beta-EC-KEP} and \eqref{eq:rho-v-T-EC-KEP}, i.e.
the same as the ones used by \citet{chandrashekar2013kinetic}.
However, depending on the expression of $\jump{ \log \frac{\rho}{p} }$, different
energy fluxes can be constructed, resulting in entropy conservative and
kinetic energy preserving schemes.

Choosing $\chi = \beta \propto \frac{\rho}{p}$, \citet{chandrashekar2013kinetic}
set $\jump{ \log \frac{\rho}{p} } = \frac{1}{\logmean{\chi}} \jump{\chi}$ and
derived his EC and KEP flux \eqref{eq:Chandrashekar-rho-v-beta-EC-KEP}.
Choosing $\chi = RT = \left( \frac{\rho}{p} \right)^{-1}$ and setting
$\jump{ \log \frac{\rho}{p} } = - \frac{1}{\logmean{\chi}} \jump{\chi}$,
the flux \eqref{eq:rho-v-T-EC-KEP} has been derived in section \ref{sec:rho-v-T}.

\subsubsection{Variant 1}
\label{sec:rho-v-invg-rho/p-variant1}

More generally, choosing $r \in \R \setminus \set{0}$, and setting
$\frac{\rho}{p} = \chi^r$, the jumps become
\begin{equation}
  \jump{ \log \frac{\rho}{p} }
  =
  \jump{ \log \chi^r }
  = r \jump{ \log \chi }
  = \frac{r}{\logmean{\chi}} \jump{\chi},
  \quad
  \jump{\frac{\rho}{p}}
  =
  \jump{\chi^r}
  =
  \frac{\chi_+^r - \chi_-^r}{\chi_+ - \chi_-} \jump{\chi},
\end{equation}
where the mean value
\begin{equation}
\label{eq:r-mean}
  \mean{\chi}_{x \mapsto x^r}
  :=
  \left( \frac{1}{r} \frac{\chi_+^r - \chi_-^r}{\chi_+ - \chi_-} \right)^{1/(r-1)}
\end{equation}
can be introduced to yield
\begin{equation}
\label{eq:r-mean-rule}
  \jump{\chi^r}
  =
  r \mean{\chi}_{x \mapsto x^r}^{r-1} \jump{\chi}.
\end{equation}
Thus, the arithmetic mean \eqref{eq:arithmetic-mean} becomes
$\mean{a} = \mean{a}_{x \mapsto x^2}$ and the geometric mean \eqref{eq:geometric-mean}
becomes $\geomean{a} = \mean{a}_{x \mapsto 1/x}$.

Using this mean value, entropy conservative, kinetic energy preserving, and
consistent numerical fluxes are
\begin{equation}
\label{eq:rho-v-rho/p^(1/r)-EC-KEP}
\begin{aligned}
  \fnumx&
  \begin{dcases}
  \fnumx_{\rho}
  =
  \logmean{\rho} \mean{v_x},
  \\
  \fnumx_{\rho v_x}
  =
  \mean{v_x} \fnumx_{\rho}
  + \frac{\mean{\rho}}{\mean{\rho/p}},
  \\
  \fnumx_{\rho v_y}
  =
  \mean{v_y} \fnumx_{\rho},
  \\
  \fnumx_{\rho e}
  =
  \left(
    \frac{1}{\gamma-1} \frac{1}{ \mean{\chi}_{x \mapsto x^r}^{r-1} \logmean{\chi} }
    - \frac{ \mean{v_x^2} + \mean{v_y^2} }{2}
  \right) \fnumx_{\rho}
  + \mean{v_x} \fnumx_{\rho v_x}
  + \mean{v_y} \fnumx_{\rho v_y},
  \end{dcases}
  \\
  \fnumy&
  \begin{dcases}
  \fnumy_{\rho}
  =
  \logmean{\rho} \mean{v_y},
  \\
  \fnumy_{\rho v_y}
  =
  \mean{v_y} \fnumx_{\rho},
  \\
  \fnumy_{\rho v_y}
  =
  \mean{v_y} \fnumy_{\rho}
  + \frac{\mean{\rho}}{\mean{\rho/p}},
  \\
  \fnumy_{\rho e}
  =
  \left(
    \frac{1}{\gamma-1} \frac{1}{ \mean{\chi}_{x \mapsto x^r}^{r-1} \logmean{\chi} }
    - \frac{ \mean{v_x^2} + \mean{v_y^2} }{2}
  \right) \fnumy_{\rho}
  + \mean{v_x} \fnumy_{\rho v_x}
  + \mean{v_y} \fnumy_{\rho v_y}.
  \end{dcases}
\end{aligned}
\end{equation}
Of course, some numerically stable procedure to compute
$\mean{\chi}_{x \mapsto x^r}^{r-1}
= \frac{1}{r} \frac{\chi_+^r - \chi_-^r}{\chi_+ - \chi_-}$
has to be derived.

\subsubsection{Variant 2}
\label{sec:rho-v-invg-rho/p-variant2}

The choice $\frac{\rho}{p} = \exp \chi$ results in
\begin{equation}
  \jump{ \log \frac{\rho}{p} }
  =
  \jump{ \log \exp \chi }
  = \jump{ \chi },
  \quad
  \jump{\frac{\rho}{p}}
  =
  \jump{ \exp \chi}
  =
  \frac{\exp \chi_+ - \exp \chi_-}{\chi_+ - \chi_-} \jump{\chi},
\end{equation}
where the mean value
\begin{equation}
\label{eq:exp-mean}
  \mean{\chi}_{x \mapsto \exp x}
  :=
  \log \frac{\exp \chi_+ - \exp \chi_-}{\chi_+ - \chi_-}
\end{equation}
can be introduced to yield
\begin{equation}
\label{eq:exp-mean-rule}
  \jump{\exp \chi}
  =
  \exp \left(\! \mean{\chi}_{x \mapsto \exp x} \right) \jump{\chi}.
\end{equation}

Using this mean value, an entropy conservative, kinetic energy preserving, and
consistent numerical flux is
\begin{equation}
\label{eq:rho-v-log-rho/p-EC-KEP}
\begin{aligned}
  \fnumx&
  \begin{dcases}
  \fnumx_{\rho}
  =
  \logmean{\rho} \mean{v_x},
  \\
  \fnumx_{\rho v_x}
  =
  \mean{v_x} \fnumx_{\rho}
  + \frac{\mean{\rho}}{\mean{\rho/p}},
  \\
  \fnumx_{\rho v_y}
  =
  \mean{v_y} \fnumx_{\rho},
  \\
  \fnumx_{\rho e}
  =
  \left(
    \frac{1}{\gamma-1} \frac{1}{ \exp\! \mean{\chi}_{x \mapsto \exp x} }
    - \frac{ \mean{v_x^2} + \mean{v_y^2} }{2}
  \right) \fnumx_{\rho}
  + \mean{v_x} \fnumx_{\rho v_x}
  + \mean{v_y} \fnumx_{\rho v_y},
  \end{dcases}
  \\
  \fnumy&
  \begin{dcases}
  \fnumy_{\rho}
  =
  \logmean{\rho} \mean{v_y},
  \\
  \fnumy_{\rho v_y}
  =
  \mean{v_y} \fnumx_{\rho},
  \\
  \fnumy_{\rho v_y}
  =
  \mean{v_y} \fnumy_{\rho}
  + \frac{\mean{\rho}}{\mean{\rho/p}},
  \\
  \fnumy_{\rho e}
  =
  \left(
    \frac{1}{\gamma-1} \frac{1}{ \exp\! \mean{\chi}_{x \mapsto \exp x} }
    - \frac{ \mean{v_x^2} + \mean{v_y^2} }{2}
  \right) \fnumy_{\rho}
  + \mean{v_x} \fnumy_{\rho v_x}
  + \mean{v_y} \fnumy_{\rho v_y}.
  \end{dcases}
\end{aligned}
\end{equation}
Again, some numerically stable procedure to compute
$\exp\! \mean{\chi}_{x \mapsto \exp x}
= \frac{\exp \chi_+ - \exp \chi_-}{\chi_+ - \chi_-}$
has to be derived.

\subsection{Using Other Variables}

Of course, some other sets of variables can be used to derive entropy conservative
numerical fluxes similar to the previous sections. However, since there is no
clear intuition which choice of variables might be ``good'', this is not 
carried out in detail here.
As noted by \citet{derigs2016novelAveraging}, an influence of the pressure in
the numerical density flux should be avoided, see also Remark~\ref{rem:positivity-density}
and the numerical tests in section \ref{sec:numerical-tests}.

\section{Reversing the Role of Energy and Entropy}
\label{sec:reversed-fluxes}

As proposed in \cite{coquel2001some} and used in \cite[Section 2.4.6]{bouchut2004nonlinear}
to derive an approximate Riemann solver based on the Suliciu relaxation approach,
the role of energy and entropy for the Euler equations can be reversed, i.e.
a conservation law for the entropy and an inequality for the energy can be considered,
cf. \cite{derigs2016novelSolver}. Then, the system reads
\begin{equation}
\label{eq:Euler-reversed}
\begin{aligned}
  \partial_t
  \underbrace{
  \begin{pmatrix}
    \rho
    \\
    \rho v_x
    \\
    \rho v_y
    \\
    \rho s
  \end{pmatrix}
  }_{= u}
  + \,\partial_x
  \underbrace{
  \begin{pmatrix}
    \rho v_x
    \\
    \rho v_x^2 + p
    \\
    \rho v_x v_y
    \\
    \rho s v_x
  \end{pmatrix}
  }_{= f_x(u)}
  + \,\partial_y
  \underbrace{
  \begin{pmatrix}
    \rho v_y
    \\
    \rho v_x v_y
    \\
    \rho v_y^2 + p
    \\
    \rho s v_y
  \end{pmatrix}
  }_{= f_y(u)}
  =
  0,
\end{aligned}
\end{equation}
and the 'entropy' condition becomes
\begin{equation}
\label{eq:Euler-reversed-entropy}
  \partial_t \underbrace{\left( \rho e \right)}_{=U}
  + \partial_x \underbrace{\left( (\rho e + p) v_x \right)}_{=F_x}
  + \partial_y \underbrace{\left( (\rho e + p) v_y \right)}_{=F_y}
  \leq 0.
\end{equation}
Since smooth solutions satisfy \eqref{eq:Euler-reversed-entropy} with equality,
they are also smooth solutions of the Euler equations \eqref{eq:Euler} with
equality in the usual entropy condition \eqref{eq:Euler-entropy}.
In the same spirit, an 'entropy' conserving numerical flux for \eqref{eq:Euler-reversed},
\eqref{eq:Euler-reversed-entropy} is an entropy conserving flux for \eqref{eq:Euler},
\eqref{eq:Euler-entropy} and vice versa.

\begin{remark}
\label{rem:reversing-entropy-energy}
  This reversion of the energy and the entropy is very specific to the Euler
  equations. Exchanging some conserved quantity with the entropy will in general
  not result in a convex ``entropy'' with such nice properties as used here.
\end{remark}

In order to express the energy $\rho e = \frac{1}{2} \rho v^2 + \rho \epsilon$
as a function of the new conserved variables $\rho, \rho v, \rho s$,
the pressure $p = (\gamma-1) \rho \epsilon$ can be inserted into the
specific entropy $s = \log \frac{p}{\rho^\gamma}$ to yield
\begin{equation}
  \rho \epsilon = \frac{\rho^\gamma}{\gamma-1} \exp s.
\end{equation}
Thus, the energy can be written as
\begin{equation}
  \rho e
  =
  \frac{1}{2} \frac{ (\rho v)^2 }{ \rho }
  + \frac{1}{\gamma-1} \rho^\gamma \exp\left( \frac{\rho s}{\rho} \right).
\end{equation}
Therefore, the new 'entropy' variables are
\begin{equation}
\label{eq:Euler-reversed-w}
\begin{aligned}
  w
  =
  \frac{\partial(\rho e)}{\partial(\rho, \rho v_x, \rho v_y, \rho s)}
  =&
  \begin{pmatrix}
    - \frac{1}{2} \frac{ (\rho v)^2 }{ \rho^2 }
    + \frac{\gamma}{\gamma-1} \rho^{\gamma-1} \exp\left( \frac{\rho s}{\rho} \right)
    - \frac{1}{\gamma-1} \rho^\gamma \frac{\rho s}{\rho^2} \exp\left( \frac{\rho s}{\rho} \right)
    \\
    \frac{ \rho v_x }{ \rho }
    \\
    \frac{ \rho v_y }{ \rho }
    \\
    \frac{1}{\gamma-1} \rho^{\gamma-1} \exp\left( \frac{\rho s}{\rho} \right)
  \end{pmatrix}
  \\
  =&
  \begin{pmatrix}
    - \frac{1}{2} v^2
    + \frac{\gamma}{\gamma-1} \rho^{\gamma-1} \exp(s)
    - \frac{1}{\gamma-1} \rho^{\gamma-1} s \exp(s)
    \\
    v_x
    \\
    v_y
    \\
    \frac{1}{\gamma-1} \rho^{\gamma-1} \exp(s)
  \end{pmatrix}.
\end{aligned}
\end{equation}
The new 'entropy' fluxes are
\begin{equation}
\begin{aligned}
  F_x
  =&
  (\rho e + p) v_x
  =
  \frac{1}{2} \rho v^2 v_x + \gamma \rho \epsilon v_x
  =
  \frac{1}{2} \rho v^2 v_x + \frac{\gamma}{\gamma-1} \rho^\gamma v_x \exp(s),
  \\
  F_y
  =&
  (\rho e + p) v_y
  =
  \frac{1}{2} \rho v^2 v_y + \gamma \rho \epsilon v_y
  =
  \frac{1}{2} \rho v^2 v_y + \frac{\gamma}{\gamma-1} \rho^\gamma v_y \exp(s),
\end{aligned}
\end{equation}
so that the new flux potentials $\psi_{x/y}$ fulfilling
$\partial_w \psi_{x/y} = f_{x/y}\left( u(w) \right)$ become
\begin{equation}
\label{eq:Euler-reversed-psi}
\begin{aligned}
  \psi_{x/y}
  =
  w \cdot f_{x/y} - F_{x/y}
  =&
  \left( - \frac{1}{2} v^2 + \frac{\gamma}{\gamma-1} \rho^{\gamma-1} \exp(s)
         - \frac{1}{\gamma-1} \rho^{\gamma-1} s \exp(s) \right) \rho v_{x/y}
  \\&
  + v_{x/y} \left( \rho v_{x/y}^2 + \rho^\gamma \exp(s) \right)
  + v_{y/x} \rho v_{x/y} v_{y/x}
  \\&
  + \frac{1}{\gamma-1} \rho^{\gamma-1} \exp(s) \rho s v_{x/y}
  - \frac{1}{2} \rho v^2 v_{x/y}
  - \frac{\gamma}{\gamma-1} \rho^\gamma v_{x/y} \exp(s)
  \\
  =&
  \rho^\gamma v_{x/y} \exp(s).
\end{aligned}
\end{equation}

As before, the conditions for 'entropy' conservation in the semidiscrete setting
of \citet{tadmor1987numerical, tadmor2003entropy} are
$\jump{w} \cdot \fnumxy - \jump{\psi_{x/y}} = 0$ \eqref{eq:EC}.
The corresponding 'entropy' fluxes are
$\Fnum_{x/y} = \mean{w} \cdot \fnumxy - \mean{\psi_{x/y}}$,
i.e. the numerical energy fluxes corresponding to an 'entropy' conservative flux
$\fnumxy$ for \eqref{eq:Euler-reversed} are
\begin{equation}
  \fnumxy_{\rho e}
  =
  \mean{w} \cdot \fnumxy - \mean{\psi_{x/y}},
\end{equation}
with $w$ as in \eqref{eq:Euler-reversed-w} and $\psi_{x/y}$ as in \eqref{eq:Euler-reversed-psi}.

Choosing $v$ as variable and writing the jumps using the product rule
\eqref{eq:arithmetic-mean-rule} as
\begin{equation}
\begin{aligned}
  \jump{w_1}
  =&
  - \mean{v_x} \jump{v_x}
  - \mean{v_y} \jump{v_y}
  + \frac{\gamma}{\gamma-1} \jump{ \rho^{\gamma-1} \exp(s) }
  - \frac{1}{\gamma-1} \jump{ \rho^{\gamma-1} s \exp(s) },
  \\
  \jump{w_2}
  =&
  \jump{v_x},
  \\
  \jump{w_3}
  =&
  \jump{v_y},
  \\
  \jump{w_4}
  =&
  \frac{1}{\gamma-1} \jump{ \rho^{\gamma-1} \exp(s) },
  \\
  \jump{\psi_x}
  =&
  \mean{\rho^\gamma \exp(s)} \jump{v_x} + \mean{v_x} \jump{\rho^\gamma \exp(s)},
  \\
  \jump{\psi_y}
  =&
  \mean{\rho^\gamma \exp(s)} \jump{v_y} + \mean{v_y} \jump{\rho^\gamma \exp(s)},
\end{aligned}
\end{equation}
the coefficients of $\jump{v_{x/y}}$ in the entropy conditions
$\jump{w} \cdot \fnumxy - \jump{\psi_{x/y}} = 0$ become
\begin{equation}
  - \mean{v_x} \fnumx_{\rho} + \fnumx_{\rho v_x} - \mean{\rho^\gamma \exp(s)},
  \qquad
  - \mean{v_y} \fnumy_{\rho} + \fnumy_{\rho v_y} - \mean{\rho^\gamma \exp(s)}.
\end{equation}
Thus, general momentum fluxes for entropy conservative numerical fluxes can be
written as
\begin{equation}
\begin{aligned}
  \fnumx&
  \begin{dcases}
  \fnumx_{\rho v_x}
  =
  \mean{v_x} \fnumx_{\rho} + \mean{\rho^\gamma \exp(s)}
  =
  \mean{v_x} \fnumx_{\rho} + \mean{p},
  \\
  \fnumx_{\rho v_y}
  =
  \mean{v_y} \fnumx_{\rho},
  \end{dcases}
  \\
  \fnumy&
  \begin{dcases}
  \fnumy_{\rho v_x}
  =
  \mean{v_x} \fnumy_{\rho},
  \\
  \fnumy_{\rho v_y}
  =
  \mean{v_y} \fnumy_{\rho} + \mean{\rho^\gamma \exp(s)}
  =
  \mean{v_y} \fnumy_{\rho} + \mean{p},
  \end{dcases}
\end{aligned}
\end{equation}
i.e. in the form proposed by \citet{jameson2008formulation} for a kinetic
energy preserving flux (in one space dimension).
Another possibility would be to split the jump of $\psi_{x/y}$ in some other way,
resulting in a numerical pressure flux $\pnum$ different from $\mean{p}$.

\subsection{Using \texorpdfstring{$\rho, v, T$}{ρ,v,T} as Variables}
\label{sec:reversed-rho-v-T}

Using the variables $\rho, v$, and $RT = \frac{p}{\rho} = \rho^{\gamma-1} \exp(s)$,
the flux potentials $\psi_{x/y} = \rho^\gamma v_{x/y} \exp(s)$ \eqref{eq:Euler-reversed-psi}
and the entropy variables \eqref{eq:Euler-reversed-w} can be written as
\begin{equation}
\label{eq:reversed-psi-w-using-rho-v-T}
  \psi_{x/y} = \rho RT v_{x/y},
  \quad
  w
  =
  \begin{pmatrix}
    - \frac{1}{2} v^2
    + \frac{\gamma}{\gamma-1} RT
    - \frac{1}{\gamma-1} s RT
    \\
    v_x
    \\
    v_y
    \\
    \frac{1}{\gamma-1} RT
  \end{pmatrix},
  \quad
  s 
  = \log RT - (\gamma-1) \log \rho.
\end{equation}

\subsubsection{Variant 1}

Using these variables, the jumps can be written using the chain rules
\eqref{eq:arithmetic-mean-rule} and \eqref{eq:logarithmic-mean-rule} as
\begin{equation}
\begin{aligned}
  \jump{w_1}
  =&
  - \frac{1}{2} \jump{v^2}
  + \frac{\gamma}{\gamma-1} \jump{ RT }
  - \frac{1}{\gamma-1} \jump{ RT \log RT }
  + \jump{RT \log \rho}
  \\
  =&
  - \mean{v_x} \jump{v_x}
  - \mean{v_y} \jump{v_y}
  + \frac{\gamma}{\gamma-1} \jump{ RT }
  - \frac{1}{\gamma-1} \frac{ \mean{RT} }{ \logmean{RT} } \jump{ RT }
  \\&
  - \frac{1}{\gamma-1} \mean{\log RT} \jump{ RT }
  + \mean{\log \rho} \jump{RT}
  + \frac{ \mean{RT} }{ \logmean{\rho} } \jump{\rho},
  \\
  \jump{w_2}
  =&
  \jump{v_x},
  \\
  \jump{w_3}
  =&
  \jump{v_y},
  \\
  \jump{w_4}
  =&
  \frac{1}{\gamma-1} \jump{ RT },
  \\
  \jump{\psi_x}
  =&
  \mean{\rho RT} \jump{v_x}
  + \mean{v_x} \jump{\rho RT}
  =
  \mean{\rho RT} \jump{v_x}
  + \mean{\rho} \mean{v_x} \jump{RT}
  + \mean{v_x} \mean{RT} \jump{\rho},
  \\
  \jump{\psi_y}
  =&
  \mean{\rho RT} \jump{v_y}
  + \mean{v_y} \jump{\rho RT}
  =
  \mean{\rho RT} \jump{v_y}
  + \mean{\rho} \mean{v_y} \jump{RT}
  + \mean{v_y} \mean{RT} \jump{\rho}.
\end{aligned}
\end{equation}
Inserting this in the entropy conditions $\jump{w} \cdot \fnumxy - \jump{\psi_{x/y}} = 0$,
\begin{equation}
\begin{aligned}
  0
  =&
  \left(
    \frac{ \mean{RT} }{ \logmean{\rho} } \fnumx_{\rho}
    - \mean{v_x} \mean{RT}
  \right) \jump{\rho}
  \\&
  + \left(
    - \mean{v_x} \fnumx_{\rho}
    + \fnumx_{\rho v_x}
    - \mean{\rho RT}
  \right) \jump{v_x}
  + \left(
    - \mean{v_y} \fnumx_{\rho}
    + \fnumx_{\rho v_y}
  \right) \jump{v_y}
  \\&
  + \Bigg(
    \left(
      \frac{\gamma}{\gamma-1}
      - \frac{1}{\gamma-1} \frac{ \mean{RT} }{ \logmean{RT} }
      - \frac{1}{\gamma-1} \mean{\log RT}
      + \mean{\log \rho}
    \right) \fnumx_{\rho}
    \\&\qquad\qquad
    + \frac{1}{\gamma-1} \fnumx_{\rho s}
    - \mean{\rho} \mean{v_x}
  \Bigg) \jump{RT},
  \\
  0
  =&
  \left(
    \frac{ \mean{RT} }{ \logmean{\rho} } \fnumy_{\rho}
    - \mean{v_y} \mean{RT}
  \right) \jump{\rho}
  \\&
  + \left(
    - \mean{v_x} \fnumy_{\rho}
    + \fnumy_{\rho v_x}
  \right) \jump{v_x}
  + \left(
    - \mean{v_y} \fnumy_{\rho}
    + \fnumy_{\rho v_y}
    - \mean{\rho RT}
  \right) \jump{v_y}
  \\&
  + \Bigg(
    \left(
      \frac{\gamma}{\gamma-1}
      - \frac{1}{\gamma-1} \frac{ \mean{RT} }{ \logmean{RT} }
      - \frac{1}{\gamma-1} \mean{\log RT}
      + \mean{\log \rho}
    \right) \fnumy_{\rho}
    \\&\qquad\qquad
    + \frac{1}{\gamma-1} \fnumy_{\rho s}
    - \mean{\rho} \mean{v_y}
  \Bigg) \jump{RT},
\end{aligned}
\end{equation}
the fluxes
\begin{equation}
\label{eq:reversed-rho-v-T-EC-KEP}
\begin{aligned}
  \fnumx&
  \begin{dcases}
  \fnumx_{\rho}
  =&
  \logmean{\rho} \mean{v_x},
  \\
  \fnumx_{\rho v_x}
  =&
  \mean{v_x} \fnumx_{\rho} + \mean{\rho RT}
  =
  \mean{v_x} \fnumx_{\rho} + \mean{p},
  \\
  \fnumx_{\rho v_y}
  =&
  \mean{v_y} \fnumx_{\rho},
  \\
  \fnumx_{\rho s}
  =&
  \left(
    \frac{ \mean{RT} }{ \logmean{RT} }
    - \gamma
    + \mean{\log RT}
    - (\gamma-1) \mean{\log \rho}
  \right) \fnumx_{\rho}
  + (\gamma-1) \mean{\rho} \mean{v_x},
  \\
  \fnumx_{\rho e}
  =&
  \left(
    - \frac{ \mean{v_x^2} + \mean{v_y^2} }{2}
    + \frac{\gamma}{\gamma-1} \mean{ RT }
    - \frac{1}{\gamma-1} \mean{ s RT }
  \right) \fnumx_{\rho}
  \\&
  + \mean{ v_x } \fnumx_{\rho v_x}
  + \mean{ v_y } \fnumx_{\rho v_y}
  + \frac{1}{\gamma-1} \mean{ RT } \fnumx_{\rho s}
  - \mean{\rho RT v_x},
  \end{dcases}
  \\
  \fnumy&
  \begin{dcases}
  \fnumy_{\rho}
  =&
  \logmean{\rho} \mean{v_y},
  \\
  \fnumy_{\rho v_x}
  =&
  \mean{v_x} \fnumy_{\rho},
  \\
  \fnumy_{\rho v_y}
  =&
  \mean{v_y} \fnumy_{\rho} + \mean{\rho RT}
  =
  \mean{v_y} \fnumy_{\rho} + \mean{p},
  \\
  \fnumy_{\rho s}
  =&
  \left(
    \frac{ \mean{RT} }{ \logmean{RT} }
    - \gamma
    + \mean{\log RT}
    - (\gamma-1) \mean{\log \rho}
  \right) \fnumy_{\rho}
  + (\gamma-1) \mean{\rho} \mean{v_y},
  \\
  \fnumy_{\rho e}
  =&
  \left(
    - \frac{ \mean{v_x^2} + \mean{v_y^2} }{2}
    + \frac{\gamma}{\gamma-1} \mean{ RT }
    - \frac{1}{\gamma-1} \mean{ s RT }
  \right) \fnumy_{\rho}
  \\&
  + \mean{ v_x } \fnumy_{\rho v_x}
  + \mean{ v_y } \fnumy_{\rho v_y}
  + \frac{1}{\gamma-1} \mean{ RT } \fnumy_{\rho s}
  - \mean{\rho RT v_y},
  \end{dcases}
\end{aligned}
\end{equation}
can be seen to be consistent, entropy conservative, and kinetic energy preserving
fluxes for the Euler equations \eqref{eq:Euler}, where the energy fluxes $\fnumxy_{\rho e}$
have been computed as $\mean{w} \cdot \fnumxy - \mean{\psi_{x/y}}$.

Inserting the numerical fluxes into the definition of the energy fluxes
$\fnumx_{\rho e},\fnumy_{\rho e}$ yields
\begin{equation}
\begin{aligned}
  \fnumx_{\rho e}
  =&
  \left(
    - \frac{ \mean{v_x^2} + \mean{v_y^2} }{2}
    + \frac{\gamma}{\gamma-1} \mean{ RT }
    - \frac{1}{\gamma-1} \mean{ s RT }
  \right) \logmean{\rho} \mean{v_x}
  \\&
  + \mean{ v_x } \left( \logmean{\rho} \mean{v_x}^2 + \mean{\rho RT} \right)
  + \mean{ v_y } \left( \logmean{\rho} \mean{v_x} \mean{v_y} \right)
  \\&
  + \frac{1}{\gamma-1} \mean{ RT }
      \left(
        \frac{ \mean{RT} }{ \logmean{RT} }
        - \gamma
        + \mean{\log RT}
        - (\gamma-1) \mean{\log \rho}
      \right) \logmean{\rho} \mean{v_x}
  \\&
  + \mean{ RT } \mean{\rho} \mean{v_x}
  - \mean{\rho v_x RT}
  \\
  =&
  \logmean{\rho} \mean{v_x} \left(
    \mean{v_x}^2 + \mean{v_y}^2 - \frac{ \mean{v_x^2} + \mean{v_y^2} }{2}
  \right)
  \\&
  + \mean{\rho} \mean{v_x} \mean{RT}
  + \mean{\rho RT} \mean{v_x}
  - \mean{\rho v_x RT}
  + \frac{1}{\gamma-1} \logmean{\rho} \mean{v_x} \frac{ \mean{RT}^2 }{ \logmean{RT} }
  \\&
  + \logmean{\rho} \mean{v_x} \left(
    \frac{1}{\gamma-1} \mean{ RT } \mean{\log RT}
    -  \mean{\log \rho} \mean{ RT }
    - \frac{1}{\gamma-1} \mean{ s RT }
  \right),
  \\
  \fnumy_{\rho e}
  =&
  \left(
    - \frac{ \mean{v_x^2} + \mean{v_y^2} }{2}
    + \frac{\gamma}{\gamma-1} \mean{ RT }
    - \frac{1}{\gamma-1} \mean{ s RT }
  \right) \logmean{\rho} \mean{v_y}
  \\&
  + \mean{ v_x } \left( \logmean{\rho} \mean{v_x} \mean{v_y} \right)
  + \mean{ v_y } \left( \logmean{\rho} \mean{v_y}^2 + \mean{\rho RT} \right)
  \\&
  + \frac{1}{\gamma-1} \mean{ RT }
      \left(
        \frac{ \mean{RT} }{ \logmean{RT} }
        - \gamma
        + \mean{\log RT}
        - (\gamma-1) \mean{\log \rho}
      \right) \logmean{\rho} \mean{v_y}
  \\&
  + \mean{ RT } \mean{\rho} \mean{v_y}
  - \mean{\rho v_y RT}
  \\
  =&
  \logmean{\rho} \mean{v_y} \left(
    \mean{v_x}^2 + \mean{v_y}^2 - \frac{ \mean{v_x^2} + \mean{v_y^2} }{2}
  \right)
  \\&
  + \mean{\rho} \mean{v_y} \mean{RT}
  + \mean{\rho RT} \mean{v_y}
  - \mean{\rho v_y RT}
  + \frac{1}{\gamma-1} \logmean{\rho} \mean{v_y} \frac{ \mean{RT}^2 }{ \logmean{RT} }
  \\&
  + \logmean{\rho} \mean{v_y} \left(
    \frac{1}{\gamma-1} \mean{ RT } \mean{\log RT}
    -  \mean{\log \rho} \mean{ RT }
    - \frac{1}{\gamma-1} \mean{ s RT }
  \right).
\end{aligned}
\end{equation}
Here, the first two lines of the results are consistent approximations of the
fluxes $\frac{1}{2} \rho v^2 v_{x / y} + \frac{\gamma}{\gamma-1} p$ with
additional terms that are consistent with zero, since
$s = \log RT - (\gamma-1) \log \rho$.

\subsection{Other Variables}

As in section \ref{sec:fluxes}, other choices of variables are possible, e.g.
$\rho, v, \chi$  with $\frac{p}{\rho} = g(\chi)$. However, this approach is
not further pursued here, since there does not seem to be a clear intuition,
which choice is preferable.

\section{Numerical Surface Fluxes / Riemann Solvers}
\label{sec:surface-fluxes}

The numerical fluxes $\fnum$ used in the surface terms \eqref{eq:surface-terms} 
of the semidiscretisation \eqref{eq:semidiscretisation} are an important ingredient
for stability and robustness of the method. In a first order finite volume setting, 
they determine the method completely.
Here, some choices of these numerical fluxes will be presented and compared. Since
they are used in the surface terms, they are frequently called surface fluxes
in order to distinguish them from the volume fluxes $\fvol$ used in the volume
terms \eqref{eq:volume-terms} of the semidiscretisation \eqref{eq:semidiscretisation}.

\subsection{Adding Dissipation to Entropy Conservative Fluxes}

Similarly to the local Lax-Friedrichs flux
\begin{equation}
\label{eq:LLF}
  \fnum_\mathrm{LLF}(u_-,u_+)
  =
  \frac{f(u_+) + f(u_-)}{2} - \frac{\lambda}{2} (u_+ - u_-),
\end{equation}
an entropy stable flux can be constructed as an entropy conservative central flux
plus an additional dissipation term.

The simplest choice is to add a local Lax-Friedrichs type dissipation of the form
$- \frac{\lambda}{2} \jump{u}$. The resulting flux is entropy stable, since
multiplication with the jump of the entropy variables results in
\begin{equation}
\begin{aligned}
  - \frac{\lambda}{2} \jump{w} \cdot \jump{u}
  =&
  - \frac{\lambda}{2} \jump{w} \cdot
  \int_{0}^1 \frac{\partial u}{\partial w}\left(w_- + \sigma(w_+ - w_-) \right)
  \cdot (w_+ - w_-) \dif \sigma
  \\
  =&
  - \frac{\lambda}{2} \jump{w} \cdot
  \int_{0}^1 \frac{\partial u}{\partial w}\left(w_- + \sigma(w_+ - w_-) \right) \dif \sigma
  \cdot \jump{w},
\end{aligned}
\end{equation}
and $\frac{\partial u}{\partial w} = \left( \frac{\partial w}{\partial u} \right)^{-1}$
is positive definite, since $w = \partial_u U$ and the entropy $U$ is convex.

Another construction uses a dissipation term of the form $- \frac{1}{2} \abs{f'(u)} \jump{u}
\approx - \frac{1}{2} \abs{f'(u)} \partial_w u \cdot \jump{w}$.
Using the scaling of the eigenvectors proposed by Barth \cite[Theorem 4]{barth1999numerical}
results in $\abs{f'(u)} = R \abs{\Lambda} \inv R$ and $\partial_w u = R R ^T$,
where $\Lambda$ contains the eigenvalues of $f'(u)$ on the diagonal. Thus,
$- \abs{f'(u)} \jump{u} \approx - R \abs{\Lambda} R^T \cdot \jump{w}$ and the matrix
$R \abs{\Lambda} R^T$ is positive definite.

Using this form, choosing $\abs{\Lambda} = \lambda \I$ and some intermediate value
$H = R R^T = \partial_w u$ results in a scalar dissipation term (SD)
$- \frac{\lambda}{2} H \jump{w}$. A matrix dissipation term (MD) is obtained by
setting $\abs{\Lambda} = \diag{ \abs{\lambda_i} }$.

Of course, the matrices $H, R, \abs{\Lambda}$ have to be evaluated at some suitable
intermediate values. Derigs et al. \cite{derigs2016novelAveraging, winters2016uniquely}
investigated this problem for ideal MHD and the Euler equations using the
entropy conservative flux \eqref{eq:Chandrashekar-rho-v-beta-EC-KEP} of
\cite{chandrashekar2013kinetic} and derived scalar and matrix dissipation operators
of the forms described above. Additionally, they proposed to use the convex combination
$- \Xi \frac{\lambda}{2} H \jump{w} - (1-\Xi) \frac{1}{2} R \abs{\Lambda} R^T \cdot \jump{w}$,
where $\Xi = \abs{ \frac{p_+ - p_-}{p_+ + p_-} }^\frac{1}{2}$ is the indicator
of the shock strength also used in \cite{chandrashekar2013kinetic}.

\subsection{Preserving Positivity of the Density}
\label{subsec:positivity-density}

The Euler equations are valid for positive density $\rho$ and pressure $p$.
In order to be robust, the numerical flux $\fnum$ should preserve these invariant
regions of the Euler equations in a first order finite volume update procedure
using an explicit Euler step in time
\begin{equation}
\label{eq:FV}
  u_i^+
  =
  u_i - \frac{\Delta t}{\Delta x} \left(
    \fnum(u_i, u_{i+1}) - \fnum(u_{i-1}, u_i)
  \right).
\end{equation}
Extensions of this property to higher order methods can be constructed using the
framework of Zhang and Shu \cite{zhang2011maximum}. As described inter alia in
\cite[Remark 2.4]{zhang2010positivity}, the (local) Lax-Friedrichs flux preserves 
positivity of both density and pressure. Here, the entropy conservative numerical
fluxes described in sections~\ref{sec:fluxes} and \ref{sec:reversed-fluxes} are
investigated. The main result concerning positivity of the density is
\begin{theorem}
\label{thm:positivity-general}
  Suppose that the numerical density flux $\fnum_\rho$ can be written as
  $\fnum_\rho = \overline\rho \cdot \mean{v} - \frac{\lambda}{2} \jump{\rho}$,
  where $\lambda \geq \max\set{ \abs{v_i}, \abs{v_{i+1}} }$ and $\overline\rho$
  is some mean value satisfying $\overline\rho \leq \mean{\rho}$, i.e.
  \begin{equation}
  \label{eq:mean-value-rho}
    \overline\rho(\rho_i, \rho_{i+1}) 
    \in \big[ \min\set{\rho_i, \rho_{i+1}}, \max\set{\rho_i, \rho_{i+1}} \big],
    \;
    \overline\rho(\rho_i, \rho_{i+1}) \leq \frac{\rho_i + \rho_{i+1}}{2}.
  \end{equation}
  Then the first order FV scheme \eqref{eq:FV} preserves the non-negativity of
  the density $\rho$ under the CFL condition
  \begin{equation}
  \label{eq:CFL-positivity-general}
    \Delta t \leq \frac{\Delta x}{2 \lambda}.
  \end{equation}
\end{theorem}

\begin{remark}
  Note that the CFL condition \eqref{eq:CFL-positivity-general} does not depend
  explicitly on the densities and that it does not require a vanishing time step
  $\Delta t$ for $\rho \to 0$.
\end{remark}

\begin{corollary}
\label{cor:positivity-logmean}
  If the numerical density flux
  $\fnum_\rho = \logmean{\rho} \mean{v} - \frac{\lambda}{2} \jump{\rho}$
  is used with $\lambda \geq \max\set{ \abs{v_i}, \abs{v_{i+1}} }$, the first 
  order FV scheme \eqref{eq:FV} preserves the non-negativity of the density $\rho$ 
  under the CFL condition \eqref{eq:CFL-positivity-general}.
\end{corollary}

\begin{proof}[Proof of Corollary~\ref{cor:positivity-logmean}]
  The logarithmic mean fulfils the conditions of Theorem~\ref{thm:positivity-general},
  especially $\logmean{\rho} \leq \mean{\rho}$ \cite{chen2005means}.
\end{proof}

\begin{remark}
\label{rem:positivity-density}
  Due to Corollary~\ref{cor:positivity-logmean}, if the entropy conservative fluxes 
  described in sections~\ref{sec:fluxes} and \ref{sec:reversed-fluxes} containing
  no contribution of the pressure in the density flux are used with dissipation 
  of LLF type in the variable $\rho$, the first order FV scheme \eqref{eq:FV} 
  preserves the non-negativity of the density $\rho$ under the CFL condition 
  \eqref{eq:CFL-positivity-general}. This is also fulfilled for the scalar 
  dissipation operator SD of \cite{derigs2016novelAveraging}.
  Contrary, the entropy conservative fluxes containing an influence of the pressure
  in the density flux do not fulfil the conditions of \autoref{thm:positivity-general}.
  This agrees with physical intuition, since the pressure has no influence on the
  density flux in the Euler equations \eqref{eq:Euler} \cite{derigs2016novelAveraging}.
\end{remark}

\begin{proof}[Proof of Theorem~\ref{thm:positivity-general}]
  The FV step \eqref{eq:FV} for the density can be separated into two parts as
  \begin{equation}
    \rho_i^+
    =
    \left( \frac{1}{2} \rho_i - \frac{\Delta t}{\Delta x} \fnum_\rho(u_i, u_{i+1}) \right)
    +
    \left( \frac{1}{2} \rho_i + \frac{\Delta t}{\Delta x} \fnum_\rho(u_{i-1}, u_i) \right).
  \end{equation}
  Since both can be handled similarly, only the first one will be analysed.
  Inserting the numerical density flux, the mean value $\overline\rho$ can be
  written as a convex combination $\overline\rho = \alpha \rho_i + (1-\alpha) \rho_{i+1}$,
  $\alpha \in [0, 1]$. Thus, the first term becomes
  \begin{equation}
  \begin{aligned}
    &
    \frac{1}{2} \rho_i - \frac{\Delta t}{\Delta x} \left(
      \overline\rho \mean{v} - \frac{\lambda}{2} \jump{\rho}
    \right)
    \\
    =&
    \rho_i \left(
      \frac{1}{2}
      - \frac{\lambda}{2} \frac{\Delta t}{\Delta x}
      - \alpha \mean{v} \frac{\Delta t}{\Delta x}
    \right)
    + \rho_{i+1} \frac{\Delta t}{\Delta x} \left(
      \frac{\lambda}{2} 
      - (1-\alpha) \mean{v}
    \right).
  \end{aligned}
  \end{equation}
  Using $\rho_i,\rho_{i+1} \geq 0$ and $\lambda \geq \max\set{ \abs{v_i}, \abs{v_{i+1}} }$,
  \begin{equation}
  \label{eq:positivity-estimate-1}
    \frac{1}{2} \rho_i - \frac{\Delta t}{\Delta x} \left(
      \overline\rho \mean{v} - \frac{\lambda}{2} \jump{\rho}
    \right)
    \geq
    \rho_i \left(
      \frac{1}{2}
      - \left( \frac{1}{2} + \alpha \right) \frac{\lambda \Delta t}{\Delta x}
    \right)
    + \rho_{i+1} \frac{\lambda \Delta t}{\Delta x} \left( \alpha - \frac{1}{2} \right).
  \end{equation}
  Two cases can be considered.
  \begin{enumerate}
    \item $\rho_i \geq \rho_{i+1}$.
    In this case, $\alpha \leq \frac{1}{2}$, since $\overline\rho \leq \mean{\rho}$.
    Thus, the second term on the right hand side of \eqref{eq:positivity-estimate-1}
    can be bounded as
    \begin{equation}
      \label{eq:positivity-estimate-2}
      \rho_{i+1} \frac{\lambda \Delta t}{\Delta x} \left( \alpha - \frac{1}{2} \right)
      \geq
      \rho_{i} \frac{\lambda \Delta t}{\Delta x} \left( \alpha - \frac{1}{2} \right).
    \end{equation}
    
    \item $\rho_i \leq \rho_{i+1}$.
    In this case, $\alpha \geq \frac{1}{2}$, since $\overline\rho \leq \mean{\rho}$.
    Again, the second term on the right hand side of \eqref{eq:positivity-estimate-1} 
    can be bounded via \eqref{eq:positivity-estimate-2}.
  \end{enumerate}
  In both cases, the term with $\rho_{i+1}$ in \eqref{eq:positivity-estimate-1} 
  can be estimated via $\rho_i$, yielding
  \begin{equation}
  \begin{aligned}
    \frac{1}{2} \rho_i - \frac{\Delta t}{\Delta x} \left(
      \overline\rho \mean{v} - \frac{\lambda}{2} \jump{\rho}
    \right)
    &\geq
    \rho_i \left(
      \frac{1}{2}
      - \left( \frac{1}{2} + \alpha \right) \frac{\lambda \Delta t}{\Delta x}
    \right)
    + \rho_{i} \frac{\lambda \Delta t}{\Delta x} \left( \alpha - \frac{1}{2} \right)
    \\
    &=
    \rho_i \left( \frac{1}{2} - \frac{\lambda \Delta t}{\Delta x} \right).
  \end{aligned}
  \end{equation}
  This is non-negative under the CFL condition \eqref{eq:CFL-positivity-general}.
\end{proof}

\subsection{Preserving Positivity of the Pressure}

Preserving the positivity of the pressure / internal energy is more complicated
than the corresponding property of the density. For the (local) Lax-Friedrichs
flux, it can be proven as described inter alia in
\cite[Remark 2.4]{zhang2010positivity}. Further investigations have to be
conducted for the case of the numerical fluxes considered here.

As a general procedure, the reversed roles of entropy and energy as in section
\ref{sec:reversed-fluxes} can be used to get entropy stable fluxes that preserve
the positivity of the internal energy as described by Bouchut~\cite{bouchut2004nonlinear}.
This corresponds to a computation of the pressure via the entropy, which has also
been used in \cite{derigs2016novelSolver} in an \emph{a posteriori} manner.
However, this direction of further research will not be pursued here.

\subsection{Suliciu Relaxation Solver}
\label{subsec:Suliciu}

The Suliciu relaxation solver described in \cite[Section 2.4]{bouchut2004nonlinear}
for the two-dimensional Euler equations in $x$ direction can be summed up as follows.
At first, intermediate wave speeds are computed via
\begin{small}
\begin{align*}
\stepcounter{equation}\tag{\theequation}
  \text{if } p_+ \geq p_-,
  &
  \begin{dcases}
    \frac{c_-}{\rho_-}
    =
    \sqrt{\gamma \frac{p_-}{\rho_-}}
    + \frac{\gamma+1}{2} \max\set{
      \frac{p_+ - p_-}{\rho_+ \sqrt{\gamma p_+ / \rho_+}} + v_{x-} - v_{x+}, 0},
    \\
    \frac{c_+}{\rho_+}
    =
    \sqrt{\gamma \frac{p_+}{\rho_+}}
    + \frac{\gamma+1}{2} \max\set{
      \frac{p_- - p_+}{c_-} + v_{x-} - v_{x+}, 0},
  \end{dcases}
  \\
  \text{if } p_+ \leq p_-,
  &
  \begin{dcases}
    \frac{c_+}{\rho_+}
    =
    \sqrt{\gamma \frac{p_+}{\rho_+}}
    + \frac{\gamma+1}{2} \max\set{
      \frac{p_- - p_+}{\rho_- \sqrt{\gamma p_- / \rho_-}} + v_{x-} - v_{x+}, 0},
    \\
    \frac{c_-}{\rho_-}
    =
    \sqrt{\gamma \frac{p_-}{\rho_-}}
    + \frac{\gamma+1}{2} \max\set{
      \frac{p_+ - p_-}{c_+} + v_{x-} - v_{x+}, 0}.
  \end{dcases}
\end{align*}
\end{small}%
Then, intermediate values are computed using $c_\pm = \rho_\pm \frac{c_\pm}{\rho_\pm}$ as
\begin{small}
\begin{equation}
\begin{aligned}
  p_-^*
  =&
  p_+^*
  =
  \frac{c_+ p_- + c_- p_+ - c_- c_+ (v_{x+} - v_{x-})}{c_- + c_+},
  &
  \frac{1}{\rho_+^*}
  =&
  \frac{1}{\rho_+} + \frac{c_- (v_{x+} - v_{x-}) + p_+ - p_-}{c_+ (c_- + c_+)},
  \\
  v_{x-}^*
  =&
  v_{x+}^*
  =
  \frac{c_- v_{x-} + c_+ v_{x+} + p_- - p_+}{c_- + c_+},
  &
  \epsilon_-^*
  =&
  \epsilon_- + \frac{(p_-^*)^2 - p_-^2}{2 c_-^2},
  \\
  \frac{1}{\rho_-^*}
  =&
  \frac{1}{\rho_-} + \frac{c_+ (v_{x+} - v_{x-}) + p_- - p_+}{c_- (c_- + c_+)},
  &
  \epsilon_+^*
  =&
  \epsilon_+ + \frac{ (p_+^*)^2 - p_+^2 }{2 c_+^2}.
\end{aligned}
\end{equation}
\end{small}%
Finally, the numerical fluxes are given by ($\fnumy$ analogously)
\begin{small}
\begin{equation}
\label{eq:Suliciu}
  \vect{ \fnum_{\rho} \\ \fnum_{\rho v_x} \\ \fnum_{\rho v_y} \\ \fnum_{\rho e} }
  =
  \begin{dcases}
    \vect{
      \rho_- v_{x-} \\
      \rho_- v_{x-}^2 + p_- \\
      \rho_- v_{x-} v_{y,-} \\
      \left( \frac{1}{2} \rho_- v_{x-}^2 + \frac{1}{2} \rho_- v_{y-}^2
              + \rho_- \epsilon_- + p_- \right) v_{x-}
    },
    \\ \qquad\qquad\qquad \text{if } 0 \leq v_{x-} - \frac{c_-}{\rho_-},
    \\
    \vect{
      \rho_-^* v_{x-}^* \\
      \rho_-^* (v_{x-}^*)^2 + p_-^* \\
      \rho_-^* v_{x-}^* v_{y,-} \\
      \left( \frac{1}{2} \rho_-^* (v_{x-}^*)^2 + \frac{1}{2} \rho_-^* (v_{y-}^*)^2
              + \rho_-^* \epsilon_-^* + p_-^* \right) v_{x-}^*
    },
    \\ \qquad\qquad\qquad \text{if } v_{x-} - \frac{c_-}{\rho_-} < 0 \leq v_{x-}^* \equiv v_{x+}^*,
    \\
    \vect{
      \rho_+^* v_{x+}^* \\
      \rho_+^* (v_{x+}^*)^2 + p_+^* \\
      \rho_+^* v_{x+}^* v_{y,+} \\
      \left( \frac{1}{2} \rho_+^* (v_{x+}^*)^2 + \frac{1}{2} \rho_+^* (v_{y+}^*)^2
              + \rho_+^* \epsilon_+^* + p_+^* \right) v_{x+}^*
    },
    \\ \qquad\qquad\qquad \text{if } v_{x-}^* \equiv v_{x+}^* < 0 \leq v_{x+} + \frac{c_+}{\rho_+},
    \\
    \vect{
      \rho_+ v_{x+} \\
      \rho_+ v_{x+}^2 + p_+ \\
      \rho_+ v_{x+} v_{y,+} \\
      \left( \frac{1}{2} \rho_+ v_{x+}^2 + \frac{1}{2} \rho_+ v_{y+}^2
              + \rho_+ \epsilon_+ + p_+ \right) v_{x+}
    },
    \\ \qquad\qquad\qquad \text{else}.
  \end{dcases}
\end{equation}
\end{small}%
This flux is entropy stable and positivity preserving for $\rho$ and $p$, with
corresponding CFL condition
\begin{equation}
\label{eq:Suliciu-CLF}
  \frac{\Delta t}{\Delta x} \max\set{
    \abs{v_{x-} - \frac{c_-}{\rho_-}}, \abs{v_{x+} + \frac{c_+}{\rho_+}} } \leq \frac{1}{2}.
\end{equation}
Additionally, it satisfies the maximum principle on the specific entropy $s$ and
resolves stationary contact discontinuities with $v_x \equiv 0, p \equiv \mathrm{const}$
exactly.

\section{Numerical Tests}
\label{sec:numerical-tests}

In this section, some numerical experiments using the methods described in the
previous sections will be conducted. Unless stated otherwise, the ratio of
specific heats is set to $\gamma = 1.4$ and the three stage, third-order,
strong stability preserving Runge-Kutta method of \citet{gottlieb1998total}
will be used as time integration method. If a one-dimensional Riemann problem 
is considered, the exact solution is computed as described in 
\cite[Section 4]{toro2009riemann}.

\subsection{Isentropic Vortex}
\label{sec:isentropic-vortex}

At first, the isentropic vortex problem of \citet[Problem 8 in section 5.1]{shu1997essentially}
will be used to test the methods for a smooth solution. The initial condition
is given by
\begin{equation}
\label{eq:isentropic-vortex}
  \begin{pmatrix}
    \rho_0 \\
    \rho v_{x,0} \\
    \rho v_{y,0} \\
    \rho e_0
  \end{pmatrix}
  =
  \begin{pmatrix}
    \rho_\infty \left( \frac{RT_0}{RT_\infty} \right)^{1/(\gamma-1)} \\
    \rho_0 v_{x,0} \\
    \rho_0 v_{y,0} \\
    \rho \frac{v_{x,0}^2 +  v_{y,0}^2}{2} + \frac{p_0}{\gamma-1}
  \end{pmatrix},
\end{equation}
where
\begin{equation}
  v_{x,0} = v_{x,\infty} + \delta v_x,
  \quad
  v_{y,0} = v_{y,\infty} + \delta v_y,
  \quad
  p_0 = \rho_0 \, RT_0,
  \quad
  RT_0 = RT_\infty + \delta RT,
\end{equation}
and
\begin{gather}
  \delta v_x(x,y)
  =
  -y \frac{\beta}{2\pi} \exp\left( \frac{1-x^2-y^2}{2 r} \right),
  \quad
  \delta v_y(x,y)
  =
  x \frac{\beta}{2\pi} \exp\left( \frac{1-x^2-y^2}{2 r} \right),
  \\
  \delta RT(x,y) 
  = 
  - \frac{\gamma-1}{\gamma} \frac{\beta}{8\pi^2} \exp\left( \frac{1-x^2-y^2}{r} \right).
\end{gather}
The parameters have been chosen as
\begin{equation}
  \rho_\infty = 1,
  \quad
  v_{x,\infty} = 1,
  \quad
  v_{y,\infty} = 0,
  \quad
  RT_\infty = 1,
  \quad
  \beta = 5,
  \quad
  r = \frac{1}{4}.
\end{equation}
The solution is computed on the domain $[-5,5]^2$ during the time interval $[0,10]$.
Thus, the perturbation is of the order of magnitude
$\exp\left( \frac{1-5^2}{2 r} \right) \approx 10^{-21}$ at the boundary and should
be approximately negligible for $64$ bit floating point values.
Therefore, the isentropic vortex should be advected with the free stream velocity
and reach its initial position at $t = 10$.

The Suliciu relaxation solver has been used as numerical flux and several
volume fluxes have been used for the subcell flux differencing form:
\begin{itemize}
  \item 
  Central:
  The central flux $\mean{f}$ resulting in a standard nodal DG method
  as described by \citet{gassner2016split}.
  
  \item
  Morinishi:
  The flux resulting in the split form of \citet{morinishi2010skew}
  as described by \citet{gassner2016split}.
  
  \item
  Ducros:
  The flux resulting in the split form of \citet{ducros2000high}
  as described by \citet{gassner2016split}.
  
  \item
  KG:
  The flux resulting in the split form of \citet{kennedy2008reduced}
  as described by \citet{gassner2016split}.
  
  \item
  Pirozzoli:
  The flux resulting in the split form of \citet{pirozzoli2011numerical}
  as described by \citet{gassner2016split}.
  
  \item
  IR:
  The entropy conservative flux \eqref{eq:Roe-EC} of \citet{roe2006affordable}
  and \citet{ismail2009affordable}.
  
  \item
  Ch:
  The entropy conservative flux \eqref{eq:Chandrashekar-rho-v-beta-EC-KEP} of
  \citet{chandrashekar2013kinetic}.
\end{itemize}
\begin{multicols}{2}
\begin{itemize}
  \item
  $\rho,v,\beta~(2)$:
  The flux \eqref{eq:Chandrashekar-rho-v-beta-EC-notKEP}.
  
  \item
  $\rho,v,\frac{1}{p}$:
  The flux \eqref{eq:rho-v-1/p-EC-KEP}.
  
  \item
  $\rho,v,p$:
  The flux \eqref{eq:rho-v-p-EC-KEP}.
  
  \item
  $\rho,v,T~(1)$:
  The flux \eqref{eq:rho-v-T-EC-KEP}.
  
  \item
  $\rho,v,T~(2)$:
  The flux \eqref{eq:rho-v-T-EC-notKEP}.
  
  \item
  $\rho,v,T~(\text{rev})$:
  The flux \eqref{eq:reversed-rho-v-T-EC-KEP}.
\end{itemize}
\end{multicols}
The errors (computed via the mass matrix $\mat{M}$, i.e. Lobatto quadrature) in
the density for varying polynomial degrees $p$ on a mesh using
$10 \times 10$ elements are shown in Table \ref{tab:isentropic-vortex}.
As can be seen there, there is not much variance in the error for $p = 1$.
For $p \in \set{2,3}$, there are variations up to approximately \num{15}{\%}
[e.g. $p=2$, Ducros, Morinishi and $p=3$, Ducros, $\rho,v,T$ (rev)].
There does not seem to be any advantage of the entropy stable formulations compared
to the other ones in this test case, similar to the results of \citet{gassner2016split}
for the Taylor Green vortex.

\begin{table}[!htp]
\small
\caption{Errors for the isentropic vortex problem \eqref{eq:isentropic-vortex}
         using $10 \times 10$ elements of varying polynomial degrees $p$,
         several volume fluxes and the Suliciu relaxation solver as numerical flux.}
\label{tab:isentropic-vortex}
\sisetup{
  output-exponent-marker=\text{e},
  round-mode=places,
  round-precision=2
}
\begin{center}
\begin{tabular}{c | c|c|c|c|c|c|c}
  \toprule
  & Central
  & Morinishi
  & Ducros
  & KG
  & Pirozzoli
  & IR
  & Ch
  \\
  $p$
    & $\norm{\mathrm{err_\rho}}_M$
    & $\norm{\mathrm{err_\rho}}_M$
    & $\norm{\mathrm{err_\rho}}_M$
    & $\norm{\mathrm{err_\rho}}_M$
    & $\norm{\mathrm{err_\rho}}_M$
    & $\norm{\mathrm{err_\rho}}_M$
    & $\norm{\mathrm{err_\rho}}_M$
  \\
  \midrule
      1  &  \num{1.203e-01}            &  \num{1.211e-01}            &  \num{1.199e-01}            &  \num{1.200e-01}            &  \num{1.200e-01}            &  \num{1.204e-01}            &  \num{1.207e-01}            \\
      2  &  \num{4.277e-02}            &  \num{4.356e-02}            &  \num{3.793e-02}            &  \num{3.830e-02}            &  \num{3.829e-02}            &  \num{3.929e-02}            &  \num{4.114e-02}            \\
      3  &  \num{1.728e-02}            &  \num{1.810e-02}            &  \num{1.595e-02}            &  \num{1.622e-02}            &  \num{1.620e-02}            &  \num{1.644e-02}            &  \num{1.796e-02}            \\
  \bottomrule
  \toprule
  & $\rho,v,\beta$ (2)
  & $\rho,v,\frac{1}{p}$
  & $\rho,v,p$
  & $\rho,v,T$ (1)
  & $\rho,v,T$ (2)
  & $\rho,V,T$ (rev)
  \\
  $p$
    & $\norm{\mathrm{err_\rho}}_M$
    & $\norm{\mathrm{err_\rho}}_M$
    & $\norm{\mathrm{err_\rho}}_M$
    & $\norm{\mathrm{err_\rho}}_M$
    & $\norm{\mathrm{err_\rho}}_M$
    & $\norm{\mathrm{err_\rho}}_M$
  \\
  \midrule
      1  &  \num{1.207e-01}            &  \num{1.210e-01}            &  \num{1.210e-01}            &  \num{1.207e-01}            &  \num{1.204e-01}            &  \num{1.203e-01}            \\
      2  &  \num{4.128e-02}            &  \num{4.107e-02}            &  \num{4.107e-02}            &  \num{4.114e-02}            &  \num{3.929e-02}            &  \num{4.124e-02}            \\
      3  &  \num{1.805e-02}            &  \num{1.790e-02}            &  \num{1.790e-02}            &  \num{1.796e-02}            &  \num{1.644e-02}            &  \num{1.892e-02}            \\
  \bottomrule
\end{tabular}
\end{center}
\end{table}

\subsection{Sod's Shock Tube: Subcell Flux Differencing}
\label{sec:sod-flux-diff}

In this section, the classical shock tube of \citet{sod1978survey} will be used
to test the semidiscretisations \eqref{eq:semidiscretisation} using the volume
terms \eqref{eq:volume-terms} and the surface terms \eqref{eq:surface-terms}.
The initial condition is given in primitive variables by
\begin{equation}
\label{eq:sod}
\begin{aligned}
  \rho_0(x) =
  \begin{cases}
    1,    & x < \frac{1}{2}, \\
    0.125, & \text{ else},
  \end{cases}
  ,\qquad
  v_0(x) = 0
  ,\qquad
  p_0(x) =
  \begin{cases}
    1, & x < \frac{1}{2}, \\
    0.1, & \text{ else},
  \end{cases}
\end{aligned}
\end{equation}
and the conservative variables are computed via $\rho v_0 = \rho_0 v_0$ and
$\rho e_0 = \frac{1}{2} \rho_0 v_0^2 + \frac{p_0}{\gamma-1}$.
The solution is computed on the domain $[-0.5, 1.5]$ from $t = 0$ until $t = 0.25$
using \num{3000} time steps.

The error of the numerical solution for the density $\rho$ is calculated using
Lobatto quadrature, i.e. the diagonal mass matrix $\mat{M}$ of the SBP operator,
for varying numbers of elements $N$ and polynomial degrees $p$. The results 
using the Suliciu relaxation solver and the local Lax-Friedrichs flux as
numerical flux are shown in Tables \ref{tab:sod-suliciu} and \ref{tab:sod-LLF},
respectively. 

There is some variance across the results for different volume fluxes up to
approximately \num{25}{\%} [e.g. for $p=1$, $N=10$, $(\rho,v,p)$ vs.
$\rho,v,T \text{(rev)}$].
However, there is no clear bias towards one volume flux across all combinations
of the polynomial degree $p$ and the number of elements $N$ [e.g. for $p=1$ and
$N=320$, $(\rho,v,p)$ has a smaller error than $\rho,v,T \text{(rev)}$].

Contrary, comparing the numerical surface fluxes, there is some clear bias.
Although the local Lax-Friedrichs flux yields a smaller error in some cases
[e.g. $p=1$, $N=10$, $(\rho,v,p)$], the Suliciu relaxation solver results in
smaller errors if the resolution is good enough. Therefore, it can be considered
superior to the LLF flux in this test case.

Additionally, the volume fluxes recovering the central form as well as the split
forms of \citet{morinishi2010skew}, \citet{ducros2000high}, \citet{kennedy2008reduced},
and \citet{pirozzoli2011numerical} have been used. The relevant results are shown
in the Tables \ref{tab:sod-nonEC-suliciu} (Suliciu) and \ref{tab:sod-nonEC-LLF}
(LLF). As can be seen there, the central flux and the splitting of \citet{morinishi2010skew}
are unstable. The splitting of \citet{ducros2000high} crashes for polynomial
degree $p = 5$ whereas the other splittings remain stable. There is not much
variance in the error for the stable calculations.

As for the entropy conservative volume fluxes, the Suliciu relaxation solver yields
less error if the resolution is good enough.

\begin{sidewaystable}
\small
\caption{Errors and experimental order of convergence (EOC) for varying polynomial
         degrees $p$ and number of elements $N$ for the Sod shock tube problem
         \eqref{eq:sod} using the Suliciu relaxation solver as numerical flux and
         entropy conserving volume fluxes.}
\label{tab:sod-suliciu}
\sisetup{
  output-exponent-marker=\text{e},
  round-mode=places,
  round-precision=2
}
\begin{center}
\begin{tabular}{rr | rr|rr|rr|rr|rr|rr|rr|rr}
  \toprule
  & &  \multicolumn{2}{c}{IR}
    &  \multicolumn{2}{c}{Ch} 
    &  \multicolumn{2}{c}{$\rho,v,\beta$ (2)}
    &  \multicolumn{2}{c}{$\rho,v,\frac{1}{p}$}
    &  \multicolumn{2}{c}{$\rho,v,p$}
    &  \multicolumn{2}{c}{$\rho,v,T$ (1)}
    &  \multicolumn{2}{c}{$\rho,v,T$ (2)} 
    &  \multicolumn{2}{c}{$\rho,V,T$ (rev)}
  \\
  $p$ & $N$
    & $\norm{\mathrm{err_\rho}}_M$ & EOC
    & $\norm{\mathrm{err_\rho}}_M$ & EOC
    & $\norm{\mathrm{err_\rho}}_M$ & EOC
    & $\norm{\mathrm{err_\rho}}_M$ & EOC
    & $\norm{\mathrm{err_\rho}}_M$ & EOC
    & $\norm{\mathrm{err_\rho}}_M$ & EOC
    & $\norm{\mathrm{err_\rho}}_M$ & EOC
    & $\norm{\mathrm{err_\rho}}_M$ & EOC
  \\
  \midrule
   1  &    10  &  \num{1.072e-01}  &               &  \num{1.063e-01}  &               &  \num{1.068e-01}  &               &  \num{1.368e-01}  &               &  \num{1.368e-01}  &               &  \num{1.063e-01}  &               &  \num{1.057e-01}  &               &  \num{1.031e-01}  &               \\
      &    20  &  \num{6.898e-02}  &  \num{ 0.64}  &  \num{6.852e-02}  &  \num{ 0.63}  &  \num{7.066e-02}  &  \num{ 0.60}  &  \num{6.941e-02}  &  \num{ 0.98}  &  \num{6.941e-02}  &  \num{ 0.98}  &  \num{6.852e-02}  &  \num{ 0.63}  &  \num{6.870e-02}  &  \num{ 0.62}  &  \num{6.829e-02}  &  \num{ 0.60}  \\
      &    40  &  \num{4.788e-02}  &  \num{ 0.53}  &  \num{4.825e-02}  &  \num{ 0.51}  &  \num{5.153e-02}  &  \num{ 0.46}  &  \num{4.382e-02}  &  \num{ 0.66}  &  \num{4.382e-02}  &  \num{ 0.66}  &  \num{4.825e-02}  &  \num{ 0.51}  &  \num{4.859e-02}  &  \num{ 0.50}  &  \num{4.792e-02}  &  \num{ 0.51}  \\
      &    80  &  \num{3.542e-02}  &  \num{ 0.43}  &  \num{3.616e-02}  &  \num{ 0.42}  &  \num{3.762e-02}  &  \num{ 0.45}  &  \num{3.378e-02}  &  \num{ 0.38}  &  \num{3.378e-02}  &  \num{ 0.38}  &  \num{3.616e-02}  &  \num{ 0.42}  &  \num{3.612e-02}  &  \num{ 0.43}  &  \num{3.641e-02}  &  \num{ 0.40}  \\
      &   160  &  \num{2.639e-02}  &  \num{ 0.42}  &  \num{2.665e-02}  &  \num{ 0.44}  &  \num{2.786e-02}  &  \num{ 0.43}  &  \num{2.552e-02}  &  \num{ 0.40}  &  \num{2.552e-02}  &  \num{ 0.40}  &  \num{2.665e-02}  &  \num{ 0.44}  &  \num{2.692e-02}  &  \num{ 0.42}  &  \num{2.703e-02}  &  \num{ 0.43}  \\
      &   320  &  \num{1.976e-02}  &  \num{ 0.42}  &  \num{1.975e-02}  &  \num{ 0.43}  &  \num{2.005e-02}  &  \num{ 0.47}  &  \num{1.907e-02}  &  \num{ 0.42}  &  \num{1.907e-02}  &  \num{ 0.42}  &  \num{1.975e-02}  &  \num{ 0.43}  &  \num{1.990e-02}  &  \num{ 0.44}  &  \num{2.029e-02}  &  \num{ 0.41}  \\
  \midrule 
   2  &    10  &  \num{6.368e-02}  &               &  \num{6.262e-02}  &               &  \num{6.278e-02}  &               &  \num{6.391e-02}  &               &  \num{6.391e-02}  &               &  \num{6.262e-02}  &               &  \num{6.308e-02}  &               &  \num{6.511e-02}  &               \\
      &    20  &  \num{3.368e-02}  &  \num{ 0.92}  &  \num{3.338e-02}  &  \num{ 0.91}  &  \num{3.353e-02}  &  \num{ 0.90}  &  \num{3.597e-02}  &  \num{ 0.83}  &  \num{3.597e-02}  &  \num{ 0.83}  &  \num{3.338e-02}  &  \num{ 0.91}  &  \num{3.323e-02}  &  \num{ 0.92}  &  \num{3.300e-02}  &  \num{ 0.98}  \\
      &    40  &  \num{2.843e-02}  &  \num{ 0.24}  &  \num{2.859e-02}  &  \num{ 0.22}  &  \num{2.881e-02}  &  \num{ 0.22}  &  \num{2.770e-02}  &  \num{ 0.38}  &  \num{2.770e-02}  &  \num{ 0.38}  &  \num{2.859e-02}  &  \num{ 0.22}  &  \num{2.852e-02}  &  \num{ 0.22}  &  \num{2.857e-02}  &  \num{ 0.21}  \\
      &    80  &  \num{2.140e-02}  &  \num{ 0.41}  &  \num{2.124e-02}  &  \num{ 0.43}  &  \num{2.140e-02}  &  \num{ 0.43}  &  \num{2.093e-02}  &  \num{ 0.40}  &  \num{2.093e-02}  &  \num{ 0.40}  &  \num{2.124e-02}  &  \num{ 0.43}  &  \num{2.126e-02}  &  \num{ 0.42}  &  \num{2.127e-02}  &  \num{ 0.43}  \\
      &   160  &  \num{1.226e-02}  &  \num{ 0.80}  &  \num{1.216e-02}  &  \num{ 0.81}  &  \num{1.222e-02}  &  \num{ 0.81}  &  \num{1.209e-02}  &  \num{ 0.79}  &  \num{1.209e-02}  &  \num{ 0.79}  &  \num{1.216e-02}  &  \num{ 0.81}  &  \num{1.222e-02}  &  \num{ 0.80}  &  \num{1.260e-02}  &  \num{ 0.76}  \\
      &   320  &  \num{6.649e-03}  &  \num{ 0.88}  &  \num{6.606e-03}  &  \num{ 0.88}  &  \num{6.657e-03}  &  \num{ 0.88}  &  \num{6.812e-03}  &  \num{ 0.83}  &  \num{6.812e-03}  &  \num{ 0.83}  &  \num{6.606e-03}  &  \num{ 0.88}  &  \num{6.603e-03}  &  \num{ 0.89}  &  \num{6.701e-03}  &  \num{ 0.91}  \\
  \midrule 
   3  &    10  &  \num{3.298e-02}  &               &  \num{3.325e-02}  &               &  \num{3.345e-02}  &               &  \num{3.694e-02}  &               &  \num{3.694e-02}  &               &  \num{3.325e-02}  &               &  \num{3.238e-02}  &               &  \num{3.175e-02}  &               \\
      &    20  &  \num{2.931e-02}  &  \num{ 0.17}  &  \num{2.913e-02}  &  \num{ 0.19}  &  \num{2.915e-02}  &  \num{ 0.20}  &  \num{2.843e-02}  &  \num{ 0.38}  &  \num{2.843e-02}  &  \num{ 0.38}  &  \num{2.913e-02}  &  \num{ 0.19}  &  \num{2.867e-02}  &  \num{ 0.18}  &  \num{2.914e-02}  &  \num{ 0.12}  \\
      &    40  &  \num{2.137e-02}  &  \num{ 0.46}  &  \num{2.098e-02}  &  \num{ 0.47}  &  \num{2.135e-02}  &  \num{ 0.45}  &  \num{2.247e-02}  &  \num{ 0.34}  &  \num{2.247e-02}  &  \num{ 0.34}  &  \num{2.098e-02}  &  \num{ 0.47}  &  \num{2.110e-02}  &  \num{ 0.44}  &  \num{2.070e-02}  &  \num{ 0.49}  \\
      &    80  &  \num{1.083e-02}  &  \num{ 0.98}  &  \num{1.090e-02}  &  \num{ 0.95}  &  \num{1.108e-02}  &  \num{ 0.95}  &  \num{1.159e-02}  &  \num{ 0.96}  &  \num{1.159e-02}  &  \num{ 0.96}  &  \num{1.090e-02}  &  \num{ 0.95}  &  \num{1.092e-02}  &  \num{ 0.95}  &  \num{1.025e-02}  &  \num{ 1.01}  \\
      &   160  &  \num{6.484e-03}  &  \num{ 0.74}  &  \num{6.452e-03}  &  \num{ 0.76}  &  \num{6.507e-03}  &  \num{ 0.77}  &  \num{7.096e-03}  &  \num{ 0.71}  &  \num{7.096e-03}  &  \num{ 0.71}  &  \num{6.452e-03}  &  \num{ 0.76}  &  \num{6.422e-03}  &  \num{ 0.77}  &  \num{6.676e-03}  &  \num{ 0.62}  \\
      &   320  &  \num{4.997e-03}  &  \num{ 0.38}  &  \num{4.971e-03}  &  \num{ 0.38}  &  \num{4.977e-03}  &  \num{ 0.39}  &  \num{5.319e-03}  &  \num{ 0.42}  &  \num{5.319e-03}  &  \num{ 0.42}  &  \num{4.971e-03}  &  \num{ 0.38}  &  \num{4.934e-03}  &  \num{ 0.38}  &  \num{5.150e-03}  &  \num{ 0.37}  \\
  \midrule 
   4  &    10  &  \num{3.978e-02}  &               &  \num{3.968e-02}  &               &  \num{4.006e-02}  &               &  \num{4.177e-02}  &               &  \num{4.177e-02}  &               &  \num{3.968e-02}  &               &  \num{3.943e-02}  &               &  \num{3.696e-02}  &               \\
      &    20  &  \num{1.320e-02}  &  \num{ 1.59}  &  \num{1.286e-02}  &  \num{ 1.63}  &  \num{1.289e-02}  &  \num{ 1.64}  &  \num{1.459e-02}  &  \num{ 1.52}  &  \num{1.459e-02}  &  \num{ 1.52}  &  \num{1.286e-02}  &  \num{ 1.63}  &  \num{1.274e-02}  &  \num{ 1.63}  &  \num{1.350e-02}  &  \num{ 1.45}  \\
      &    40  &  \num{1.058e-02}  &  \num{ 0.32}  &  \num{1.031e-02}  &  \num{ 0.32}  &  \num{1.023e-02}  &  \num{ 0.33}  &  \num{1.178e-02}  &  \num{ 0.31}  &  \num{1.178e-02}  &  \num{ 0.31}  &  \num{1.031e-02}  &  \num{ 0.32}  &  \num{1.046e-02}  &  \num{ 0.29}  &  \num{1.135e-02}  &  \num{ 0.25}  \\
      &    80  &  \num{1.303e-02}  &  \num{-0.30}  &  \num{1.266e-02}  &  \num{-0.30}  &  \num{1.281e-02}  &  \num{-0.32}  &  \num{1.267e-02}  &  \num{-0.10}  &  \num{1.267e-02}  &  \num{-0.10}  &  \num{1.266e-02}  &  \num{-0.30}  &  \num{1.263e-02}  &  \num{-0.27}  &  \num{1.269e-02}  &  \num{-0.16}  \\
      &   160  &  \num{7.097e-03}  &  \num{ 0.88}  &  \num{6.948e-03}  &  \num{ 0.87}  &  \num{6.945e-03}  &  \num{ 0.88}  &  \num{7.467e-03}  &  \num{ 0.76}  &  \num{7.467e-03}  &  \num{ 0.76}  &  \num{6.948e-03}  &  \num{ 0.87}  &  \num{6.965e-03}  &  \num{ 0.86}  &  \num{7.383e-03}  &  \num{ 0.78}  \\
      &   320  &  \num{4.591e-03}  &  \num{ 0.63}  &  \num{4.493e-03}  &  \num{ 0.63}  &  \num{4.463e-03}  &  \num{ 0.64}  &  \num{4.897e-03}  &  \num{ 0.61}  &  \num{4.897e-03}  &  \num{ 0.61}  &  \num{4.493e-03}  &  \num{ 0.63}  &  \num{4.475e-03}  &  \num{ 0.64}  &  \num{4.848e-03}  &  \num{ 0.61}  \\
  \midrule 
   5  &    10  &  \num{2.968e-02}  &               &  \num{2.872e-02}  &               &  \num{2.884e-02}  &               &  \num{2.866e-02}  &               &  \num{2.866e-02}  &               &  \num{2.872e-02}  &               &  \num{2.856e-02}  &               &  \num{2.701e-02}  &               \\
      &    20  &  \num{2.238e-02}  &  \num{ 0.41}  &  \num{2.191e-02}  &  \num{ 0.39}  &  \num{2.243e-02}  &  \num{ 0.36}  &  \num{2.466e-02}  &  \num{ 0.22}  &  \num{2.466e-02}  &  \num{ 0.22}  &  \num{2.191e-02}  &  \num{ 0.39}  &  \num{2.208e-02}  &  \num{ 0.37}  &  \num{2.049e-02}  &  \num{ 0.40}  \\
      &    40  &  \num{1.100e-02}  &  \num{ 1.02}  &  \num{1.101e-02}  &  \num{ 0.99}  &  \num{1.154e-02}  &  \num{ 0.96}  &  \num{1.277e-02}  &  \num{ 0.95}  &  \num{1.277e-02}  &  \num{ 0.95}  &  \num{1.101e-02}  &  \num{ 0.99}  &  \num{1.109e-02}  &  \num{ 0.99}  &  \num{1.029e-02}  &  \num{ 0.99}  \\
      &    80  &  \num{6.344e-03}  &  \num{ 0.79}  &  \num{6.223e-03}  &  \num{ 0.82}  &  \num{6.364e-03}  &  \num{ 0.86}  &  \num{7.393e-03}  &  \num{ 0.79}  &  \num{7.393e-03}  &  \num{ 0.79}  &  \num{6.223e-03}  &  \num{ 0.82}  &  \num{6.240e-03}  &  \num{ 0.83}  &  \num{6.484e-03}  &  \num{ 0.67}  \\
      &   160  &  \num{5.560e-03}  &  \num{ 0.19}  &  \num{5.394e-03}  &  \num{ 0.21}  &  \num{5.425e-03}  &  \num{ 0.23}  &  \num{6.103e-03}  &  \num{ 0.28}  &  \num{6.103e-03}  &  \num{ 0.28}  &  \num{5.394e-03}  &  \num{ 0.21}  &  \num{5.461e-03}  &  \num{ 0.19}  &  \num{5.728e-03}  &  \num{ 0.18}  \\
      &   320  &  \num{4.655e-03}  &  \num{ 0.26}  &  \num{4.429e-03}  &  \num{ 0.28}  &  \num{4.367e-03}  &  \num{ 0.31}  &  \num{4.920e-03}  &  \num{ 0.31}  &  \num{4.920e-03}  &  \num{ 0.31}  &  \num{4.429e-03}  &  \num{ 0.28}  &  \num{4.504e-03}  &  \num{ 0.28}  &  \num{4.960e-03}  &  \num{ 0.21}  \\
  \bottomrule
\end{tabular}
\end{center}
\end{sidewaystable}

\begin{sidewaystable}
\small
\caption{Errors and experimental order of convergence (EOC) for varying polynomial
         degrees $p$ and number of elements $N$ for the Sod shock tube problem
         \eqref{eq:sod} using the local Lax-Friedrichs numerical flux and entropy
         conserving volume fluxes.}
\label{tab:sod-LLF}
\sisetup{
  output-exponent-marker=\text{e},
  round-mode=places,
  round-precision=2
}
\begin{center}
\begin{tabular}{rr | rr|rr|rr|rr|rr|rr|rr|rr}
  \toprule
  & &  \multicolumn{2}{c}{IR}
    &  \multicolumn{2}{c}{Ch} 
    &  \multicolumn{2}{c}{$\rho,v,\beta$ (2)}
    &  \multicolumn{2}{c}{$\rho,v,\frac{1}{p}$}
    &  \multicolumn{2}{c}{$\rho,v,p$}
    &  \multicolumn{2}{c}{$\rho,v,T$ (1)}
    &  \multicolumn{2}{c}{$\rho,v,T$ (2)} 
    &  \multicolumn{2}{c}{$\rho,V,T$ (rev)}
  \\
  $p$ & $N$
    & $\norm{\mathrm{err_\rho}}_M$ & EOC
    & $\norm{\mathrm{err_\rho}}_M$ & EOC
    & $\norm{\mathrm{err_\rho}}_M$ & EOC
    & $\norm{\mathrm{err_\rho}}_M$ & EOC
    & $\norm{\mathrm{err_\rho}}_M$ & EOC
    & $\norm{\mathrm{err_\rho}}_M$ & EOC
    & $\norm{\mathrm{err_\rho}}_M$ & EOC
    & $\norm{\mathrm{err_\rho}}_M$ & EOC
  \\
  \midrule
   1  &    10  &  \num{1.019e-01}  &               &  \num{1.024e-01}  &               &  \num{1.024e-01}  &               &  \num{1.219e-01}  &               &  \num{1.219e-01}  &               &  \num{1.024e-01}  &               &  \num{1.027e-01}  &               &  \num{1.041e-01}  &               \\
      &    20  &  \num{6.942e-02}  &  \num{ 0.55}  &  \num{6.873e-02}  &  \num{ 0.58}  &  \num{6.943e-02}  &  \num{ 0.56}  &  \num{7.250e-02}  &  \num{ 0.75}  &  \num{7.250e-02}  &  \num{ 0.75}  &  \num{6.873e-02}  &  \num{ 0.58}  &  \num{6.955e-02}  &  \num{ 0.56}  &  \num{7.183e-02}  &  \num{ 0.53}  \\
      &    40  &  \num{5.128e-02}  &  \num{ 0.44}  &  \num{5.215e-02}  &  \num{ 0.40}  &  \num{5.500e-02}  &  \num{ 0.34}  &  \num{4.607e-02}  &  \num{ 0.65}  &  \num{4.607e-02}  &  \num{ 0.65}  &  \num{5.215e-02}  &  \num{ 0.40}  &  \num{5.142e-02}  &  \num{ 0.44}  &  \num{4.996e-02}  &  \num{ 0.52}  \\
      &    80  &  \num{3.745e-02}  &  \num{ 0.45}  &  \num{3.806e-02}  &  \num{ 0.45}  &  \num{3.995e-02}  &  \num{ 0.46}  &  \num{3.528e-02}  &  \num{ 0.38}  &  \num{3.528e-02}  &  \num{ 0.38}  &  \num{3.806e-02}  &  \num{ 0.45}  &  \num{3.822e-02}  &  \num{ 0.43}  &  \num{3.823e-02}  &  \num{ 0.39}  \\
      &   160  &  \num{2.828e-02}  &  \num{ 0.40}  &  \num{2.862e-02}  &  \num{ 0.41}  &  \num{3.051e-02}  &  \num{ 0.39}  &  \num{2.602e-02}  &  \num{ 0.44}  &  \num{2.602e-02}  &  \num{ 0.44}  &  \num{2.862e-02}  &  \num{ 0.41}  &  \num{2.895e-02}  &  \num{ 0.40}  &  \num{2.896e-02}  &  \num{ 0.40}  \\
      &   320  &  \num{2.104e-02}  &  \num{ 0.43}  &  \num{2.103e-02}  &  \num{ 0.44}  &  \num{2.187e-02}  &  \num{ 0.48}  &  \num{1.942e-02}  &  \num{ 0.42}  &  \num{1.942e-02}  &  \num{ 0.42}  &  \num{2.103e-02}  &  \num{ 0.44}  &  \num{2.132e-02}  &  \num{ 0.44}  &  \num{2.168e-02}  &  \num{ 0.42}  \\
  \midrule 
   2  &    10  &  \num{6.446e-02}  &               &  \num{6.301e-02}  &               &  \num{6.255e-02}  &               &  \num{6.322e-02}  &               &  \num{6.322e-02}  &               &  \num{6.301e-02}  &               &  \num{6.369e-02}  &               &  \num{6.756e-02}  &               \\
      &    20  &  \num{3.662e-02}  &  \num{ 0.82}  &  \num{3.647e-02}  &  \num{ 0.79}  &  \num{3.608e-02}  &  \num{ 0.79}  &  \num{3.854e-02}  &  \num{ 0.71}  &  \num{3.854e-02}  &  \num{ 0.71}  &  \num{3.647e-02}  &  \num{ 0.79}  &  \num{3.582e-02}  &  \num{ 0.83}  &  \num{3.520e-02}  &  \num{ 0.94}  \\
      &    40  &  \num{3.006e-02}  &  \num{ 0.29}  &  \num{3.036e-02}  &  \num{ 0.26}  &  \num{3.040e-02}  &  \num{ 0.25}  &  \num{2.922e-02}  &  \num{ 0.40}  &  \num{2.922e-02}  &  \num{ 0.40}  &  \num{3.036e-02}  &  \num{ 0.26}  &  \num{3.008e-02}  &  \num{ 0.25}  &  \num{3.005e-02}  &  \num{ 0.23}  \\
      &    80  &  \num{2.186e-02}  &  \num{ 0.46}  &  \num{2.183e-02}  &  \num{ 0.48}  &  \num{2.183e-02}  &  \num{ 0.48}  &  \num{2.187e-02}  &  \num{ 0.42}  &  \num{2.187e-02}  &  \num{ 0.42}  &  \num{2.183e-02}  &  \num{ 0.48}  &  \num{2.163e-02}  &  \num{ 0.48}  &  \num{2.163e-02}  &  \num{ 0.47}  \\
      &   160  &  \num{1.292e-02}  &  \num{ 0.76}  &  \num{1.288e-02}  &  \num{ 0.76}  &  \num{1.288e-02}  &  \num{ 0.76}  &  \num{1.265e-02}  &  \num{ 0.79}  &  \num{1.265e-02}  &  \num{ 0.79}  &  \num{1.288e-02}  &  \num{ 0.76}  &  \num{1.289e-02}  &  \num{ 0.75}  &  \num{1.331e-02}  &  \num{ 0.70}  \\
      &   320  &  \num{6.913e-03}  &  \num{ 0.90}  &  \num{6.891e-03}  &  \num{ 0.90}  &  \num{6.901e-03}  &  \num{ 0.90}  &  \num{7.122e-03}  &  \num{ 0.83}  &  \num{7.122e-03}  &  \num{ 0.83}  &  \num{6.891e-03}  &  \num{ 0.90}  &  \num{6.868e-03}  &  \num{ 0.91}  &  \num{6.977e-03}  &  \num{ 0.93}  \\
  \midrule 
   3  &    10  &  \num{3.480e-02}  &               &  \num{3.481e-02}  &               &  \num{3.504e-02}  &               &  \num{3.789e-02}  &               &  \num{3.789e-02}  &               &  \num{3.481e-02}  &               &  \num{3.409e-02}  &               &  \num{3.295e-02}  &               \\
      &    20  &  \num{3.096e-02}  &  \num{ 0.17}  &  \num{3.085e-02}  &  \num{ 0.17}  &  \num{3.074e-02}  &  \num{ 0.19}  &  \num{3.207e-02}  &  \num{ 0.24}  &  \num{3.207e-02}  &  \num{ 0.24}  &  \num{3.085e-02}  &  \num{ 0.17}  &  \num{3.015e-02}  &  \num{ 0.18}  &  \num{3.037e-02}  &  \num{ 0.12}  \\
      &    40  &  \num{2.188e-02}  &  \num{ 0.50}  &  \num{2.159e-02}  &  \num{ 0.52}  &  \num{2.165e-02}  &  \num{ 0.51}  &  \num{2.851e-02}  &  \num{ 0.17}  &  \num{2.851e-02}  &  \num{ 0.17}  &  \num{2.159e-02}  &  \num{ 0.52}  &  \num{2.127e-02}  &  \num{ 0.50}  &  \num{2.088e-02}  &  \num{ 0.54}  \\
      &    80  &  \num{1.162e-02}  &  \num{ 0.91}  &  \num{1.167e-02}  &  \num{ 0.89}  &  \num{1.159e-02}  &  \num{ 0.90}  &  \num{2.155e-02}  &  \num{ 0.40}  &  \num{2.155e-02}  &  \num{ 0.40}  &  \num{1.167e-02}  &  \num{ 0.89}  &  \num{1.142e-02}  &  \num{ 0.90}  &  \num{1.086e-02}  &  \num{ 0.94}  \\
      &   160  &  \num{7.225e-03}  &  \num{ 0.69}  &  \num{7.094e-03}  &  \num{ 0.72}  &  \num{6.929e-03}  &  \num{ 0.74}  &  \num{1.997e-02}  &  \num{ 0.11}  &  \num{1.997e-02}  &  \num{ 0.11}  &  \num{7.094e-03}  &  \num{ 0.72}  &  \num{6.872e-03}  &  \num{ 0.73}  &  \num{7.671e-03}  &  \num{ 0.50}  \\
      &   320  &  \num{5.773e-03}  &  \num{ 0.32}  &  \num{5.632e-03}  &  \num{ 0.33}  &  \num{5.497e-03}  &  \num{ 0.33}  &  \num{1.595e-02}  &  \num{ 0.32}  &  \num{1.595e-02}  &  \num{ 0.32}  &  \num{5.632e-03}  &  \num{ 0.33}  &  \num{5.523e-03}  &  \num{ 0.32}  &  \num{6.361e-03}  &  \num{ 0.27}  \\
  \midrule 
   4  &    10  &  \num{4.033e-02}  &               &  \num{4.027e-02}  &               &  \num{4.037e-02}  &               &  \num{4.239e-02}  &               &  \num{4.239e-02}  &               &  \num{4.027e-02}  &               &  \num{3.979e-02}  &               &  \num{3.789e-02}  &               \\
      &    20  &  \num{1.425e-02}  &  \num{ 1.50}  &  \num{1.402e-02}  &  \num{ 1.52}  &  \num{1.345e-02}  &  \num{ 1.59}  &  \num{1.963e-02}  &  \num{ 1.11}  &  \num{1.963e-02}  &  \num{ 1.11}  &  \num{1.402e-02}  &  \num{ 1.52}  &  \num{1.316e-02}  &  \num{ 1.60}  &  \num{1.372e-02}  &  \num{ 1.47}  \\
      &    40  &  \num{1.054e-02}  &  \num{ 0.44}  &  \num{1.012e-02}  &  \num{ 0.47}  &  \num{1.000e-02}  &  \num{ 0.43}  &  \num{2.396e-02}  &  \num{-0.29}  &  \num{2.396e-02}  &  \num{-0.29}  &  \num{1.012e-02}  &  \num{ 0.47}  &  \num{1.023e-02}  &  \num{ 0.36}  &  \num{1.163e-02}  &  \num{ 0.24}  \\
      &    80  &  \num{1.300e-02}  &  \num{-0.30}  &  \num{1.276e-02}  &  \num{-0.34}  &  \num{1.286e-02}  &  \num{-0.36}  &  \num{3.481e-02}  &  \num{-0.54}  &  \num{3.481e-02}  &  \num{-0.54}  &  \num{1.276e-02}  &  \num{-0.34}  &  \num{1.269e-02}  &  \num{-0.31}  &  \num{1.248e-02}  &  \num{-0.10}  \\
      &   160  &  \num{7.429e-03}  &  \num{ 0.81}  &  \num{7.275e-03}  &  \num{ 0.81}  &  \num{7.274e-03}  &  \num{ 0.82}  &  \num{2.583e-02}  &  \num{ 0.43}  &  \num{2.583e-02}  &  \num{ 0.43}  &  \num{7.275e-03}  &  \num{ 0.81}  &  \num{7.321e-03}  &  \num{ 0.79}  &  \num{7.866e-03}  &  \num{ 0.67}  \\
      &   320  &  \num{4.904e-03}  &  \num{ 0.60}  &  \num{4.781e-03}  &  \num{ 0.61}  &  \num{4.708e-03}  &  \num{ 0.63}  &  \num{1.767e-02}  &  \num{ 0.55}  &  \num{1.767e-02}  &  \num{ 0.55}  &  \num{4.781e-03}  &  \num{ 0.61}  &  \num{4.795e-03}  &  \num{ 0.61}  &  \num{5.486e-03}  &  \num{ 0.52}  \\
  \midrule 
   5  &    10  &  \num{3.017e-02}  &               &  \num{2.929e-02}  &               &  \num{2.935e-02}  &               &  \num{3.016e-02}  &               &  \num{3.016e-02}  &               &  \num{2.929e-02}  &               &  \num{2.914e-02}  &               &  \num{2.728e-02}  &               \\
      &    20  &  \num{2.305e-02}  &  \num{ 0.39}  &  \num{2.264e-02}  &  \num{ 0.37}  &  \num{2.265e-02}  &  \num{ 0.37}  &  \num{3.052e-02}  &  \num{-0.02}  &  \num{3.052e-02}  &  \num{-0.02}  &  \num{2.264e-02}  &  \num{ 0.37}  &  \num{2.236e-02}  &  \num{ 0.38}  &  \num{2.110e-02}  &  \num{ 0.37}  \\
      &    40  &  \num{1.195e-02}  &  \num{ 0.95}  &  \num{1.168e-02}  &  \num{ 0.96}  &  \num{1.196e-02}  &  \num{ 0.92}  &  \num{3.241e-02}  &  \num{-0.09}  &  \num{3.241e-02}  &  \num{-0.09}  &  \num{1.168e-02}  &  \num{ 0.96}  &  \num{1.188e-02}  &  \num{ 0.91}  &  \num{1.082e-02}  &  \num{ 0.96}  \\
      &    80  &  \num{7.294e-03}  &  \num{ 0.71}  &  \num{6.822e-03}  &  \num{ 0.78}  &  \num{6.781e-03}  &  \num{ 0.82}  &  \num{2.853e-02}  &  \num{ 0.18}  &  \num{2.853e-02}  &  \num{ 0.18}  &  \num{6.822e-03}  &  \num{ 0.78}  &  \num{7.084e-03}  &  \num{ 0.75}  &  \num{7.138e-03}  &  \num{ 0.60}  \\
      &   160  &  \num{6.479e-03}  &  \num{ 0.17}  &  \num{6.035e-03}  &  \num{ 0.18}  &  \num{6.009e-03}  &  \num{ 0.17}  &  \num{1.808e-02}  &  \num{ 0.66}  &  \num{1.808e-02}  &  \num{ 0.66}  &  \num{6.035e-03}  &  \num{ 0.18}  &  \num{6.343e-03}  &  \num{ 0.16}  &  \num{6.500e-03}  &  \num{ 0.14}  \\
      &   320  &  \num{5.407e-03}  &  \num{ 0.26}  &  \num{4.922e-03}  &  \num{ 0.29}  &  \num{4.820e-03}  &  \num{ 0.32}  &  \num{1.203e-02}  &  \num{ 0.59}  &  \num{1.203e-02}  &  \num{ 0.59}  &  \num{4.921e-03}  &  \num{ 0.29}  &  \num{5.224e-03}  &  \num{ 0.28}  &  \num{5.658e-03}  &  \num{ 0.20}  \\
  \bottomrule
\end{tabular}
\end{center}
\end{sidewaystable}

\begin{table}
\small
\caption{Errors and experimental order of convergence (EOC) for varying polynomial
         degrees $p$ and number of elements $N$ for the Sod shock tube problem
         \eqref{eq:sod} using the Suliciu relaxation solver as numerical flux
         and "simple" volume fluxes.}
\label{tab:sod-nonEC-suliciu}
\sisetup{
  output-exponent-marker=\text{e},
  round-mode=places,
  round-precision=2
}
\begin{center}
\begin{tabular}{rr | rr|rr|rr|rr|rr}
  \toprule
  & &  \multicolumn{2}{c}{Central}
    &  \multicolumn{2}{c}{Morinishi} 
    &  \multicolumn{2}{c}{Ducros}
    &  \multicolumn{2}{c}{KG}
    &  \multicolumn{2}{c}{Pirozzoli}
  \\
  $p$ & $N$
    & $\norm{\mathrm{err_\rho}}_M$ & EOC
    & $\norm{\mathrm{err_\rho}}_M$ & EOC
    & $\norm{\mathrm{err_\rho}}_M$ & EOC
    & $\norm{\mathrm{err_\rho}}_M$ & EOC
    & $\norm{\mathrm{err_\rho}}_M$ & EOC
  \\
  \midrule
   1  &    10  &           $*$     &               &           $*$     &               &  \num{1.043e-01}  &               &  \num{1.047e-01}  &               &  \num{1.056e-01}  &               \\
      &    20  &           $*$     &        $*$    &           $*$     &        $*$    &  \num{6.664e-02}  &  \num{ 0.65}  &  \num{6.636e-02}  &  \num{ 0.66}  &  \num{6.683e-02}  &  \num{ 0.66}  \\
      &    40  &           $*$     &        $*$    &           $*$     &        $*$    &  \num{4.199e-02}  &  \num{ 0.67}  &  \num{4.506e-02}  &  \num{ 0.56}  &  \num{4.594e-02}  &  \num{ 0.54}  \\
      &    80  &           $*$     &        $*$    &           $*$     &        $*$    &  \num{3.262e-02}  &  \num{ 0.36}  &  \num{3.476e-02}  &  \num{ 0.37}  &  \num{3.498e-02}  &  \num{ 0.39}  \\
      &   160  &           $*$     &        $*$    &           $*$     &        $*$    &  \num{2.450e-02}  &  \num{ 0.41}  &  \num{2.589e-02}  &  \num{ 0.43}  &  \num{2.584e-02}  &  \num{ 0.44}  \\
      &   320  &           $*$     &        $*$    &           $*$     &        $*$    &  \num{1.855e-02}  &  \num{ 0.40}  &  \num{1.953e-02}  &  \num{ 0.41}  &  \num{1.954e-02}  &  \num{ 0.40}  \\
  \midrule 
   2  &    10  &           $*$     &               &  \num{8.744e-02}  &               &  \num{6.382e-02}  &               &  \num{6.404e-02}  &               &  \num{6.393e-02}  &               \\
      &    20  &           $*$     &        $*$    &           $*$     &        $*$    &  \num{3.355e-02}  &  \num{ 0.93}  &  \num{3.405e-02}  &  \num{ 0.91}  &  \num{3.413e-02}  &  \num{ 0.91}  \\
      &    40  &           $*$     &        $*$    &           $*$     &        $*$    &  \num{2.556e-02}  &  \num{ 0.39}  &  \num{2.656e-02}  &  \num{ 0.36}  &  \num{2.688e-02}  &  \num{ 0.34}  \\
      &    80  &           $*$     &        $*$    &           $*$     &        $*$    &  \num{1.927e-02}  &  \num{ 0.41}  &  \num{1.997e-02}  &  \num{ 0.41}  &  \num{1.999e-02}  &  \num{ 0.43}  \\
      &   160  &           $*$     &        $*$    &           $*$     &        $*$    &  \num{1.177e-02}  &  \num{ 0.71}  &  \num{1.219e-02}  &  \num{ 0.71}  &  \num{1.238e-02}  &  \num{ 0.69}  \\
      &   320  &           $*$     &        $*$    &           $*$     &        $*$    &  \num{6.324e-03}  &  \num{ 0.90}  &  \num{6.471e-03}  &  \num{ 0.91}  &  \num{6.546e-03}  &  \num{ 0.92}  \\
  \midrule 
   3  &    10  &           $*$     &               &           $*$     &               &  \num{3.288e-02}  &               &  \num{3.300e-02}  &               &  \num{3.328e-02}  &               \\
      &    20  &           $*$     &        $*$    &           $*$     &        $*$    &  \num{2.506e-02}  &  \num{ 0.39}  &  \num{2.562e-02}  &  \num{ 0.36}  &  \num{2.600e-02}  &  \num{ 0.36}  \\
      &    40  &           $*$     &        $*$    &           $*$     &        $*$    &  \num{1.743e-02}  &  \num{ 0.52}  &  \num{1.934e-02}  &  \num{ 0.41}  &  \num{1.933e-02}  &  \num{ 0.43}  \\
      &    80  &           $*$     &        $*$    &           $*$     &        $*$    &  \num{9.934e-03}  &  \num{ 0.81}  &  \num{9.892e-03}  &  \num{ 0.97}  &  \num{9.791e-03}  &  \num{ 0.98}  \\
      &   160  &           $*$     &        $*$    &           $*$     &        $*$    &  \num{6.045e-03}  &  \num{ 0.72}  &  \num{5.973e-03}  &  \num{ 0.73}  &  \num{6.038e-03}  &  \num{ 0.70}  \\
      &   320  &           $*$     &        $*$    &           $*$     &        $*$    &  \num{5.026e-03}  &  \num{ 0.27}  &  \num{4.826e-03}  &  \num{ 0.31}  &  \num{4.850e-03}  &  \num{ 0.32}  \\
  \midrule 
   4  &    10  &           $*$     &               &           $*$     &               &  \num{3.480e-02}  &               &  \num{3.620e-02}  &               &  \num{3.627e-02}  &               \\
      &    20  &           $*$     &        $*$    &           $*$     &        $*$    &  \num{1.347e-02}  &  \num{ 1.37}  &  \num{1.260e-02}  &  \num{ 1.52}  &  \num{1.273e-02}  &  \num{ 1.51}  \\
      &    40  &           $*$     &        $*$    &           $*$     &        $*$    &  \num{1.212e-02}  &  \num{ 0.15}  &  \num{1.165e-02}  &  \num{ 0.11}  &  \num{1.168e-02}  &  \num{ 0.12}  \\
      &    80  &           $*$     &        $*$    &           $*$     &        $*$    &  \num{1.033e-02}  &  \num{ 0.23}  &  \num{1.128e-02}  &  \num{ 0.05}  &  \num{1.146e-02}  &  \num{ 0.03}  \\
      &   160  &           $*$     &        $*$    &           $*$     &        $*$    &  \num{7.134e-03}  &  \num{ 0.53}  &  \num{7.041e-03}  &  \num{ 0.68}  &  \num{7.151e-03}  &  \num{ 0.68}  \\
      &   320  &           $*$     &        $*$    &           $*$     &        $*$    &  \num{5.171e-03}  &  \num{ 0.46}  &  \num{4.767e-03}  &  \num{ 0.56}  &  \num{4.738e-03}  &  \num{ 0.59}  \\
  \midrule 
   5  &    10  &           $*$     &               &           $*$     &               &           $*$     &               &  \num{2.530e-02}  &               &  \num{2.585e-02}  &               \\
      &    20  &           $*$     &        $*$    &           $*$     &        $*$    &           $*$     &        $*$    &  \num{1.917e-02}  &  \num{ 0.40}  &  \num{1.922e-02}  &  \num{ 0.43}  \\
      &    40  &           $*$     &        $*$    &           $*$     &        $*$    &           $*$     &        $*$    &  \num{9.554e-03}  &  \num{ 1.00}  &  \num{9.396e-03}  &  \num{ 1.03}  \\
      &    80  &           $*$     &        $*$    &           $*$     &        $*$    &           $*$     &        $*$    &  \num{6.474e-03}  &  \num{ 0.56}  &  \num{6.467e-03}  &  \num{ 0.54}  \\
      &   160  &           $*$     &        $*$    &           $*$     &        $*$    &           $*$     &        $*$    &  \num{5.681e-03}  &  \num{ 0.19}  &  \num{5.680e-03}  &  \num{ 0.19}  \\
      &   320  &           $*$     &        $*$    &           $*$     &        $*$    &           $*$     &        $*$    &  \num{5.240e-03}  &  \num{ 0.12}  &  \num{5.187e-03}  &  \num{ 0.13}  \\
  \bottomrule
\end{tabular}
\end{center}
\end{table}

\begin{table}
\small
\caption{Errors and experimental order of convergence (EOC) for varying polynomial
         degrees $p$ and number of elements $N$ for the Sod shock tube problem
         \eqref{eq:sod} using the local Lax-Friedrichs numerical flux
         and "simple" volume fluxes.}
\label{tab:sod-nonEC-LLF}
\sisetup{
  output-exponent-marker=\text{e},
  round-mode=places,
  round-precision=2
}
\begin{center}
\begin{tabular}{rr | rr|rr|rr|rr|rr}
  \toprule
  & &  \multicolumn{2}{c}{Central}
    &  \multicolumn{2}{c}{Morinishi} 
    &  \multicolumn{2}{c}{Ducros}
    &  \multicolumn{2}{c}{KG}
    &  \multicolumn{2}{c}{Pirozzoli}
  \\
  $p$ & $N$
    & $\norm{\mathrm{err_\rho}}_M$ & EOC
    & $\norm{\mathrm{err_\rho}}_M$ & EOC
    & $\norm{\mathrm{err_\rho}}_M$ & EOC
    & $\norm{\mathrm{err_\rho}}_M$ & EOC
    & $\norm{\mathrm{err_\rho}}_M$ & EOC
  \\
  \midrule 
   1  &    10  &           $*$     &               &           $*$     &               &  \num{1.031e-01}  &               &  \num{1.020e-01}  &               &  \num{1.020e-01}  &               \\
      &    20  &           $*$     &        $*$    &           $*$     &        $*$    &  \num{6.744e-02}  &  \num{ 0.61}  &  \num{6.746e-02}  &  \num{ 0.60}  &  \num{6.748e-02}  &  \num{ 0.60}  \\
      &    40  &           $*$     &        $*$    &           $*$     &        $*$    &  \num{4.380e-02}  &  \num{ 0.62}  &  \num{4.817e-02}  &  \num{ 0.49}  &  \num{4.935e-02}  &  \num{ 0.45}  \\
      &    80  &           $*$     &        $*$    &           $*$     &        $*$    &  \num{3.382e-02}  &  \num{ 0.37}  &  \num{3.628e-02}  &  \num{ 0.41}  &  \num{3.678e-02}  &  \num{ 0.42}  \\
      &   160  &           $*$     &        $*$    &           $*$     &        $*$    &  \num{2.595e-02}  &  \num{ 0.38}  &  \num{2.763e-02}  &  \num{ 0.39}  &  \num{2.771e-02}  &  \num{ 0.41}  \\
      &   320  &           $*$     &        $*$    &           $*$     &        $*$    &           $*$     &        $*$    &  \num{2.091e-02}  &  \num{ 0.40}  &  \num{2.081e-02}  &  \num{ 0.41}  \\
  \midrule 
   2  &    10  &  \num{8.353e-02}  &               &  \num{5.868e-02}  &               &  \num{6.596e-02}  &               &  \num{6.537e-02}  &               &  \num{6.462e-02}  &               \\
      &    20  &           $*$     &        $*$    &           $*$     &        $*$    &  \num{3.421e-02}  &  \num{ 0.95}  &  \num{3.515e-02}  &  \num{ 0.90}  &  \num{3.531e-02}  &  \num{ 0.87}  \\
      &    40  &           $*$     &        $*$    &           $*$     &        $*$    &  \num{2.672e-02}  &  \num{ 0.36}  &  \num{2.784e-02}  &  \num{ 0.34}  &  \num{2.818e-02}  &  \num{ 0.33}  \\
      &    80  &           $*$     &        $*$    &           $*$     &        $*$    &  \num{1.932e-02}  &  \num{ 0.47}  &  \num{2.016e-02}  &  \num{ 0.47}  &  \num{2.022e-02}  &  \num{ 0.48}  \\
      &   160  &           $*$     &        $*$    &           $*$     &        $*$    &  \num{1.235e-02}  &  \num{ 0.65}  &  \num{1.273e-02}  &  \num{ 0.66}  &  \num{1.291e-02}  &  \num{ 0.65}  \\
      &   320  &           $*$     &        $*$    &           $*$     &        $*$    &  \num{6.539e-03}  &  \num{ 0.92}  &  \num{6.655e-03}  &  \num{ 0.94}  &  \num{6.734e-03}  &  \num{ 0.94}  \\
  \midrule 
   3  &    10  &           $*$     &               &           $*$     &               &  \num{3.314e-02}  &               &  \num{3.437e-02}  &               &  \num{3.501e-02}  &               \\
      &    20  &           $*$     &        $*$    &           $*$     &        $*$    &  \num{2.583e-02}  &  \num{ 0.36}  &  \num{2.715e-02}  &  \num{ 0.34}  &  \num{2.775e-02}  &  \num{ 0.34}  \\
      &    40  &           $*$     &        $*$    &           $*$     &        $*$    &  \num{1.741e-02}  &  \num{ 0.57}  &  \num{1.926e-02}  &  \num{ 0.50}  &  \num{1.932e-02}  &  \num{ 0.52}  \\
      &    80  &           $*$     &        $*$    &           $*$     &        $*$    &  \num{1.045e-02}  &  \num{ 0.74}  &  \num{1.044e-02}  &  \num{ 0.88}  &  \num{1.040e-02}  &  \num{ 0.89}  \\
      &   160  &           $*$     &        $*$    &           $*$     &        $*$    &  \num{6.912e-03}  &  \num{ 0.60}  &  \num{6.567e-03}  &  \num{ 0.67}  &  \num{6.632e-03}  &  \num{ 0.65}  \\
      &   320  &           $*$     &        $*$    &           $*$     &        $*$    &  \num{6.075e-03}  &  \num{ 0.19}  &  \num{5.575e-03}  &  \num{ 0.24}  &  \num{5.589e-03}  &  \num{ 0.25}  \\
  \midrule 
   4  &    10  &           $*$     &               &           $*$     &               &  \num{3.485e-02}  &               &  \num{3.679e-02}  &               &  \num{3.700e-02}  &               \\
      &    20  &           $*$     &        $*$    &           $*$     &        $*$    &  \num{1.270e-02}  &  \num{ 1.46}  &  \num{1.231e-02}  &  \num{ 1.58}  &  \num{1.259e-02}  &  \num{ 1.56}  \\
      &    40  &           $*$     &        $*$    &           $*$     &        $*$    &  \num{1.179e-02}  &  \num{ 0.11}  &  \num{1.129e-02}  &  \num{ 0.13}  &  \num{1.137e-02}  &  \num{ 0.15}  \\
      &    80  &           $*$     &        $*$    &           $*$     &        $*$    &  \num{1.020e-02}  &  \num{ 0.21}  &  \num{1.126e-02}  &  \num{ 0.00}  &  \num{1.145e-02}  &  \num{-0.01}  \\
      &   160  &           $*$     &        $*$    &           $*$     &        $*$    &  \num{7.428e-03}  &  \num{ 0.46}  &  \num{7.490e-03}  &  \num{ 0.59}  &  \num{7.572e-03}  &  \num{ 0.60}  \\
      &   320  &           $*$     &        $*$    &           $*$     &        $*$    &  \num{5.469e-03}  &  \num{ 0.44}  &  \num{5.219e-03}  &  \num{ 0.52}  &  \num{5.150e-03}  &  \num{ 0.56}  \\
  \midrule 
   5  &    10  &           $*$     &               &           $*$     &               &           $*$     &               &  \num{2.567e-02}  &               &  \num{2.651e-02}  &               \\
      &    20  &           $*$     &        $*$    &           $*$     &        $*$    &           $*$     &        $*$    &  \num{1.934e-02}  &  \num{ 0.41}  &  \num{1.933e-02}  &  \num{ 0.46}  \\
      &    40  &           $*$     &        $*$    &           $*$     &        $*$    &           $*$     &        $*$    &  \num{1.114e-02}  &  \num{ 0.80}  &  \num{1.083e-02}  &  \num{ 0.84}  \\
      &    80  &           $*$     &        $*$    &           $*$     &        $*$    &           $*$     &        $*$    &  \num{7.682e-03}  &  \num{ 0.54}  &  \num{7.525e-03}  &  \num{ 0.52}  \\
      &   160  &           $*$     &        $*$    &           $*$     &        $*$    &           $*$     &        $*$    &  \num{6.941e-03}  &  \num{ 0.15}  &  \num{6.821e-03}  &  \num{ 0.14}  \\
      &   320  &           $*$     &        $*$    &           $*$     &        $*$    &           $*$     &        $*$    &  \num{6.348e-03}  &  \num{ 0.13}  &  \num{6.184e-03}  &  \num{ 0.14}  \\
  \bottomrule
\end{tabular}
\end{center}
\end{table}

\clearpage
\subsection{Modified Version of Sod's Shock Tube: Subcell Flux Differencing}
\label{sec:modsod-flux-diff}

In this section, the modified version of the shock tube of \citet{sod1978survey}
as described by \citet[Section 6.4, Test 1]{toro2009riemann} will be used to 
test the semidiscretisations \eqref{eq:semidiscretisation} using the volume
terms \eqref{eq:volume-terms} and the surface terms \eqref{eq:surface-terms}.
The initial condition is given in primitive variables by
\begin{equation}
\label{eq:modsod}
\begin{aligned}
  \rho_0(x) =
  \begin{cases}
    1,    & x < \frac{3}{10}, \\
    0.125 & \text{ else},
  \end{cases}
  ,\qquad
  v_0(x) = 
  \begin{cases}
    0.75, & x < \frac{3}{10}, \\
    0     & \text{ else},
  \end{cases}
  ,\qquad
  p_0(x) =
  \begin{cases}
    1,   & x < \frac{3}{10}, \\
    0.1, & \text{ else},
  \end{cases}
\end{aligned}
\end{equation}
and the conservative variables are again computed via $\rho v_0 = \rho_0 v_0$ and
$\rho e_0 = \frac{1}{2} \rho_0 v_0^2 + \frac{p_0}{\gamma-1}$.
The solution is computed on the domain $[0, 1]$ from $t = 0$ until $t = 0.2$
using \num{15000} steps and the error is computed as in section \ref{sec:sod-flux-diff}.

The results using the Suliciu relaxation solver and the local Lax-Friedrichs flux
for varying polynomial degrees $p$ and number of elements $N$ are shown in
Table \ref{tab:modsod-suliciu} and \ref{tab:modsod-LLF}, respectively.
Again, there are some variances across the volume fluxes of up to \num{33}{\%}
[e.g. $p=4$, $N=10$, $(\rho,v,p)$ vs. $(\rho,v,\beta (2))$] but no flux seems
to be clearly superior.
As in the previous test case in section \ref{sec:sod-flux-diff}, the Suliciu
solver performs better than the LLF flux -- at least, if the resolution is good
enough.

As in the previous section \ref{sec:sod-flux-diff}, the volume fluxes recovering
the central form as well as the split forms of \citet{morinishi2010skew},
\citet{ducros2000high}, \citet{kennedy2008reduced}, and \citet{pirozzoli2011numerical}
have been used. The relevant results are shown in the Tables
\ref{tab:modsod-nonEC-suliciu} (Suliciu) and \ref{tab:modsod-nonEC-LLF} (LLF).

Contrary to the results of the unmodified shock tube of Sod, all calculations
are stable for low resolution. The splitting of \citet{morinishi2010skew} blows
up at first if the Suliciu solver is used, while the central flux crashes at first
for the LLF flux. Moreover, even the splittings of \citet{kennedy2008reduced,
pirozzoli2011numerical} that remained stable in the previous section
\ref{sec:sod-flux-diff} blow up for higher polynomial degrees.

\begin{sidewaystable}
\small
\caption{Errors and experimental order of convergence (EOC) for varying polynomial
         degrees $p$ and number of elements $N$ for the modified Sod shock tube
         problem \eqref{eq:modsod} using the Suliciu relaxation solver as numerical
         flux and entropy conservative volume fluxes.}
\label{tab:modsod-suliciu}
\sisetup{
  output-exponent-marker=\text{e},
  round-mode=places,
  round-precision=2
}
\begin{center}
\begin{tabular}{rr | rr|rr|rr|rr|rr|rr|rr|rr}
  \toprule
  & &  \multicolumn{2}{c}{IR}
    &  \multicolumn{2}{c}{Ch} 
    &  \multicolumn{2}{c}{$\rho,v,\beta$ (2)}
    &  \multicolumn{2}{c}{$\rho,v,\frac{1}{p}$}
    &  \multicolumn{2}{c}{$\rho,v,p$}
    &  \multicolumn{2}{c}{$\rho,v,T$ (1)}
    &  \multicolumn{2}{c}{$\rho,v,T$ (2)} 
    &  \multicolumn{2}{c}{$\rho,V,T$ (rev)}
  \\
  $p$ & $N$
    & $\norm{\mathrm{err_\rho}}_M$ & EOC
    & $\norm{\mathrm{err_\rho}}_M$ & EOC
    & $\norm{\mathrm{err_\rho}}_M$ & EOC
    & $\norm{\mathrm{err_\rho}}_M$ & EOC
    & $\norm{\mathrm{err_\rho}}_M$ & EOC
    & $\norm{\mathrm{err_\rho}}_M$ & EOC
    & $\norm{\mathrm{err_\rho}}_M$ & EOC
    & $\norm{\mathrm{err_\rho}}_M$ & EOC
  \\
  \midrule
   1  &    10  &  \num{1.397e-01}  &               &  \num{1.421e-01}  &               &  \num{1.421e-01}  &               &  \num{1.434e-01}  &               &  \num{1.434e-01}  &               &  \num{1.421e-01}  &               &  \num{1.424e-01}  &               &  \num{1.418e-01}  &               \\
      &    20  &  \num{7.930e-02}  &  \num{ 0.82}  &  \num{7.737e-02}  &  \num{ 0.88}  &  \num{7.848e-02}  &  \num{ 0.86}  &  \num{8.823e-02}  &  \num{ 0.70}  &  \num{8.823e-02}  &  \num{ 0.70}  &  \num{7.737e-02}  &  \num{ 0.88}  &  \num{7.665e-02}  &  \num{ 0.89}  &  \num{7.549e-02}  &  \num{ 0.91}  \\
      &    40  &  \num{6.455e-02}  &  \num{ 0.30}  &  \num{6.486e-02}  &  \num{ 0.25}  &  \num{6.692e-02}  &  \num{ 0.23}  &  \num{6.126e-02}  &  \num{ 0.53}  &  \num{6.126e-02}  &  \num{ 0.53}  &  \num{6.486e-02}  &  \num{ 0.25}  &  \num{6.435e-02}  &  \num{ 0.25}  &  \num{6.486e-02}  &  \num{ 0.22}  \\
      &    80  &  \num{5.195e-02}  &  \num{ 0.31}  &  \num{5.241e-02}  &  \num{ 0.31}  &  \num{5.418e-02}  &  \num{ 0.30}  &  \num{5.271e-02}  &  \num{ 0.22}  &  \num{5.271e-02}  &  \num{ 0.22}  &  \num{5.241e-02}  &  \num{ 0.31}  &  \num{5.200e-02}  &  \num{ 0.31}  &  \num{5.202e-02}  &  \num{ 0.32}  \\
      &   160  &  \num{3.634e-02}  &  \num{ 0.52}  &  \num{3.672e-02}  &  \num{ 0.51}  &  \num{3.801e-02}  &  \num{ 0.51}  &  \num{3.585e-02}  &  \num{ 0.56}  &  \num{3.585e-02}  &  \num{ 0.56}  &  \num{3.672e-02}  &  \num{ 0.51}  &  \num{3.655e-02}  &  \num{ 0.51}  &  \num{3.653e-02}  &  \num{ 0.51}  \\
      &   320  &  \num{2.471e-02}  &  \num{ 0.56}  &  \num{2.484e-02}  &  \num{ 0.56}  &  \num{2.558e-02}  &  \num{ 0.57}  &  \num{2.433e-02}  &  \num{ 0.56}  &  \num{2.433e-02}  &  \num{ 0.56}  &  \num{2.484e-02}  &  \num{ 0.56}  &  \num{2.483e-02}  &  \num{ 0.56}  &  \num{2.511e-02}  &  \num{ 0.54}  \\
  \midrule 
   2  &    10  &  \num{8.308e-02}  &               &  \num{8.282e-02}  &               &  \num{8.312e-02}  &               &  \num{8.734e-02}  &               &  \num{8.734e-02}  &               &  \num{8.282e-02}  &               &  \num{8.181e-02}  &               &  \num{8.236e-02}  &               \\
      &    20  &  \num{4.893e-02}  &  \num{ 0.76}  &  \num{4.839e-02}  &  \num{ 0.78}  &  \num{4.802e-02}  &  \num{ 0.79}  &  \num{5.270e-02}  &  \num{ 0.73}  &  \num{5.270e-02}  &  \num{ 0.73}  &  \num{4.839e-02}  &  \num{ 0.78}  &  \num{4.764e-02}  &  \num{ 0.78}  &  \num{5.045e-02}  &  \num{ 0.71}  \\
      &    40  &  \num{4.841e-02}  &  \num{ 0.02}  &  \num{4.847e-02}  &  \num{-0.00}  &  \num{4.868e-02}  &  \num{-0.02}  &  \num{4.527e-02}  &  \num{ 0.22}  &  \num{4.527e-02}  &  \num{ 0.22}  &  \num{4.847e-02}  &  \num{-0.00}  &  \num{4.833e-02}  &  \num{-0.02}  &  \num{4.774e-02}  &  \num{ 0.08}  \\
      &    80  &  \num{2.172e-02}  &  \num{ 1.16}  &  \num{2.168e-02}  &  \num{ 1.16}  &  \num{2.177e-02}  &  \num{ 1.16}  &  \num{2.457e-02}  &  \num{ 0.88}  &  \num{2.457e-02}  &  \num{ 0.88}  &  \num{2.168e-02}  &  \num{ 1.16}  &  \num{2.131e-02}  &  \num{ 1.18}  &  \num{2.223e-02}  &  \num{ 1.10}  \\
      &   160  &  \num{1.881e-02}  &  \num{ 0.21}  &  \num{1.880e-02}  &  \num{ 0.21}  &  \num{1.860e-02}  &  \num{ 0.23}  &  \num{2.145e-02}  &  \num{ 0.20}  &  \num{2.145e-02}  &  \num{ 0.20}  &  \num{1.880e-02}  &  \num{ 0.21}  &  \num{1.868e-02}  &  \num{ 0.19}  &  \num{1.943e-02}  &  \num{ 0.19}  \\
      &   320  &  \num{1.547e-02}  &  \num{ 0.28}  &  \num{1.565e-02}  &  \num{ 0.26}  &  \num{1.584e-02}  &  \num{ 0.23}  &  \num{1.572e-02}  &  \num{ 0.45}  &  \num{1.572e-02}  &  \num{ 0.45}  &  \num{1.565e-02}  &  \num{ 0.26}  &  \num{1.558e-02}  &  \num{ 0.26}  &  \num{1.555e-02}  &  \num{ 0.32}  \\
  \midrule 
   3  &    10  &  \num{5.762e-02}  &               &  \num{5.531e-02}  &               &  \num{5.210e-02}  &               &  \num{6.956e-02}  &               &  \num{6.956e-02}  &               &  \num{5.531e-02}  &               &  \num{5.404e-02}  &               &  \num{5.960e-02}  &               \\
      &    20  &  \num{5.415e-02}  &  \num{ 0.09}  &  \num{5.258e-02}  &  \num{ 0.07}  &  \num{5.180e-02}  &  \num{ 0.01}  &  \num{6.150e-02}  &  \num{ 0.18}  &  \num{6.150e-02}  &  \num{ 0.18}  &  \num{5.258e-02}  &  \num{ 0.07}  &  \num{5.153e-02}  &  \num{ 0.07}  &  \num{5.236e-02}  &  \num{ 0.19}  \\
      &    40  &  \num{2.454e-02}  &  \num{ 1.14}  &  \num{2.289e-02}  &  \num{ 1.20}  &  \num{2.220e-02}  &  \num{ 1.22}  &  \num{3.975e-02}  &  \num{ 0.63}  &  \num{3.975e-02}  &  \num{ 0.63}  &  \num{2.289e-02}  &  \num{ 1.20}  &  \num{2.175e-02}  &  \num{ 1.24}  &  \num{2.770e-02}  &  \num{ 0.92}  \\
      &    80  &  \num{2.063e-02}  &  \num{ 0.25}  &  \num{1.935e-02}  &  \num{ 0.24}  &  \num{1.873e-02}  &  \num{ 0.24}  &  \num{3.168e-02}  &  \num{ 0.33}  &  \num{3.168e-02}  &  \num{ 0.33}  &  \num{1.935e-02}  &  \num{ 0.24}  &  \num{1.872e-02}  &  \num{ 0.22}  &  \num{2.387e-02}  &  \num{ 0.21}  \\
      &   160  &  \num{1.598e-02}  &  \num{ 0.37}  &  \num{1.616e-02}  &  \num{ 0.26}  &  \num{1.667e-02}  &  \num{ 0.17}  &  \num{2.772e-02}  &  \num{ 0.19}  &  \num{2.772e-02}  &  \num{ 0.19}  &  \num{1.616e-02}  &  \num{ 0.26}  &  \num{1.640e-02}  &  \num{ 0.19}  &  \num{1.558e-02}  &  \num{ 0.62}  \\
      &   320  &  \num{1.254e-02}  &  \num{ 0.35}  &  \num{1.223e-02}  &  \num{ 0.40}  &  \num{1.246e-02}  &  \num{ 0.42}  &  \num{2.353e-02}  &  \num{ 0.24}  &  \num{2.353e-02}  &  \num{ 0.24}  &  \num{1.223e-02}  &  \num{ 0.40}  &  \num{1.180e-02}  &  \num{ 0.48}  &  \num{1.434e-02}  &  \num{ 0.12}  \\
  \midrule 
   4  &    10  &  \num{4.216e-02}  &               &  \num{3.885e-02}  &               &  \num{3.787e-02}  &               &  \num{5.647e-02}  &               &  \num{5.647e-02}  &               &  \num{3.885e-02}  &               &  \num{3.820e-02}  &               &  \num{4.112e-02}  &               \\
      &    20  &  \num{3.597e-02}  &  \num{ 0.23}  &  \num{3.361e-02}  &  \num{ 0.21}  &  \num{3.421e-02}  &  \num{ 0.15}  &  \num{3.965e-02}  &  \num{ 0.51}  &  \num{3.965e-02}  &  \num{ 0.51}  &  \num{3.361e-02}  &  \num{ 0.21}  &  \num{3.347e-02}  &  \num{ 0.19}  &  \num{3.079e-02}  &  \num{ 0.42}  \\
      &    40  &  \num{2.741e-02}  &  \num{ 0.39}  &  \num{2.664e-02}  &  \num{ 0.33}  &  \num{2.714e-02}  &  \num{ 0.33}  &  \num{3.032e-02}  &  \num{ 0.39}  &  \num{3.032e-02}  &  \num{ 0.39}  &  \num{2.664e-02}  &  \num{ 0.33}  &  \num{2.635e-02}  &  \num{ 0.35}  &  \num{2.388e-02}  &  \num{ 0.37}  \\
      &    80  &  \num{2.330e-02}  &  \num{ 0.23}  &  \num{2.285e-02}  &  \num{ 0.22}  &  \num{2.332e-02}  &  \num{ 0.22}  &  \num{2.347e-02}  &  \num{ 0.37}  &  \num{2.347e-02}  &  \num{ 0.37}  &  \num{2.285e-02}  &  \num{ 0.22}  &  \num{2.273e-02}  &  \num{ 0.21}  &  \num{2.075e-02}  &  \num{ 0.20}  \\
      &   160  &  \num{1.533e-02}  &  \num{ 0.60}  &  \num{1.397e-02}  &  \num{ 0.71}  &  \num{1.355e-02}  &  \num{ 0.78}  &  \num{1.488e-02}  &  \num{ 0.66}  &  \num{1.488e-02}  &  \num{ 0.66}  &  \num{1.397e-02}  &  \num{ 0.71}  &  \num{1.312e-02}  &  \num{ 0.79}  &  \num{1.662e-02}  &  \num{ 0.32}  \\
      &   320  &  \num{1.274e-02}  &  \num{ 0.27}  &  \num{1.247e-02}  &  \num{ 0.16}  &  \num{1.260e-02}  &  \num{ 0.10}  &  \num{1.594e-02}  &  \num{-0.10}  &  \num{1.594e-02}  &  \num{-0.10}  &  \num{1.247e-02}  &  \num{ 0.16}  &  \num{1.204e-02}  &  \num{ 0.12}  &  \num{1.291e-02}  &  \num{ 0.36}  \\
  \midrule 
   5  &    10  &  \num{5.947e-02}  &               &  \num{5.845e-02}  &               &  \num{5.912e-02}  &               &  \num{6.268e-02}  &               &  \num{6.268e-02}  &               &  \num{5.845e-02}  &               &  \num{5.618e-02}  &               &  \num{5.569e-02}  &               \\
      &    20  &  \num{2.977e-02}  &  \num{ 1.00}  &  \num{2.840e-02}  &  \num{ 1.04}  &  \num{2.850e-02}  &  \num{ 1.05}  &  \num{3.348e-02}  &  \num{ 0.90}  &  \num{3.348e-02}  &  \num{ 0.90}  &  \num{2.840e-02}  &  \num{ 1.04}  &  \num{2.712e-02}  &  \num{ 1.05}  &  \num{2.792e-02}  &  \num{ 1.00}  \\
      &    40  &  \num{2.618e-02}  &  \num{ 0.19}  &  \num{2.240e-02}  &  \num{ 0.34}  &  \num{2.327e-02}  &  \num{ 0.29}  &  \num{2.666e-02}  &  \num{ 0.33}  &  \num{2.666e-02}  &  \num{ 0.33}  &  \num{2.240e-02}  &  \num{ 0.34}  &  \num{2.409e-02}  &  \num{ 0.17}  &  \num{2.517e-02}  &  \num{ 0.15}  \\
      &    80  &  \num{2.081e-02}  &  \num{ 0.33}  &  \num{2.030e-02}  &  \num{ 0.14}  &  \num{2.124e-02}  &  \num{ 0.13}  &  \num{2.167e-02}  &  \num{ 0.30}  &  \num{2.167e-02}  &  \num{ 0.30}  &  \num{2.030e-02}  &  \num{ 0.14}  &  \num{2.039e-02}  &  \num{ 0.24}  &  \num{1.925e-02}  &  \num{ 0.39}  \\
      &   160  &  \num{1.655e-02}  &  \num{ 0.33}  &  \num{1.609e-02}  &  \num{ 0.34}  &  \num{1.724e-02}  &  \num{ 0.30}  &  \num{1.851e-02}  &  \num{ 0.23}  &  \num{1.851e-02}  &  \num{ 0.23}  &  \num{1.609e-02}  &  \num{ 0.34}  &  \num{1.662e-02}  &  \num{ 0.29}  &  \num{1.520e-02}  &  \num{ 0.34}  \\
      &   320  &  \num{1.451e-02}  &  \num{ 0.19}  &  \num{1.432e-02}  &  \num{ 0.17}  &  \num{1.603e-02}  &  \num{ 0.10}  &  \num{1.956e-02}  &  \num{-0.08}  &  \num{1.956e-02}  &  \num{-0.08}  &  \num{1.432e-02}  &  \num{ 0.17}  &  \num{1.467e-02}  &  \num{ 0.18}  &  \num{1.332e-02}  &  \num{ 0.19}  \\
  \bottomrule
\end{tabular}
\end{center}
\end{sidewaystable}

\begin{sidewaystable}
\small
\caption{Errors and experimental order of convergence (EOC) for varying polynomial
         degrees $p$ and number of elements $N$ for the modified Sod shock tube
         problem \eqref{eq:modsod} using the local Lax-Friedrichs numerical
         flux and entropy conservative volume fluxes.}
\label{tab:modsod-LLF}
\sisetup{
  output-exponent-marker=\text{e},
  round-mode=places,
  round-precision=2
}
\begin{center}
\begin{tabular}{rr | rr|rr|rr|rr|rr|rr|rr|rr}
  \toprule
  & &  \multicolumn{2}{c}{IR}
    &  \multicolumn{2}{c}{Ch} 
    &  \multicolumn{2}{c}{$\rho,v,\beta$ (2)}
    &  \multicolumn{2}{c}{$\rho,v,\frac{1}{p}$}
    &  \multicolumn{2}{c}{$\rho,v,p$}
    &  \multicolumn{2}{c}{$\rho,v,T$ (1)}
    &  \multicolumn{2}{c}{$\rho,v,T$ (2)} 
    &  \multicolumn{2}{c}{$\rho,V,T$ (rev)}
  \\
  $p$ & $N$
    & $\norm{\mathrm{err_\rho}}_M$ & EOC
    & $\norm{\mathrm{err_\rho}}_M$ & EOC
    & $\norm{\mathrm{err_\rho}}_M$ & EOC
    & $\norm{\mathrm{err_\rho}}_M$ & EOC
    & $\norm{\mathrm{err_\rho}}_M$ & EOC
    & $\norm{\mathrm{err_\rho}}_M$ & EOC
    & $\norm{\mathrm{err_\rho}}_M$ & EOC
    & $\norm{\mathrm{err_\rho}}_M$ & EOC
  \\
  \midrule 
   1  &    10  &  \num{1.421e-01}  &               &  \num{1.443e-01}  &               &  \num{1.442e-01}  &               &  \num{1.389e-01}  &               &  \num{1.389e-01}  &               &  \num{1.443e-01}  &               &  \num{1.445e-01}  &               &  \num{1.446e-01}  &               \\
      &    20  &  \num{7.614e-02}  &  \num{ 0.90}  &  \num{7.471e-02}  &  \num{ 0.95}  &  \num{7.531e-02}  &  \num{ 0.94}  &  \num{9.050e-02}  &  \num{ 0.62}  &  \num{9.050e-02}  &  \num{ 0.62}  &  \num{7.471e-02}  &  \num{ 0.95}  &  \num{7.439e-02}  &  \num{ 0.96}  &  \num{7.370e-02}  &  \num{ 0.97}  \\
      &    40  &  \num{6.673e-02}  &  \num{ 0.19}  &  \num{6.689e-02}  &  \num{ 0.16}  &  \num{6.874e-02}  &  \num{ 0.13}  &  \num{6.233e-02}  &  \num{ 0.54}  &  \num{6.233e-02}  &  \num{ 0.54}  &  \num{6.689e-02}  &  \num{ 0.16}  &  \num{6.654e-02}  &  \num{ 0.16}  &  \num{6.677e-02}  &  \num{ 0.14}  \\
      &    80  &  \num{5.380e-02}  &  \num{ 0.31}  &  \num{5.423e-02}  &  \num{ 0.30}  &  \num{5.634e-02}  &  \num{ 0.29}  &  \num{5.346e-02}  &  \num{ 0.22}  &  \num{5.346e-02}  &  \num{ 0.22}  &  \num{5.423e-02}  &  \num{ 0.30}  &  \num{5.405e-02}  &  \num{ 0.30}  &  \num{5.401e-02}  &  \num{ 0.31}  \\
      &   160  &  \num{3.781e-02}  &  \num{ 0.51}  &  \num{3.829e-02}  &  \num{ 0.50}  &  \num{3.970e-02}  &  \num{ 0.51}  &  \num{3.658e-02}  &  \num{ 0.55}  &  \num{3.658e-02}  &  \num{ 0.55}  &  \num{3.829e-02}  &  \num{ 0.50}  &  \num{3.817e-02}  &  \num{ 0.50}  &  \num{3.804e-02}  &  \num{ 0.51}  \\
      &   320  &  \num{2.582e-02}  &  \num{ 0.55}  &  \num{2.597e-02}  &  \num{ 0.56}  &  \num{2.676e-02}  &  \num{ 0.57}  &  \num{2.493e-02}  &  \num{ 0.55}  &  \num{2.493e-02}  &  \num{ 0.55}  &  \num{2.597e-02}  &  \num{ 0.56}  &  \num{2.604e-02}  &  \num{ 0.55}  &  \num{2.628e-02}  &  \num{ 0.53}  \\
  \midrule 
   2  &    10  &  \num{8.771e-02}  &               &  \num{8.720e-02}  &               &  \num{8.740e-02}  &               &  \num{9.151e-02}  &               &  \num{9.151e-02}  &               &  \num{8.720e-02}  &               &  \num{8.634e-02}  &               &  \num{8.622e-02}  &               \\
      &    20  &  \num{5.260e-02}  &  \num{ 0.74}  &  \num{5.218e-02}  &  \num{ 0.74}  &  \num{5.183e-02}  &  \num{ 0.75}  &  \num{5.369e-02}  &  \num{ 0.77}  &  \num{5.369e-02}  &  \num{ 0.77}  &  \num{5.218e-02}  &  \num{ 0.74}  &  \num{5.205e-02}  &  \num{ 0.73}  &  \num{5.431e-02}  &  \num{ 0.67}  \\
      &    40  &  \num{4.829e-02}  &  \num{ 0.12}  &  \num{4.876e-02}  &  \num{ 0.10}  &  \num{4.899e-02}  &  \num{ 0.08}  &  \num{4.772e-02}  &  \num{ 0.17}  &  \num{4.772e-02}  &  \num{ 0.17}  &  \num{4.876e-02}  &  \num{ 0.10}  &  \num{4.839e-02}  &  \num{ 0.11}  &  \num{4.763e-02}  &  \num{ 0.19}  \\
      &    80  &  \num{2.255e-02}  &  \num{ 1.10}  &  \num{2.312e-02}  &  \num{ 1.08}  &  \num{2.300e-02}  &  \num{ 1.09}  &  \num{2.531e-02}  &  \num{ 0.92}  &  \num{2.531e-02}  &  \num{ 0.92}  &  \num{2.312e-02}  &  \num{ 1.08}  &  \num{2.277e-02}  &  \num{ 1.09}  &  \num{2.381e-02}  &  \num{ 1.00}  \\
      &   160  &  \num{1.946e-02}  &  \num{ 0.21}  &  \num{2.011e-02}  &  \num{ 0.20}  &  \num{1.964e-02}  &  \num{ 0.23}  &  \num{2.122e-02}  &  \num{ 0.25}  &  \num{2.122e-02}  &  \num{ 0.25}  &  \num{2.011e-02}  &  \num{ 0.20}  &  \num{2.007e-02}  &  \num{ 0.18}  &  \num{2.102e-02}  &  \num{ 0.18}  \\
      &   320  &  \num{1.556e-02}  &  \num{ 0.32}  &  \num{1.614e-02}  &  \num{ 0.32}  &  \num{1.625e-02}  &  \num{ 0.27}  &  \num{1.431e-02}  &  \num{ 0.57}  &  \num{1.431e-02}  &  \num{ 0.57}  &  \num{1.614e-02}  &  \num{ 0.32}  &  \num{1.598e-02}  &  \num{ 0.33}  &  \num{1.609e-02}  &  \num{ 0.39}  \\
  \midrule 
   3  &    10  &  \num{5.778e-02}  &               &  \num{5.574e-02}  &               &  \num{5.312e-02}  &               &  \num{6.406e-02}  &               &  \num{6.406e-02}  &               &  \num{5.574e-02}  &               &  \num{5.571e-02}  &               &  \num{5.963e-02}  &               \\
      &    20  &  \num{5.408e-02}  &  \num{ 0.10}  &  \num{5.328e-02}  &  \num{ 0.07}  &  \num{5.292e-02}  &  \num{ 0.01}  &  \num{6.006e-02}  &  \num{ 0.09}  &  \num{6.006e-02}  &  \num{ 0.09}  &  \num{5.328e-02}  &  \num{ 0.07}  &  \num{5.185e-02}  &  \num{ 0.10}  &  \num{5.244e-02}  &  \num{ 0.19}  \\
      &    40  &  \num{2.496e-02}  &  \num{ 1.12}  &  \num{2.390e-02}  &  \num{ 1.16}  &  \num{2.336e-02}  &  \num{ 1.18}  &  \num{3.639e-02}  &  \num{ 0.72}  &  \num{3.639e-02}  &  \num{ 0.72}  &  \num{2.390e-02}  &  \num{ 1.16}  &  \num{2.284e-02}  &  \num{ 1.18}  &  \num{2.699e-02}  &  \num{ 0.96}  \\
      &    80  &  \num{2.101e-02}  &  \num{ 0.25}  &  \num{2.021e-02}  &  \num{ 0.24}  &  \num{1.944e-02}  &  \num{ 0.26}  &  \num{2.883e-02}  &  \num{ 0.34}  &  \num{2.883e-02}  &  \num{ 0.34}  &  \num{2.021e-02}  &  \num{ 0.24}  &  \num{1.952e-02}  &  \num{ 0.23}  &  \num{2.367e-02}  &  \num{ 0.19}  \\
      &   160  &  \num{1.730e-02}  &  \num{ 0.28}  &  \num{1.724e-02}  &  \num{ 0.23}  &  \num{1.730e-02}  &  \num{ 0.17}  &  \num{2.270e-02}  &  \num{ 0.34}  &  \num{2.270e-02}  &  \num{ 0.34}  &  \num{1.724e-02}  &  \num{ 0.23}  &  \num{1.706e-02}  &  \num{ 0.19}  &  \num{1.685e-02}  &  \num{ 0.49}  \\
      &   320  &  \num{1.243e-02}  &  \num{ 0.48}  &  \num{1.247e-02}  &  \num{ 0.47}  &  \num{1.275e-02}  &  \num{ 0.44}  &  \num{2.302e-02}  &  \num{-0.02}  &  \num{2.302e-02}  &  \num{-0.02}  &  \num{1.247e-02}  &  \num{ 0.47}  &  \num{1.194e-02}  &  \num{ 0.51}  &  \num{1.369e-02}  &  \num{ 0.30}  \\
  \midrule 
   4  &    10  &  \num{3.845e-02}  &               &  \num{3.709e-02}  &               &  \num{3.651e-02}  &               &  \num{5.283e-02}  &               &  \num{5.283e-02}  &               &  \num{3.709e-02}  &               &  \num{3.543e-02}  &               &  \num{3.981e-02}  &               \\
      &    20  &  \num{3.484e-02}  &  \num{ 0.14}  &  \num{3.368e-02}  &  \num{ 0.14}  &  \num{3.428e-02}  &  \num{ 0.09}  &  \num{3.732e-02}  &  \num{ 0.50}  &  \num{3.732e-02}  &  \num{ 0.50}  &  \num{3.368e-02}  &  \num{ 0.14}  &  \num{3.407e-02}  &  \num{ 0.06}  &  \num{2.985e-02}  &  \num{ 0.42}  \\
      &    40  &  \num{2.688e-02}  &  \num{ 0.37}  &  \num{2.585e-02}  &  \num{ 0.38}  &  \num{2.688e-02}  &  \num{ 0.35}  &  \num{2.981e-02}  &  \num{ 0.32}  &  \num{2.981e-02}  &  \num{ 0.32}  &  \num{2.585e-02}  &  \num{ 0.38}  &  \num{2.654e-02}  &  \num{ 0.36}  &  \num{2.419e-02}  &  \num{ 0.30}  \\
      &    80  &  \num{2.442e-02}  &  \num{ 0.14}  &  \num{2.379e-02}  &  \num{ 0.12}  &  \num{2.396e-02}  &  \num{ 0.17}  &  \num{2.532e-02}  &  \num{ 0.24}  &  \num{2.532e-02}  &  \num{ 0.24}  &  \num{2.379e-02}  &  \num{ 0.12}  &  \num{2.421e-02}  &  \num{ 0.13}  &  \num{2.396e-02}  &  \num{ 0.01}  \\
      &   160  &  \num{1.392e-02}  &  \num{ 0.81}  &  \num{1.349e-02}  &  \num{ 0.82}  &  \num{1.372e-02}  &  \num{ 0.80}  &  \num{1.659e-02}  &  \num{ 0.61}  &  \num{1.659e-02}  &  \num{ 0.61}  &  \num{1.349e-02}  &  \num{ 0.82}  &  \num{1.422e-02}  &  \num{ 0.77}  &  \num{1.638e-02}  &  \num{ 0.55}  \\
      &   320  &  \num{1.403e-02}  &  \num{-0.01}  &  \num{1.365e-02}  &  \num{-0.02}  &  \num{1.401e-02}  &  \num{-0.03}  &  \num{1.939e-02}  &  \num{-0.22}  &  \num{1.939e-02}  &  \num{-0.22}  &  \num{1.365e-02}  &  \num{-0.02}  &  \num{1.395e-02}  &  \num{ 0.03}  &  \num{1.496e-02}  &  \num{ 0.13}  \\
  \midrule 
   5  &    10  &  \num{6.491e-02}  &               &  \num{6.491e-02}  &               &  \num{6.406e-02}  &               &  \num{6.531e-02}  &               &  \num{6.531e-02}  &               &  \num{6.491e-02}  &               &  \num{6.207e-02}  &               &  \num{5.829e-02}  &               \\
      &    20  &  \num{3.441e-02}  &  \num{ 0.92}  &  \num{3.387e-02}  &  \num{ 0.94}  &  \num{3.246e-02}  &  \num{ 0.98}  &  \num{3.487e-02}  &  \num{ 0.91}  &  \num{3.487e-02}  &  \num{ 0.91}  &  \num{3.387e-02}  &  \num{ 0.94}  &  \num{3.206e-02}  &  \num{ 0.95}  &  \num{3.250e-02}  &  \num{ 0.84}  \\
      &    40  &  \num{3.096e-02}  &  \num{ 0.15}  &  \num{2.846e-02}  &  \num{ 0.25}  &  \num{2.549e-02}  &  \num{ 0.35}  &  \num{2.751e-02}  &  \num{ 0.34}  &  \num{2.751e-02}  &  \num{ 0.34}  &  \num{2.846e-02}  &  \num{ 0.25}  &  \num{2.852e-02}  &  \num{ 0.17}  &  \num{2.923e-02}  &  \num{ 0.15}  \\
      &    80  &  \num{2.463e-02}  &  \num{ 0.33}  &  \num{2.473e-02}  &  \num{ 0.20}  &  \num{2.350e-02}  &  \num{ 0.12}  &  \num{2.186e-02}  &  \num{ 0.33}  &  \num{2.186e-02}  &  \num{ 0.33}  &  \num{2.473e-02}  &  \num{ 0.20}  &  \num{2.435e-02}  &  \num{ 0.23}  &  \num{2.320e-02}  &  \num{ 0.33}  \\
      &   160  &  \num{1.841e-02}  &  \num{ 0.42}  &  \num{1.831e-02}  &  \num{ 0.43}  &  \num{1.779e-02}  &  \num{ 0.40}  &  \num{1.620e-02}  &  \num{ 0.43}  &  \num{1.620e-02}  &  \num{ 0.43}  &  \num{1.831e-02}  &  \num{ 0.43}  &  \num{1.859e-02}  &  \num{ 0.39}  &  \num{1.711e-02}  &  \num{ 0.44}  \\
      &   320  &  \num{1.700e-02}  &  \num{ 0.12}  &  \num{1.721e-02}  &  \num{ 0.09}  &  \num{1.700e-02}  &  \num{ 0.07}  &  \num{1.618e-02}  &  \num{ 0.00}  &  \num{1.618e-02}  &  \num{ 0.00}  &  \num{1.721e-02}  &  \num{ 0.09}  &  \num{1.671e-02}  &  \num{ 0.15}  &  \num{1.555e-02}  &  \num{ 0.14}  \\
  \bottomrule
\end{tabular}
\end{center}
\end{sidewaystable}

\begin{table}
\small
\caption{Errors and experimental order of convergence (EOC) for varying polynomial
         degrees $p$ and number of elements $N$ for the modified Sod shock tube
         problem \eqref{eq:modsod} using the Suliciu relaxation solver as numerical
         flux and "simple" volume fluxes.}
\label{tab:modsod-nonEC-suliciu}
\sisetup{
  output-exponent-marker=\text{e},
  round-mode=places,
  round-precision=2
}
\begin{center}
\begin{tabular}{rr | rr|rr|rr|rr|rr}
  \toprule
  & &  \multicolumn{2}{c}{Central}
    &  \multicolumn{2}{c}{Morinishi} 
    &  \multicolumn{2}{c}{Ducros}
    &  \multicolumn{2}{c}{KG}
    &  \multicolumn{2}{c}{Pirozzoli}
  \\
  $p$ & $N$
    & $\norm{\mathrm{err_\rho}}_M$ & EOC
    & $\norm{\mathrm{err_\rho}}_M$ & EOC
    & $\norm{\mathrm{err_\rho}}_M$ & EOC
    & $\norm{\mathrm{err_\rho}}_M$ & EOC
    & $\norm{\mathrm{err_\rho}}_M$ & EOC
  \\
  \midrule
   1  &    10  &  \num{1.418e-01}  &               &  \num{1.424e-01}  &               &  \num{1.340e-01}  &               &  \num{1.334e-01}  &               &  \num{1.334e-01}  &               \\
      &    20  &  \num{7.459e-02}  &  \num{ 0.93}  &  \num{8.548e-02}  &  \num{ 0.74}  &  \num{7.209e-02}  &  \num{ 0.89}  &  \num{7.408e-02}  &  \num{ 0.85}  &  \num{7.482e-02}  &  \num{ 0.83}  \\
      &    40  &  \num{6.627e-02}  &  \num{ 0.17}  &  \num{1.049e-01}  &  \num{-0.29}  &  \num{5.710e-02}  &  \num{ 0.34}  &  \num{6.074e-02}  &  \num{ 0.29}  &  \num{6.172e-02}  &  \num{ 0.28}  \\
      &    80  &  \num{5.995e-02}  &  \num{ 0.14}  &           $*$     &        $*$    &  \num{4.491e-02}  &  \num{ 0.35}  &  \num{4.786e-02}  &  \num{ 0.34}  &  \num{4.861e-02}  &  \num{ 0.34}  \\
      &   160  &  \num{5.251e-02}  &  \num{ 0.19}  &           $*$     &        $*$    &  \num{3.103e-02}  &  \num{ 0.53}  &  \num{3.347e-02}  &  \num{ 0.52}  &  \num{3.416e-02}  &  \num{ 0.51}  \\
      &   320  &  \num{4.864e-02}  &  \num{ 0.11}  &           $*$     &        $*$    &  \num{2.143e-02}  &  \num{ 0.53}  &  \num{2.337e-02}  &  \num{ 0.52}  &  \num{2.401e-02}  &  \num{ 0.51}  \\
  \midrule 
   2  &    10  &           $*$     &               &  \num{8.379e-02}  &               &  \num{7.828e-02}  &               &  \num{8.032e-02}  &               &  \num{8.020e-02}  &               \\
      &    20  &           $*$     &        $*$    &  \num{4.619e-02}  &  \num{ 0.86}  &  \num{4.965e-02}  &  \num{ 0.66}  &  \num{5.096e-02}  &  \num{ 0.66}  &  \num{5.198e-02}  &  \num{ 0.63}  \\
      &    40  &           $*$     &        $*$    &  \num{4.174e-02}  &  \num{ 0.15}  &  \num{3.995e-02}  &  \num{ 0.31}  &  \num{4.200e-02}  &  \num{ 0.28}  &  \num{4.211e-02}  &  \num{ 0.30}  \\
      &    80  &           $*$     &        $*$    &  \num{2.728e-02}  &  \num{ 0.61}  &  \num{2.096e-02}  &  \num{ 0.93}  &  \num{2.135e-02}  &  \num{ 0.98}  &  \num{2.178e-02}  &  \num{ 0.95}  \\
      &   160  &           $*$     &        $*$    &           $*$     &        $*$    &  \num{2.045e-02}  &  \num{ 0.04}  &  \num{2.092e-02}  &  \num{ 0.03}  &  \num{2.111e-02}  &  \num{ 0.04}  \\
      &   320  &           $*$     &        $*$    &           $*$     &        $*$    &  \num{1.269e-02}  &  \num{ 0.69}  &  \num{1.411e-02}  &  \num{ 0.57}  &  \num{1.448e-02}  &  \num{ 0.54}  \\
  \midrule 
   3  &    10  &           $*$     &               &           $*$     &               &           $*$     &               &           $*$     &               &  \num{6.015e-02}  &               \\
      &    20  &           $*$     &        $*$    &           $*$     &        $*$    &           $*$     &        $*$    &           $*$     &        $*$    &           $*$     &        $*$    \\
      &    40  &           $*$     &        $*$    &           $*$     &        $*$    &           $*$     &        $*$    &           $*$     &        $*$    &           $*$     &        $*$    \\
      &    80  &           $*$     &        $*$    &           $*$     &        $*$    &           $*$     &        $*$    &           $*$     &        $*$    &           $*$     &        $*$    \\
      &   160  &           $*$     &        $*$    &           $*$     &        $*$    &           $*$     &        $*$    &           $*$     &        $*$    &           $*$     &        $*$    \\
      &   320  &           $*$     &        $*$    &           $*$     &        $*$    &           $*$     &        $*$    &           $*$     &        $*$    &           $*$     &        $*$    \\
  \midrule 
   4  &    10  &           $*$     &               &           $*$     &               &           $*$     &               &           $*$     &               &           $*$     &               \\
      &    20  &           $*$     &        $*$    &           $*$     &        $*$    &           $*$     &        $*$    &           $*$     &        $*$    &           $*$     &        $*$    \\
      &    40  &           $*$     &        $*$    &           $*$     &        $*$    &           $*$     &        $*$    &           $*$     &        $*$    &           $*$     &        $*$    \\
      &    80  &           $*$     &        $*$    &           $*$     &        $*$    &           $*$     &        $*$    &           $*$     &        $*$    &           $*$     &        $*$    \\
      &   160  &           $*$     &        $*$    &           $*$     &        $*$    &           $*$     &        $*$    &           $*$     &        $*$    &           $*$     &        $*$    \\
      &   320  &           $*$     &        $*$    &           $*$     &        $*$    &           $*$     &        $*$    &           $*$     &        $*$    &           $*$     &        $*$    \\
  \midrule 
   5  &    10  &           $*$     &               &           $*$     &               &           $*$     &               &           $*$     &               &           $*$     &               \\
      &    20  &           $*$     &        $*$    &           $*$     &        $*$    &           $*$     &        $*$    &           $*$     &        $*$    &           $*$     &        $*$    \\
      &    40  &           $*$     &        $*$    &           $*$     &        $*$    &           $*$     &        $*$    &           $*$     &        $*$    &           $*$     &        $*$    \\
      &    80  &           $*$     &        $*$    &           $*$     &        $*$    &           $*$     &        $*$    &           $*$     &        $*$    &           $*$     &        $*$    \\
      &   160  &           $*$     &        $*$    &           $*$     &        $*$    &           $*$     &        $*$    &           $*$     &        $*$    &           $*$     &        $*$    \\
      &   320  &           $*$     &        $*$    &           $*$     &        $*$    &           $*$     &        $*$    &           $*$     &        $*$    &           $*$     &        $*$    \\
  \bottomrule
\end{tabular}
\end{center}
\end{table}

\begin{table}
\small
\caption{Errors and experimental order of convergence (EOC) for varying polynomial
         degrees $p$ and number of elements $N$ for the modified Sod shock tube
         problem \eqref{eq:modsod} using the local Lax-Friedrichs numerical
         flux and "simple" volume fluxes.}
\label{tab:modsod-nonEC-LLF}
\sisetup{
  output-exponent-marker=\text{e},
  round-mode=places,
  round-precision=2
}
\begin{center}
\begin{tabular}{rr | rr|rr|rr|rr|rr}
  \toprule
  & &  \multicolumn{2}{c}{Central}
    &  \multicolumn{2}{c}{Morinishi} 
    &  \multicolumn{2}{c}{Ducros}
    &  \multicolumn{2}{c}{KG}
    &  \multicolumn{2}{c}{Pirozzoli}
  \\
  $p$ & $N$
    & $\norm{\mathrm{err_\rho}}_M$ & EOC
    & $\norm{\mathrm{err_\rho}}_M$ & EOC
    & $\norm{\mathrm{err_\rho}}_M$ & EOC
    & $\norm{\mathrm{err_\rho}}_M$ & EOC
    & $\norm{\mathrm{err_\rho}}_M$ & EOC
  \\
  \midrule
   1  &    10  &  \num{1.443e-01}  &               &  \num{1.445e-01}  &               &  \num{1.351e-01}  &               &  \num{1.343e-01}  &               &  \num{1.342e-01}  &               \\
      &    20  &  \num{7.295e-02}  &  \num{ 0.98}  &  \num{8.232e-02}  &  \num{ 0.81}  &  \num{6.952e-02}  &  \num{ 0.96}  &  \num{7.177e-02}  &  \num{ 0.90}  &  \num{7.230e-02}  &  \num{ 0.89}  \\
      &    40  &  \num{6.676e-02}  &  \num{ 0.13}  &  \num{9.441e-02}  &  \num{-0.20}  &  \num{5.911e-02}  &  \num{ 0.23}  &  \num{6.302e-02}  &  \num{ 0.19}  &  \num{6.437e-02}  &  \num{ 0.17}  \\
      &    80  &  \num{6.084e-02}  &  \num{ 0.13}  &           $*$     &        $*$    &  \num{4.617e-02}  &  \num{ 0.36}  &  \num{4.936e-02}  &  \num{ 0.35}  &  \num{5.039e-02}  &  \num{ 0.35}  \\
      &   160  &  \num{5.163e-02}  &  \num{ 0.24}  &           $*$     &        $*$    &  \num{3.221e-02}  &  \num{ 0.52}  &  \num{3.484e-02}  &  \num{ 0.50}  &  \num{3.575e-02}  &  \num{ 0.49}  \\
      &   320  &  \num{4.717e-02}  &  \num{ 0.13}  &           $*$     &        $*$    &  \num{2.246e-02}  &  \num{ 0.52}  &  \num{2.455e-02}  &  \num{ 0.50}  &  \num{2.533e-02}  &  \num{ 0.50}  \\
  \midrule 
   2  &    10  &  \num{8.424e-02}  &               &  \num{8.455e-02}  &               &  \num{8.618e-02}  &               &  \num{8.668e-02}  &               &  \num{8.672e-02}  &               \\
      &    20  &  \num{5.075e-02}  &  \num{ 0.73}  &  \num{5.227e-02}  &  \num{ 0.69}  &  \num{5.362e-02}  &  \num{ 0.68}  &  \num{5.428e-02}  &  \num{ 0.68}  &  \num{5.499e-02}  &  \num{ 0.66}  \\
      &    40  &  \num{3.781e-02}  &  \num{ 0.42}  &  \num{4.919e-02}  &  \num{ 0.09}  &  \num{3.985e-02}  &  \num{ 0.43}  &  \num{4.200e-02}  &  \num{ 0.37}  &  \num{4.219e-02}  &  \num{ 0.38}  \\
      &    80  &  \num{2.300e-02}  &  \num{ 0.72}  &           $*$     &        $*$    &  \num{2.254e-02}  &  \num{ 0.82}  &  \num{2.211e-02}  &  \num{ 0.93}  &  \num{2.253e-02}  &  \num{ 0.90}  \\
      &   160  &  \num{2.144e-02}  &  \num{ 0.10}  &           $*$     &        $*$    &  \num{2.108e-02}  &  \num{ 0.10}  &  \num{2.075e-02}  &  \num{ 0.09}  &  \num{2.102e-02}  &  \num{ 0.10}  \\
      &   320  &  \num{1.429e-02}  &  \num{ 0.59}  &           $*$     &        $*$    &  \num{1.322e-02}  &  \num{ 0.67}  &  \num{1.404e-02}  &  \num{ 0.56}  &  \num{1.443e-02}  &  \num{ 0.54}  \\
  \midrule 
   3  &    10  &           $*$     &               &           $*$     &               &           $*$     &               &  \num{6.105e-02}  &               &  \num{6.223e-02}  &               \\
      &    20  &           $*$     &        $*$    &           $*$     &        $*$    &           $*$     &        $*$    &  \num{4.526e-02}  &  \num{ 0.43}  &  \num{4.579e-02}  &  \num{ 0.44}  \\
      &    40  &           $*$     &        $*$    &           $*$     &        $*$    &           $*$     &        $*$    &  \num{2.477e-02}  &  \num{ 0.87}  &  \num{2.511e-02}  &  \num{ 0.87}  \\
      &    80  &           $*$     &        $*$    &           $*$     &        $*$    &           $*$     &        $*$    &  \num{2.279e-02}  &  \num{ 0.12}  &  \num{2.293e-02}  &  \num{ 0.13}  \\
      &   160  &           $*$     &        $*$    &           $*$     &        $*$    &           $*$     &        $*$    &  \num{1.522e-02}  &  \num{ 0.58}  &  \num{1.546e-02}  &  \num{ 0.57}  \\
      &   320  &           $*$     &        $*$    &           $*$     &        $*$    &           $*$     &        $*$    &  \num{1.160e-02}  &  \num{ 0.39}  &  \num{1.169e-02}  &  \num{ 0.40}  \\
  \midrule 
   4  &    10  &           $*$     &               &           $*$     &               &           $*$     &               &           $*$     &               &           $*$     &               \\
      &    20  &           $*$     &        $*$    &           $*$     &        $*$    &           $*$     &        $*$    &           $*$     &        $*$    &           $*$     &        $*$    \\
      &    40  &           $*$     &        $*$    &           $*$     &        $*$    &           $*$     &        $*$    &           $*$     &        $*$    &           $*$     &        $*$    \\
      &    80  &           $*$     &        $*$    &           $*$     &        $*$    &           $*$     &        $*$    &           $*$     &        $*$    &           $*$     &        $*$    \\
      &   160  &           $*$     &        $*$    &           $*$     &        $*$    &           $*$     &        $*$    &           $*$     &        $*$    &           $*$     &        $*$    \\
      &   320  &           $*$     &        $*$    &           $*$     &        $*$    &           $*$     &        $*$    &           $*$     &        $*$    &           $*$     &        $*$    \\
  \midrule 
   5  &    10  &           $*$     &               &           $*$     &               &           $*$     &               &           $*$     &               &           $*$     &               \\
      &    20  &           $*$     &        $*$    &           $*$     &        $*$    &           $*$     &        $*$    &           $*$     &        $*$    &           $*$     &        $*$    \\
      &    40  &           $*$     &        $*$    &           $*$     &        $*$    &           $*$     &        $*$    &           $*$     &        $*$    &           $*$     &        $*$    \\
      &    80  &           $*$     &        $*$    &           $*$     &        $*$    &           $*$     &        $*$    &           $*$     &        $*$    &           $*$     &        $*$    \\
      &   160  &           $*$     &        $*$    &           $*$     &        $*$    &           $*$     &        $*$    &           $*$     &        $*$    &           $*$     &        $*$    \\
      &   320  &           $*$     &        $*$    &           $*$     &        $*$    &           $*$     &        $*$    &           $*$     &        $*$    &           $*$     &        $*$    \\
  \bottomrule
\end{tabular}
\end{center}
\end{table}

\clearpage
\subsection{Sod's Shock Tube: Finite Volume Setting}
\label{sec:sod-FV}

Here, the classical shock tube of \citet{sod1978survey} with initial condition
\eqref{eq:sod} of section \ref{sec:sod-flux-diff} will be used again, but in the
context of first order finite volume methods.

The entropy conservative flux \eqref{eq:Chandrashekar-rho-v-beta-EC-KEP} of
\citet{chandrashekar2013kinetic} has been used with the scalar dissipation (SD)
of \citet{derigs2016novelAveraging}, the matrix (MD) and hybrid (HD) dissipation
of \citet{winters2016uniquely} and the local Lax-Friedrichs (LLF) dissipation
operator. The last one has also been used for the other entropy conservative
fluxes. Additionally, the classical LLF flux and Suliciu relaxation solver of
\citet{bouchut2004nonlinear} are tested.

The results are shown in Table \ref{tab:sod-FV}.
Here, the matrix dissipation (MD) and the Suliciu solver perform equally good
and yield less error than the other fluxes. Additionally, there is nearly no
variance across the solvers using the LLF or scalar dissipation operator.

\begin{table}[!htp]
\small
\caption{Errors and experimental order of convergence (EOC) for varying number
         of elements $N$ for the Sod shock tube problem \eqref{eq:sod} using
         several numerical fluxes.}
\label{tab:sod-FV}
\sisetup{
  output-exponent-marker=\text{e},
  round-mode=places,
  round-precision=2
}
\begin{center}
\begin{tabular}{r | rr|rr|rr|rr}
  \toprule
  & \multicolumn{2}{c}{Ch + SD DWGW} 
  & \multicolumn{2}{c}{Ch + MD DWGW} 
  & \multicolumn{2}{c}{Ch + HD DWGW} 
  & \multicolumn{2}{c}{Ch + LLF} 
  \\
  $N$
    & $\norm{\mathrm{err_\rho}}_M$ & EOC
    & $\norm{\mathrm{err_\rho}}_M$ & EOC
    & $\norm{\mathrm{err_\rho}}_M$ & EOC
    & $\norm{\mathrm{err_\rho}}_M$ & EOC
  \\
  \midrule
    100  &  \num{4.618e-02}  &               &  \num{3.965e-02}  &               &  \num{4.221e-02}  &               &  \num{4.622e-02}  &               \\
    200  &  \num{3.745e-02}  &  \num{ 0.30}  &  \num{3.040e-02}  &  \num{ 0.38}  &  \num{3.234e-02}  &  \num{ 0.38}  &  \num{3.747e-02}  &  \num{ 0.30}  \\
    400  &  \num{2.796e-02}  &  \num{ 0.42}  &  \num{2.190e-02}  &  \num{ 0.47}  &  \num{2.309e-02}  &  \num{ 0.49}  &  \num{2.797e-02}  &  \num{ 0.42}  \\
    800  &  \num{2.077e-02}  &  \num{ 0.43}  &  \num{1.623e-02}  &  \num{ 0.43}  &  \num{1.687e-02}  &  \num{ 0.45}  &  \num{2.077e-02}  &  \num{ 0.43}  \\
   1600  &  \num{1.556e-02}  &  \num{ 0.42}  &  \num{1.228e-02}  &  \num{ 0.40}  &  \num{1.258e-02}  &  \num{ 0.42}  &  \num{1.556e-02}  &  \num{ 0.42}  \\
   3200  &  \num{1.181e-02}  &  \num{ 0.40}  &  \num{9.329e-03}  &  \num{ 0.40}  &  \num{9.486e-03}  &  \num{ 0.41}  &  \num{1.181e-02}  &  \num{ 0.40}  \\
   6400  &  \num{9.241e-03}  &  \num{ 0.35}  &  \num{7.345e-03}  &  \num{ 0.34}  &  \num{7.415e-03}  &  \num{ 0.36}  &  \num{9.242e-03}  &  \num{ 0.35}  \\
  12800  &  \num{7.216e-03}  &  \num{ 0.36}  &  \num{5.597e-03}  &  \num{ 0.39}  &  \num{5.637e-03}  &  \num{ 0.40}  &  \num{7.217e-03}  &  \num{ 0.36}  \\
  \bottomrule
  \toprule
  & \multicolumn{2}{c}{$\rho,v,\beta$ (2) + LLF}
  & \multicolumn{2}{c}{$\rho,v,\frac{1}{p}$ + LLF}
  & \multicolumn{2}{c}{$\rho,v,p$ + LLF}
  & \multicolumn{2}{c}{$\rho,v,T$ (1) + LLF}
  \\
  $N$
    & $\norm{\mathrm{err_\rho}}_M$ & EOC
    & $\norm{\mathrm{err_\rho}}_M$ & EOC
    & $\norm{\mathrm{err_\rho}}_M$ & EOC
    & $\norm{\mathrm{err_\rho}}_M$ & EOC
  \\
  \midrule
    100  &  \num{4.620e-02}  &               &  \num{4.607e-02}  &               &  \num{4.607e-02}  &               &  \num{4.622e-02}  &               \\
    200  &  \num{3.746e-02}  &  \num{ 0.30}  &  \num{3.740e-02}  &  \num{ 0.30}  &  \num{3.740e-02}  &  \num{ 0.30}  &  \num{3.747e-02}  &  \num{ 0.30}  \\
    400  &  \num{2.796e-02}  &  \num{ 0.42}  &  \num{2.795e-02}  &  \num{ 0.42}  &  \num{2.795e-02}  &  \num{ 0.42}  &  \num{2.797e-02}  &  \num{ 0.42}  \\
    800  &  \num{2.076e-02}  &  \num{ 0.43}  &  \num{2.077e-02}  &  \num{ 0.43}  &  \num{2.077e-02}  &  \num{ 0.43}  &  \num{2.077e-02}  &  \num{ 0.43}  \\
   1600  &  \num{1.556e-02}  &  \num{ 0.42}  &  \num{1.557e-02}  &  \num{ 0.42}  &  \num{1.557e-02}  &  \num{ 0.42}  &  \num{1.556e-02}  &  \num{ 0.42}  \\
   3200  &  \num{1.181e-02}  &  \num{ 0.40}  &  \num{1.182e-02}  &  \num{ 0.40}  &  \num{1.182e-02}  &  \num{ 0.40}  &  \num{1.181e-02}  &  \num{ 0.40}  \\
   6400  &  \num{9.241e-03}  &  \num{ 0.35}  &  \num{9.243e-03}  &  \num{ 0.35}  &  \num{9.243e-03}  &  \num{ 0.35}  &  \num{9.242e-03}  &  \num{ 0.35}  \\
  12800  &  \num{7.216e-03}  &  \num{ 0.36}  &  \num{7.217e-03}  &  \num{ 0.36}  &  \num{7.217e-03}  &  \num{ 0.36}  &  \num{7.217e-03}  &  \num{ 0.36}  \\
  \bottomrule
  \toprule
  & \multicolumn{2}{c}{$\rho,v,T$ (2) + LLF} 
  & \multicolumn{2}{c}{$\rho,V,T$ (rev) + LLF}
  & \multicolumn{2}{c}{LLF} 
  & \multicolumn{2}{c}{Suliciu} 
  \\
  $N$
    & $\norm{\mathrm{err_\rho}}_M$ & EOC
    & $\norm{\mathrm{err_\rho}}_M$ & EOC
    & $\norm{\mathrm{err_\rho}}_M$ & EOC
    & $\norm{\mathrm{err_\rho}}_M$ & EOC
  \\
  \midrule
    100  &  \num{4.623e-02}  &               &  \num{4.632e-02}  &               &  \num{4.576e-02}  &               &  \num{3.945e-02}  &               \\
    200  &  \num{3.747e-02}  &  \num{ 0.30}  &  \num{3.755e-02}  &  \num{ 0.30}  &  \num{3.721e-02}  &  \num{ 0.30}  &  \num{3.031e-02}  &  \num{ 0.38}  \\
    400  &  \num{2.797e-02}  &  \num{ 0.42}  &  \num{2.800e-02}  &  \num{ 0.42}  &  \num{2.785e-02}  &  \num{ 0.42}  &  \num{2.191e-02}  &  \num{ 0.47}  \\
    800  &  \num{2.077e-02}  &  \num{ 0.43}  &  \num{2.078e-02}  &  \num{ 0.43}  &  \num{2.072e-02}  &  \num{ 0.43}  &  \num{1.626e-02}  &  \num{ 0.43}  \\
   1600  &  \num{1.556e-02}  &  \num{ 0.42}  &  \num{1.557e-02}  &  \num{ 0.42}  &  \num{1.555e-02}  &  \num{ 0.41}  &  \num{1.230e-02}  &  \num{ 0.40}  \\
   3200  &  \num{1.181e-02}  &  \num{ 0.40}  &  \num{1.182e-02}  &  \num{ 0.40}  &  \num{1.182e-02}  &  \num{ 0.40}  &  \num{9.359e-03}  &  \num{ 0.39}  \\
   6400  &  \num{9.241e-03}  &  \num{ 0.35}  &  \num{9.243e-03}  &  \num{ 0.35}  &  \num{9.249e-03}  &  \num{ 0.35}  &  \num{7.362e-03}  &  \num{ 0.35}  \\
  12800  &  \num{7.216e-03}  &  \num{ 0.36}  &  \num{7.217e-03}  &  \num{ 0.36}  &  \num{7.229e-03}  &  \num{ 0.36}  &  \num{5.626e-03}  &  \num{ 0.39}  \\
  \bottomrule
\end{tabular}
\end{center}
\end{table}

\subsection{Modified Version of Sod's Shock Tube: Finite Volume Setting}
\label{sec:modsod-FV}

Similar to the previous section, the modified Sod shock tube problem of section
\ref{sec:modsod-flux-diff} is used to test the finite volume fluxes.
The results are shown in Table \ref{tab:modsod-FV}.
Again, the matrix dissipation and the Suliciu solver perform equally good and
are superior to the other fluxes. As in section \ref{sec:sod-FV}, there is nearly
no variance across the methods with matrix / LLF dissipation.

\begin{table}[!htp]
\small
\caption{Errors and experimental order of convergence (EOC) for varying number
         of elements $N$ for the modified Sod shock tube problem \eqref{eq:modsod}
         using several numerical fluxes.}
\label{tab:modsod-FV}
\sisetup{
  output-exponent-marker=\text{e},
  round-mode=places,
  round-precision=2
}
\begin{center}
\begin{tabular}{r | rr|rr|rr|rr}
  \toprule
  & \multicolumn{2}{c}{Ch + SD DWGW} 
  & \multicolumn{2}{c}{Ch + MD DWGW} 
  & \multicolumn{2}{c}{Ch + HD DWGW} 
  & \multicolumn{2}{c}{Ch + LLF} 
  \\
  $N$
    & $\norm{\mathrm{err_\rho}}_M$ & EOC
    & $\norm{\mathrm{err_\rho}}_M$ & EOC
    & $\norm{\mathrm{err_\rho}}_M$ & EOC
    & $\norm{\mathrm{err_\rho}}_M$ & EOC
  \\
  \midrule
    100  &  \num{5.548e-02}  &               &  \num{4.026e-02}  &               &  \num{4.592e-02}  &               &  \num{5.549e-02}  &               \\
    200  &  \num{4.669e-02}  &  \num{ 0.25}  &  \num{3.308e-02}  &  \num{ 0.28}  &  \num{3.610e-02}  &  \num{ 0.35}  &  \num{4.669e-02}  &  \num{ 0.25}  \\
    400  &  \num{3.700e-02}  &  \num{ 0.34}  &  \num{2.647e-02}  &  \num{ 0.32}  &  \num{2.795e-02}  &  \num{ 0.37}  &  \num{3.700e-02}  &  \num{ 0.34}  \\
    800  &  \num{2.898e-02}  &  \num{ 0.35}  &  \num{2.140e-02}  &  \num{ 0.31}  &  \num{2.212e-02}  &  \num{ 0.34}  &  \num{2.898e-02}  &  \num{ 0.35}  \\
   1600  &  \num{2.231e-02}  &  \num{ 0.38}  &  \num{1.678e-02}  &  \num{ 0.35}  &  \num{1.715e-02}  &  \num{ 0.37}  &  \num{2.231e-02}  &  \num{ 0.38}  \\
   3200  &  \num{1.758e-02}  &  \num{ 0.34}  &  \num{1.373e-02}  &  \num{ 0.29}  &  \num{1.387e-02}  &  \num{ 0.31}  &  \num{1.758e-02}  &  \num{ 0.34}  \\
   6400  &  \num{1.392e-02}  &  \num{ 0.34}  &  \num{1.101e-02}  &  \num{ 0.32}  &  \num{1.108e-02}  &  \num{ 0.32}  &  \num{1.392e-02}  &  \num{ 0.34}  \\
  12800  &  \num{1.098e-02}  &  \num{ 0.34}  &  \num{8.510e-03}  &  \num{ 0.37}  &  \num{8.572e-03}  &  \num{ 0.37}  &  \num{1.099e-02}  &  \num{ 0.34}  \\
  \bottomrule
  \toprule
  & \multicolumn{2}{c}{$\rho,v,\beta$ (2) + LLF}
  & \multicolumn{2}{c}{$\rho,v,\frac{1}{p}$ + LLF}
  & \multicolumn{2}{c}{$\rho,v,p$ + LLF}
  & \multicolumn{2}{c}{$\rho,v,T$ (1) + LLF}
  \\
  $N$
    & $\norm{\mathrm{err_\rho}}_M$ & EOC
    & $\norm{\mathrm{err_\rho}}_M$ & EOC
    & $\norm{\mathrm{err_\rho}}_M$ & EOC
    & $\norm{\mathrm{err_\rho}}_M$ & EOC
  \\
  \midrule
    100  &  \num{5.548e-02}  &               &  \num{5.534e-02}  &               &  \num{5.534e-02}  &               &  \num{5.549e-02}  &               \\
    200  &  \num{4.668e-02}  &  \num{ 0.25}  &  \num{4.668e-02}  &  \num{ 0.25}  &  \num{4.668e-02}  &  \num{ 0.25}  &  \num{4.669e-02}  &  \num{ 0.25}  \\
    400  &  \num{3.699e-02}  &  \num{ 0.34}  &  \num{3.702e-02}  &  \num{ 0.33}  &  \num{3.702e-02}  &  \num{ 0.33}  &  \num{3.700e-02}  &  \num{ 0.34}  \\
    800  &  \num{2.897e-02}  &  \num{ 0.35}  &  \num{2.899e-02}  &  \num{ 0.35}  &  \num{2.899e-02}  &  \num{ 0.35}  &  \num{2.898e-02}  &  \num{ 0.35}  \\
   1600  &  \num{2.230e-02}  &  \num{ 0.38}  &  \num{2.231e-02}  &  \num{ 0.38}  &  \num{2.231e-02}  &  \num{ 0.38}  &  \num{2.231e-02}  &  \num{ 0.38}  \\
   3200  &  \num{1.758e-02}  &  \num{ 0.34}  &  \num{1.759e-02}  &  \num{ 0.34}  &  \num{1.759e-02}  &  \num{ 0.34}  &  \num{1.758e-02}  &  \num{ 0.34}  \\
   6400  &  \num{1.392e-02}  &  \num{ 0.34}  &  \num{1.393e-02}  &  \num{ 0.34}  &  \num{1.393e-02}  &  \num{ 0.34}  &  \num{1.392e-02}  &  \num{ 0.34}  \\
  12800  &  \num{1.098e-02}  &  \num{ 0.34}  &  \num{1.099e-02}  &  \num{ 0.34}  &  \num{1.099e-02}  &  \num{ 0.34}  &  \num{1.099e-02}  &  \num{ 0.34}  \\
  \bottomrule
  \toprule
  & \multicolumn{2}{c}{$\rho,v,T$ (2) + LLF} 
  & \multicolumn{2}{c}{$\rho,V,T$ (rev) + LLF}
  & \multicolumn{2}{c}{LLF} 
  & \multicolumn{2}{c}{Suliciu} 
  \\
  $N$
    & $\norm{\mathrm{err_\rho}}_M$ & EOC
    & $\norm{\mathrm{err_\rho}}_M$ & EOC
    & $\norm{\mathrm{err_\rho}}_M$ & EOC
    & $\norm{\mathrm{err_\rho}}_M$ & EOC
  \\
  \midrule
    100  &  \num{5.550e-02}  &               &  \num{5.554e-02}  &               &  \num{5.581e-02}  &               &  \num{4.130e-02}  &               \\
    200  &  \num{4.670e-02}  &  \num{ 0.25}  &  \num{4.673e-02}  &  \num{ 0.25}  &  \num{4.690e-02}  &  \num{ 0.25}  &  \num{3.335e-02}  &  \num{ 0.31}  \\
    400  &  \num{3.700e-02}  &  \num{ 0.34}  &  \num{3.703e-02}  &  \num{ 0.34}  &  \num{3.714e-02}  &  \num{ 0.34}  &  \num{2.657e-02}  &  \num{ 0.33}  \\
    800  &  \num{2.898e-02}  &  \num{ 0.35}  &  \num{2.899e-02}  &  \num{ 0.35}  &  \num{2.907e-02}  &  \num{ 0.35}  &  \num{2.147e-02}  &  \num{ 0.31}  \\
   1600  &  \num{2.230e-02}  &  \num{ 0.38}  &  \num{2.231e-02}  &  \num{ 0.38}  &  \num{2.238e-02}  &  \num{ 0.38}  &  \num{1.688e-02}  &  \num{ 0.35}  \\
   3200  &  \num{1.758e-02}  &  \num{ 0.34}  &  \num{1.758e-02}  &  \num{ 0.34}  &  \num{1.763e-02}  &  \num{ 0.34}  &  \num{1.379e-02}  &  \num{ 0.29}  \\
   6400  &  \num{1.392e-02}  &  \num{ 0.34}  &  \num{1.393e-02}  &  \num{ 0.34}  &  \num{1.396e-02}  &  \num{ 0.34}  &  \num{1.107e-02}  &  \num{ 0.32}  \\
  12800  &  \num{1.099e-02}  &  \num{ 0.34}  &  \num{1.099e-02}  &  \num{ 0.34}  &  \num{1.103e-02}  &  \num{ 0.34}  &  \num{8.605e-03}  &  \num{ 0.36}  \\
  \bottomrule
\end{tabular}
\end{center}
\end{table}

\subsection{Near Vacuum Rarefaction}
\label{sec:nearVacuumRarefaction-FV}

In this section, the rarefaction waves near vacuum as described by
\citet[Section 4.3.3, Test 2]{toro2009riemann} will be used to test the methods.
The initial condition is given in primitive variables by
\begin{equation}
\label{eq:nearVacuumRarefaction}
\begin{aligned}
  \rho_0(x) = 1
  ,\qquad
  v_0(x) =
  \begin{cases}
    -2, & x < \frac{1}{2}, \\
     2, & \text{ else},
  \end{cases}
  ,\qquad
  p_0(x) = 0.4,
\end{aligned}
\end{equation}
and the conservative variables are again computed via $\rho v_0 = \rho_0 v_0$ and
$\rho e_0 = \frac{1}{2} \rho_0 v_0^2 + \frac{p_0}{\gamma-1}$.
The solution is computed on the domain $[0, 1]$ from $t = 0$ until $t = 0.15$.

Using the same finite volume methods as in section \ref{sec:sod-FV}, the results
are shown in Table \ref{tab:nearVacuumRarefaction-FV}.
Across the varying number of elements $N$, no flux is clearly superior.

\begin{table}[!htp]
\small
\caption{Errors and experimental order of convergence (EOC) for varying number
         of elements $N$ for the near vacuum rarefaction problem
         \eqref{eq:nearVacuumRarefaction} using several numerical fluxes.}
\label{tab:nearVacuumRarefaction-FV}
\sisetup{
  output-exponent-marker=\text{e},
  round-mode=places,
  round-precision=2
}
\begin{center}
\begin{tabular}{r | rr|rr|rr|rr}
  \toprule
  & \multicolumn{2}{c}{Ch + SD DWGW} 
  & \multicolumn{2}{c}{Ch + MD DWGW} 
  & \multicolumn{2}{c}{Ch + HD DWGW} 
  & \multicolumn{2}{c}{Ch + LLF} 
  \\
  $N$
    & $\norm{\mathrm{err_\rho}}_M$ & EOC
    & $\norm{\mathrm{err_\rho}}_M$ & EOC
    & $\norm{\mathrm{err_\rho}}_M$ & EOC
    & $\norm{\mathrm{err_\rho}}_M$ & EOC
  \\
  \midrule
    100  &  \num{8.833e-02}  &               &  \num{7.942e-02}  &               &  \num{7.979e-02}  &               &  \num{8.835e-02}  &               \\
    200  &  \num{5.655e-02}  &  \num{ 0.64}  &  \num{5.670e-02}  &  \num{ 0.49}  &  \num{5.622e-02}  &  \num{ 0.51}  &  \num{5.662e-02}  &  \num{ 0.64}  \\
    400  &  \num{4.008e-02}  &  \num{ 0.50}  &  \num{4.178e-02}  &  \num{ 0.44}  &  \num{4.140e-02}  &  \num{ 0.44}  &  \num{4.014e-02}  &  \num{ 0.50}  \\
    800  &  \num{2.909e-02}  &  \num{ 0.46}  &  \num{2.994e-02}  &  \num{ 0.48}  &  \num{2.979e-02}  &  \num{ 0.47}  &  \num{2.912e-02}  &  \num{ 0.46}  \\
   1600  &  \num{2.024e-02}  &  \num{ 0.52}  &  \num{2.053e-02}  &  \num{ 0.54}  &  \num{2.048e-02}  &  \num{ 0.54}  &  \num{2.026e-02}  &  \num{ 0.52}  \\
   3200  &  \num{1.324e-02}  &  \num{ 0.61}  &  \num{1.334e-02}  &  \num{ 0.62}  &  \num{1.332e-02}  &  \num{ 0.62}  &  \num{1.325e-02}  &  \num{ 0.61}  \\
   6400  &  \num{8.035e-03}  &  \num{ 0.72}  &  \num{8.083e-03}  &  \num{ 0.72}  &  \num{8.060e-03}  &  \num{ 0.72}  &  \num{8.041e-03}  &  \num{ 0.72}  \\
  12800  &  \num{4.251e-03}  &  \num{ 0.92}  &  \num{4.217e-03}  &  \num{ 0.94}  &  \num{4.198e-03}  &  \num{ 0.94}  &  \num{4.262e-03}  &  \num{ 0.92}  \\
  \bottomrule
  \toprule
  & \multicolumn{2}{c}{$\rho,v,\beta$ (2) + LLF}
  & \multicolumn{2}{c}{$\rho,v,\frac{1}{p}$ + LLF}
  & \multicolumn{2}{c}{$\rho,v,p$ + LLF}
  & \multicolumn{2}{c}{$\rho,v,T$ (1) + LLF}
  \\
  $N$
    & $\norm{\mathrm{err_\rho}}_M$ & EOC
    & $\norm{\mathrm{err_\rho}}_M$ & EOC
    & $\norm{\mathrm{err_\rho}}_M$ & EOC
    & $\norm{\mathrm{err_\rho}}_M$ & EOC
  \\
  \midrule
    100  &  \num{8.798e-02}  &               &  \num{8.874e-02}  &               &  \num{8.874e-02}  &               &  \num{8.835e-02}  &               \\
    200  &  \num{5.648e-02}  &  \num{ 0.64}  &  \num{5.684e-02}  &  \num{ 0.64}  &  \num{5.684e-02}  &  \num{ 0.64}  &  \num{5.662e-02}  &  \num{ 0.64}  \\
    400  &  \num{4.009e-02}  &  \num{ 0.49}  &  \num{4.039e-02}  &  \num{ 0.49}  &  \num{4.039e-02}  &  \num{ 0.49}  &  \num{4.014e-02}  &  \num{ 0.50}  \\
    800  &  \num{2.910e-02}  &  \num{ 0.46}  &  \num{2.929e-02}  &  \num{ 0.46}  &  \num{2.929e-02}  &  \num{ 0.46}  &  \num{2.912e-02}  &  \num{ 0.46}  \\
   1600  &  \num{2.025e-02}  &  \num{ 0.52}  &  \num{2.034e-02}  &  \num{ 0.53}  &  \num{2.034e-02}  &  \num{ 0.53}  &  \num{2.026e-02}  &  \num{ 0.52}  \\
   3200  &  \num{1.324e-02}  &  \num{ 0.61}  &  \num{1.329e-02}  &  \num{ 0.61}  &  \num{1.329e-02}  &  \num{ 0.61}  &  \num{1.325e-02}  &  \num{ 0.61}  \\
   6400  &  \num{8.035e-03}  &  \num{ 0.72}  &  \num{8.065e-03}  &  \num{ 0.72}  &  \num{8.065e-03}  &  \num{ 0.72}  &  \num{8.041e-03}  &  \num{ 0.72}  \\
  12800  &  \num{4.252e-03}  &  \num{ 0.92}  &  \num{4.285e-03}  &  \num{ 0.91}  &  \num{4.285e-03}  &  \num{ 0.91}  &  \num{4.262e-03}  &  \num{ 0.92}  \\
  \bottomrule
  \toprule
  & \multicolumn{2}{c}{$\rho,v,T$ (2) + LLF} 
  & \multicolumn{2}{c}{$\rho,V,T$ (rev) + LLF}
  & \multicolumn{2}{c}{LLF} 
  & \multicolumn{2}{c}{Suliciu} 
  \\
  $N$
    & $\norm{\mathrm{err_\rho}}_M$ & EOC
    & $\norm{\mathrm{err_\rho}}_M$ & EOC
    & $\norm{\mathrm{err_\rho}}_M$ & EOC
    & $\norm{\mathrm{err_\rho}}_M$ & EOC
  \\
  \midrule
    100  &  \num{8.798e-02}  &               &  \num{8.839e-02}  &               &  \num{8.070e-02}  &               &  \num{7.716e-02}  &               \\
    200  &  \num{5.632e-02}  &  \num{ 0.64}  &  \num{5.668e-02}  &  \num{ 0.64}  &  \num{5.516e-02}  &  \num{ 0.55}  &  \num{5.553e-02}  &  \num{ 0.47}  \\
    400  &  \num{3.994e-02}  &  \num{ 0.50}  &  \num{4.018e-02}  &  \num{ 0.50}  &  \num{4.003e-02}  &  \num{ 0.46}  &  \num{4.104e-02}  &  \num{ 0.44}  \\
    800  &  \num{2.901e-02}  &  \num{ 0.46}  &  \num{2.915e-02}  &  \num{ 0.46}  &  \num{2.894e-02}  &  \num{ 0.47}  &  \num{2.946e-02}  &  \num{ 0.48}  \\
   1600  &  \num{2.020e-02}  &  \num{ 0.52}  &  \num{2.027e-02}  &  \num{ 0.52}  &  \num{2.007e-02}  &  \num{ 0.53}  &  \num{2.024e-02}  &  \num{ 0.54}  \\
   3200  &  \num{1.322e-02}  &  \num{ 0.61}  &  \num{1.325e-02}  &  \num{ 0.61}  &  \num{1.311e-02}  &  \num{ 0.61}  &  \num{1.318e-02}  &  \num{ 0.62}  \\
   6400  &  \num{8.019e-03}  &  \num{ 0.72}  &  \num{8.045e-03}  &  \num{ 0.72}  &  \num{7.918e-03}  &  \num{ 0.73}  &  \num{7.979e-03}  &  \num{ 0.72}  \\
  12800  &  \num{4.234e-03}  &  \num{ 0.92}  &  \num{4.264e-03}  &  \num{ 0.92}  &  \num{4.054e-03}  &  \num{ 0.97}  &  \num{4.117e-03}  &  \num{ 0.95}  \\
  \bottomrule
\end{tabular}
\end{center}
\end{table}

\subsection{Left Half of the Blast Wave Problem of Woodward and Colella}
\label{sec:blastWaveLeftWC-FV}

In this section, the left half of the blast wave problem of
\citet[Section IV.a]{woodward1984numerical} as described by
\citet[Section 4.3.3, Test 3]{toro2009riemann} is considered. In primitive variables,
it is given by
\begin{equation}
\label{eq:blastWaveLeftWC}
\begin{aligned}
  \rho_0(x) = 1
  ,\qquad
  v_0(x) = 0
  ,\qquad
  p_0(x) =
  \begin{cases}
    1000, & x < \frac{1}{2}, \\
    0.01, & \text{ else},
  \end{cases}
\end{aligned}
\end{equation}
and the conservative variables are again computed via $\rho v_0 = \rho_0 v_0$ and
$\rho e_0 = \frac{1}{2} \rho_0 v_0^2 + \frac{p_0}{\gamma-1}$.
The solution is computed on the domain $[0, 1]$ from $t = 0$ until $t = 0.012$
using the finite volume methods described in section \ref{sec:sod-FV}.

The results are shown in table \ref{tab:blastWaveLeftWC-FV}. The simulations
using the fluxes with variables $\rho, v$ and $p$ or $\frac{1}{p}$ crashed since
they left the invariant region for the Euler equations. As described by
\citet{derigs2016novelAveraging}, the reason is the appearance of the pressure in
the density flux as in the entropy conservative flux of \citet{roe2006affordable}.

The Suliciu solver and the matrix dissipation (MD) yielded similar errors until
the last one crashed using \num{12800} elements. There is a bit more variance across
the other fluxes than in the previous sections \ref{sec:sod-FV}, \ref{sec:modsod-FV}
and \ref{sec:nearVacuumRarefaction-FV}. However, in the end, the Suliciu relaxation
solver performs better than the others.

\begin{table}[!htp]
\small
\caption{Errors and experimental order of convergence (EOC) for varying number
         of elements $N$ for the left half \eqref{eq:blastWaveLeftWC} of 
         the blast wave problem of \citet{woodward1984numerical} using several
         numerical fluxes.}
\label{tab:blastWaveLeftWC-FV}
\sisetup{
  output-exponent-marker=\text{e},
  round-mode=places,
  round-precision=2
}
\begin{center}
\begin{tabular}{r | rr|rr|rr|rr}
  \toprule
  & \multicolumn{2}{c}{Ch + SD DWGW} 
  & \multicolumn{2}{c}{Ch + MD DWGW} 
  & \multicolumn{2}{c}{Ch + HD DWGW} 
  & \multicolumn{2}{c}{Ch + LLF} 
  \\
  $N$
    & $\norm{\mathrm{err_\rho}}_M$ & EOC
    & $\norm{\mathrm{err_\rho}}_M$ & EOC
    & $\norm{\mathrm{err_\rho}}_M$ & EOC
    & $\norm{\mathrm{err_\rho}}_M$ & EOC
  \\
  \midrule
    100  &  \num{7.347e-01}  &               &  \num{7.090e-01}  &               &  \num{7.033e-01}  &               &  \num{7.350e-01}  &               \\
    200  &  \num{6.596e-01}  &  \num{ 0.16}  &  \num{6.333e-01}  &  \num{ 0.16}  &  \num{6.242e-01}  &  \num{ 0.17}  &  \num{6.594e-01}  &  \num{ 0.16}  \\
    400  &  \num{5.323e-01}  &  \num{ 0.31}  &  \num{4.919e-01}  &  \num{ 0.36}  &  \num{4.887e-01}  &  \num{ 0.35}  &  \num{5.322e-01}  &  \num{ 0.31}  \\
    800  &  \num{4.327e-01}  &  \num{ 0.30}  &  \num{3.923e-01}  &  \num{ 0.33}  &  \num{3.908e-01}  &  \num{ 0.32}  &  \num{4.326e-01}  &  \num{ 0.30}  \\
   1600  &  \num{3.507e-01}  &  \num{ 0.30}  &  \num{3.121e-01}  &  \num{ 0.33}  &  \num{3.118e-01}  &  \num{ 0.33}  &  \num{3.507e-01}  &  \num{ 0.30}  \\
   3200  &  \num{2.910e-01}  &  \num{ 0.27}  &  \num{2.579e-01}  &  \num{ 0.27}  &  \num{2.573e-01}  &  \num{ 0.28}  &  \num{2.909e-01}  &  \num{ 0.27}  \\
   6400  &  \num{2.379e-01}  &  \num{ 0.29}  &  \num{2.059e-01}  &  \num{ 0.32}  &  \num{2.070e-01}  &  \num{ 0.31}  &  \num{2.379e-01}  &  \num{ 0.29}  \\
  12800  &  \num{1.916e-01}  &  \num{ 0.31}  &          $*$      &               &  \num{1.618e-01}  &  \num{ 0.36}  &  \num{1.922e-01}  &  \num{ 0.31}  \\
  \bottomrule
  \toprule
  & \multicolumn{2}{c}{$\rho,v,\beta$ (2) + LLF}
  & \multicolumn{2}{c}{$\rho,v,\frac{1}{p}$ + LLF}
  & \multicolumn{2}{c}{$\rho,v,p$ + LLF}
  & \multicolumn{2}{c}{$\rho,v,T$ (1) + LLF}
  \\
  $N$
    & $\norm{\mathrm{err_\rho}}_M$ & EOC
    & $\norm{\mathrm{err_\rho}}_M$ & EOC
    & $\norm{\mathrm{err_\rho}}_M$ & EOC
    & $\norm{\mathrm{err_\rho}}_M$ & EOC
  \\
  \midrule
    100  &  \num{7.336e-01}  &               &          $*$      &               &          $*$      &               &  \num{7.350e-01}  &               \\
    200  &  \num{6.588e-01}  &  \num{ 0.16}  &          $*$      &               &          $*$      &               &  \num{6.594e-01}  &  \num{ 0.16}  \\
    400  &  \num{5.319e-01}  &  \num{ 0.31}  &          $*$      &               &          $*$      &               &  \num{5.322e-01}  &  \num{ 0.31}  \\
    800  &  \num{4.325e-01}  &  \num{ 0.30}  &          $*$      &               &          $*$      &               &  \num{4.326e-01}  &  \num{ 0.30}  \\
   1600  &  \num{3.506e-01}  &  \num{ 0.30}  &          $*$      &               &          $*$      &               &  \num{3.507e-01}  &  \num{ 0.30}  \\
   3200  &  \num{2.909e-01}  &  \num{ 0.27}  &          $*$      &               &          $*$      &               &  \num{2.909e-01}  &  \num{ 0.27}  \\
   6400  &  \num{2.378e-01}  &  \num{ 0.29}  &          $*$      &               &          $*$      &               &  \num{2.379e-01}  &  \num{ 0.29}  \\
  12800  &  \num{1.920e-01}  &  \num{ 0.31}  &          $*$      &               &          $*$      &               &  \num{1.922e-01}  &  \num{ 0.31}  \\
  \bottomrule
  \toprule
  & \multicolumn{2}{c}{$\rho,v,T$ (2) + LLF} 
  & \multicolumn{2}{c}{$\rho,V,T$ (rev) + LLF}
  & \multicolumn{2}{c}{LLF} 
  & \multicolumn{2}{c}{Suliciu} 
  \\
  $N$
    & $\norm{\mathrm{err_\rho}}_M$ & EOC
    & $\norm{\mathrm{err_\rho}}_M$ & EOC
    & $\norm{\mathrm{err_\rho}}_M$ & EOC
    & $\norm{\mathrm{err_\rho}}_M$ & EOC
  \\
  \midrule
    100  &  \num{7.338e-01}  &               &  \num{7.366e-01}  &               &  \num{7.424e-01}  &               &  \num{7.139e-01}  &               \\
    200  &  \num{6.587e-01}  &  \num{ 0.16}  &  \num{6.601e-01}  &  \num{ 0.16}  &  \num{6.644e-01}  &  \num{ 0.16}  &  \num{6.336e-01}  &  \num{ 0.17}  \\
    400  &  \num{5.318e-01}  &  \num{ 0.31}  &  \num{5.327e-01}  &  \num{ 0.31}  &  \num{5.374e-01}  &  \num{ 0.31}  &  \num{4.953e-01}  &  \num{ 0.36}  \\
    800  &  \num{4.325e-01}  &  \num{ 0.30}  &  \num{4.329e-01}  &  \num{ 0.30}  &  \num{4.362e-01}  &  \num{ 0.30}  &  \num{3.948e-01}  &  \num{ 0.33}  \\
   1600  &  \num{3.506e-01}  &  \num{ 0.30}  &  \num{3.509e-01}  &  \num{ 0.30}  &  \num{3.531e-01}  &  \num{ 0.30}  &  \num{3.145e-01}  &  \num{ 0.33}  \\
   3200  &  \num{2.909e-01}  &  \num{ 0.27}  &  \num{2.911e-01}  &  \num{ 0.27}  &  \num{2.922e-01}  &  \num{ 0.27}  &  \num{2.587e-01}  &  \num{ 0.28}  \\
   6400  &  \num{2.378e-01}  &  \num{ 0.29}  &  \num{2.380e-01}  &  \num{ 0.29}  &  \num{2.388e-01}  &  \num{ 0.29}  &  \num{2.080e-01}  &  \num{ 0.31}  \\
  12800  &  \num{1.921e-01}  &  \num{ 0.31}  &  \num{1.923e-01}  &  \num{ 0.31}  &  \num{1.930e-01}  &  \num{ 0.31}  &  \num{1.635e-01}  &  \num{ 0.35}  \\
  \bottomrule
\end{tabular}
\end{center}
\end{table}

\subsection{Slowly Moving Contact Discontinuity}
\label{sec:slowContact-FV}

In this section,  initial condition of the previous test case is used, but with
a non-vanishing initial velocity, resulting in a slowly moving contact discontinuity
as described by \citet[Section 6.4, Test 5]{toro2009riemann}.
The initial condition is given in primitive variables by
\begin{equation}
\label{eq:slowContact}
\begin{aligned}
  \rho_0(x) = 1
  ,\qquad
  v_0(x) = -19.59745
  ,\qquad
  p_0(x) =
  \begin{cases}
    1000, & x < \frac{4}{5}, \\
    0.01, & \text{ else},
  \end{cases}
\end{aligned}
\end{equation}
and the conservative variables are again computed via $\rho v_0 = \rho_0 v_0$ and
$\rho e_0 = \frac{1}{2} \rho_0 v_0^2 + \frac{p_0}{\gamma-1}$.
The solution is computed on the domain $[0, 1]$ from $t = 0$ until $t = 0.012$
using the finite volume methods of section \ref{sec:sod-FV}.

The results are shown in Table \ref{tab:slowContact-FV}. As in the previous section
\ref{sec:blastWaveLeftWC-FV}, the fluxes using the pressure in the density flux
are unstable. The scalar and LLF dissipation fluxes yield similar errors with
some variances across the methods, but the Suliciu solver is superior. In most
cases, it is also better than the scalar dissipation (SD).

\begin{table}[!htp]
\small
\caption{Errors and experimental order of convergence (EOC) for varying number
         of elements $N$ for the slowly moving contact discontinuity
         \eqref{eq:slowContact} using several numerical fluxes.}
\label{tab:slowContact-FV}
\sisetup{
  output-exponent-marker=\text{e},
  round-mode=places,
  round-precision=2
}
\begin{center}
\begin{tabular}{r | rr|rr|rr|rr}
  \toprule
  & \multicolumn{2}{c}{Ch + SD DWGW} 
  & \multicolumn{2}{c}{Ch + MD DWGW} 
  & \multicolumn{2}{c}{Ch + HD DWGW} 
  & \multicolumn{2}{c}{Ch + LLF} 
  \\
  $N$
    & $\norm{\mathrm{err_\rho}}_M$ & EOC
    & $\norm{\mathrm{err_\rho}}_M$ & EOC
    & $\norm{\mathrm{err_\rho}}_M$ & EOC
    & $\norm{\mathrm{err_\rho}}_M$ & EOC
  \\
  \midrule
    100  &  \num{7.843e-01}  &               &  \num{4.822e-01}  &               &  \num{6.482e-01}  &               &  \num{7.854e-01}  &               \\
    200  &  \num{6.638e-01}  &  \num{ 0.24}  &  \num{6.685e-01}  &  \num{-0.47}  &  \num{4.522e-01}  &  \num{ 0.52}  &  \num{6.645e-01}  &  \num{ 0.24}  \\
    400  &  \num{5.774e-01}  &  \num{ 0.20}  &  \num{2.377e-01}  &  \num{ 1.49}  &  \num{4.099e-01}  &  \num{ 0.14}  &  \num{5.769e-01}  &  \num{ 0.20}  \\
    800  &  \num{4.384e-01}  &  \num{ 0.40}  &  \num{3.411e-01}  &  \num{-0.52}  &  \num{2.465e-01}  &  \num{ 0.73}  &  \num{4.382e-01}  &  \num{ 0.40}  \\
   1600  &  \num{3.508e-01}  &  \num{ 0.32}  &  \num{2.320e-01}  &  \num{ 0.56}  &  \num{1.665e-01}  &  \num{ 0.57}  &  \num{3.508e-01}  &  \num{ 0.32}  \\
   3200  &  \num{2.972e-01}  &  \num{ 0.24}  &  \num{1.024e-01}  &  \num{ 1.18}  &  \num{1.401e-01}  &  \num{ 0.25}  &  \num{2.971e-01}  &  \num{ 0.24}  \\
   6400  &  \num{2.360e-01}  &  \num{ 0.33}  &  \num{1.141e-01}  &  \num{-0.16}  &  \num{8.147e-02}  &  \num{ 0.78}  &  \num{2.360e-01}  &  \num{ 0.33}  \\
  12800  &  \num{1.986e-01}  &  \num{ 0.25}  &  \num{8.589e-02}  &  \num{ 0.41}  &  \num{6.570e-02}  &  \num{ 0.31}  &  \num{1.985e-01}  &  \num{ 0.25}  \\
  \bottomrule
  \toprule
  & \multicolumn{2}{c}{$\rho,v,\beta$ (2) + LLF}
  & \multicolumn{2}{c}{$\rho,v,\frac{1}{p}$ + LLF}
  & \multicolumn{2}{c}{$\rho,v,p$ + LLF}
  & \multicolumn{2}{c}{$\rho,v,T$ (1) + LLF}
  \\
  $N$
    & $\norm{\mathrm{err_\rho}}_M$ & EOC
    & $\norm{\mathrm{err_\rho}}_M$ & EOC
    & $\norm{\mathrm{err_\rho}}_M$ & EOC
    & $\norm{\mathrm{err_\rho}}_M$ & EOC
  \\
  \midrule
    100  &  \num{7.856e-01}  &               &          $*$      &               &          $*$      &               &  \num{7.854e-01}  &               \\
    200  &  \num{6.648e-01}  &  \num{ 0.24}  &          $*$      &               &          $*$      &               &  \num{6.645e-01}  &  \num{ 0.24}  \\
    400  &  \num{5.763e-01}  &  \num{ 0.21}  &          $*$      &               &          $*$      &               &  \num{5.769e-01}  &  \num{ 0.20}  \\
    800  &  \num{4.380e-01}  &  \num{ 0.40}  &          $*$      &               &          $*$      &               &  \num{4.382e-01}  &  \num{ 0.40}  \\
   1600  &  \num{3.507e-01}  &  \num{ 0.32}  &          $*$      &               &          $*$      &               &  \num{3.508e-01}  &  \num{ 0.32}  \\
   3200  &  \num{2.969e-01}  &  \num{ 0.24}  &          $*$      &               &          $*$      &               &  \num{2.971e-01}  &  \num{ 0.24}  \\
   6400  &  \num{2.359e-01}  &  \num{ 0.33}  &          $*$      &               &          $*$      &               &  \num{2.360e-01}  &  \num{ 0.33}  \\
  12800  &  \num{1.985e-01}  &  \num{ 0.25}  &          $*$      &               &          $*$      &               &  \num{1.985e-01}  &  \num{ 0.25}  \\
  \bottomrule
  \toprule
  & \multicolumn{2}{c}{$\rho,v,T$ (2) + LLF} 
  & \multicolumn{2}{c}{$\rho,V,T$ (rev) + LLF}
  & \multicolumn{2}{c}{LLF} 
  & \multicolumn{2}{c}{Suliciu} 
  \\
  $N$
    & $\norm{\mathrm{err_\rho}}_M$ & EOC
    & $\norm{\mathrm{err_\rho}}_M$ & EOC
    & $\norm{\mathrm{err_\rho}}_M$ & EOC
    & $\norm{\mathrm{err_\rho}}_M$ & EOC
  \\
  \midrule
    100  &  \num{7.860e-01}  &               &  \num{7.910e-01}  &               &  \num{7.960e-01}  &               &  \num{5.084e-01}  &               \\
    200  &  \num{6.649e-01}  &  \num{ 0.24}  &  \num{6.661e-01}  &  \num{ 0.25}  &  \num{6.768e-01}  &  \num{ 0.23}  &  \num{2.799e-01}  &  \num{ 0.86}  \\
    400  &  \num{5.767e-01}  &  \num{ 0.21}  &  \num{5.772e-01}  &  \num{ 0.21}  &  \num{5.840e-01}  &  \num{ 0.21}  &  \num{2.848e-01}  &  \num{-0.03}  \\
    800  &  \num{4.381e-01}  &  \num{ 0.40}  &  \num{4.382e-01}  &  \num{ 0.40}  &  \num{4.445e-01}  &  \num{ 0.39}  &  \num{1.453e-01}  &  \num{ 0.97}  \\
   1600  &  \num{3.507e-01}  &  \num{ 0.32}  &  \num{3.508e-01}  &  \num{ 0.32}  &  \num{3.548e-01}  &  \num{ 0.33}  &  \num{9.770e-02}  &  \num{ 0.57}  \\
   3200  &  \num{2.970e-01}  &  \num{ 0.24}  &  \num{2.971e-01}  &  \num{ 0.24}  &  \num{2.994e-01}  &  \num{ 0.24}  &  \num{9.025e-02}  &  \num{ 0.11}  \\
   6400  &  \num{2.360e-01}  &  \num{ 0.33}  &  \num{2.360e-01}  &  \num{ 0.33}  &  \num{2.374e-01}  &  \num{ 0.33}  &  \num{4.810e-02}  &  \num{ 0.91}  \\
  12800  &  \num{1.985e-01}  &  \num{ 0.25}  &  \num{1.985e-01}  &  \num{ 0.25}  &  \num{1.994e-01}  &  \num{ 0.25}  &  \num{3.946e-02}  &  \num{ 0.29}  \\
  \bottomrule
\end{tabular}
\end{center}
\end{table}

\subsection{Right Half of the Blast Wave Problem of Woodward and Colella}
\label{sec:blastWaveRightWC-FV}

In this section, the right half of the blast wave problem of
\citet[Section IV.a]{woodward1984numerical} as described by
\citet[Section 4.3.3, Test 4]{toro2009riemann} is considered. In primitive variables,
it is given by
\begin{equation}
\label{eq:blastWaveRightWC}
\begin{aligned}
  \rho_0(x) = 1
  ,\qquad
  v_0(x) = 0
  ,\qquad
  p_0(x) =
  \begin{cases}
    0.01, & x < \frac{1}{2}, \\
    100, & \text{ else},
  \end{cases}
\end{aligned}
\end{equation}
and the conservative variables are again computed via $\rho v_0 = \rho_0 v_0$ and
$\rho e_0 = \frac{1}{2} \rho_0 v_0^2 + \frac{p_0}{\gamma-1}$.
The solution is computed on the domain $[0, 1]$ from $t = 0$ until $t = 0.035$
using again the finite volume methods described in section \ref{sec:sod-FV}.

The results are shown in Table \ref{tab:blastWaveRightWC-FV}. As before, the
fluxes with variables $\rho, v$ and $p$ or $\frac{1}{p}$ are unstable, and
the other fluxes yield similar errors, while the matrix dissipation and Suliciu
solver perform a bit better than the others.

\begin{table}[!htp]
\small
\caption{Errors and experimental order of convergence (EOC) for varying number
         of elements $N$ for the right half \eqref{eq:blastWaveRightWC} of 
         the blast wave problem of \citet{woodward1984numerical} using several
         numerical fluxes.}
\label{tab:blastWaveRightWC-FV}
\sisetup{
  output-exponent-marker=\text{e},
  round-mode=places,
  round-precision=2
}
\begin{center}
\begin{tabular}{r | rr|rr|rr|rr}
  \toprule
  & \multicolumn{2}{c}{Ch + SD DWGW} 
  & \multicolumn{2}{c}{Ch + MD DWGW} 
  & \multicolumn{2}{c}{Ch + HD DWGW} 
  & \multicolumn{2}{c}{Ch + LLF} 
  \\
  $N$
    & $\norm{\mathrm{err_\rho}}_M$ & EOC
    & $\norm{\mathrm{err_\rho}}_M$ & EOC
    & $\norm{\mathrm{err_\rho}}_M$ & EOC
    & $\norm{\mathrm{err_\rho}}_M$ & EOC
  \\
  \midrule
    100  &  \num{7.253e-01}  &               &  \num{6.904e-01}  &               &  \num{6.888e-01}  &               &  \num{7.259e-01}  &               \\
    200  &  \num{6.451e-01}  &  \num{ 0.17}  &  \num{6.112e-01}  &  \num{ 0.18}  &  \num{6.047e-01}  &  \num{ 0.19}  &  \num{6.451e-01}  &  \num{ 0.17}  \\
    400  &  \num{5.259e-01}  &  \num{ 0.29}  &  \num{4.905e-01}  &  \num{ 0.32}  &  \num{4.865e-01}  &  \num{ 0.31}  &  \num{5.257e-01}  &  \num{ 0.30}  \\
    800  &  \num{4.378e-01}  &  \num{ 0.26}  &  \num{4.016e-01}  &  \num{ 0.29}  &  \num{3.990e-01}  &  \num{ 0.29}  &  \num{4.376e-01}  &  \num{ 0.26}  \\
   1600  &  \num{3.575e-01}  &  \num{ 0.29}  &  \num{3.240e-01}  &  \num{ 0.31}  &  \num{3.221e-01}  &  \num{ 0.31}  &  \num{3.573e-01}  &  \num{ 0.29}  \\
   3200  &  \num{2.856e-01}  &  \num{ 0.32}  &  \num{2.531e-01}  &  \num{ 0.36}  &  \num{2.525e-01}  &  \num{ 0.35}  &  \num{2.855e-01}  &  \num{ 0.32}  \\
   6400  &  \num{2.329e-01}  &  \num{ 0.29}  &  \num{2.020e-01}  &  \num{ 0.33}  &  \num{2.028e-01}  &  \num{ 0.32}  &  \num{2.329e-01}  &  \num{ 0.29}  \\
  12800  &  \num{1.877e-01}  &  \num{ 0.31}  &  \num{1.643e-01}  &  \num{ 0.30}  &  \num{1.588e-01}  &  \num{ 0.35}  &  \num{1.881e-01}  &  \num{ 0.31}  \\
  \bottomrule
  \toprule
  & \multicolumn{2}{c}{$\rho,v,\beta$ (2) + LLF}
  & \multicolumn{2}{c}{$\rho,v,\frac{1}{p}$ + LLF}
  & \multicolumn{2}{c}{$\rho,v,p$ + LLF}
  & \multicolumn{2}{c}{$\rho,v,T$ (1) + LLF}
  \\
  $N$
    & $\norm{\mathrm{err_\rho}}_M$ & EOC
    & $\norm{\mathrm{err_\rho}}_M$ & EOC
    & $\norm{\mathrm{err_\rho}}_M$ & EOC
    & $\norm{\mathrm{err_\rho}}_M$ & EOC
  \\
  \midrule
    100  &  \num{7.243e-01}  &               &          $*$      &               &          $*$      &               &  \num{7.259e-01}  &               \\
    200  &  \num{6.443e-01}  &  \num{ 0.17}  &          $*$      &               &          $*$      &               &  \num{6.451e-01}  &  \num{ 0.17}  \\
    400  &  \num{5.254e-01}  &  \num{ 0.29}  &          $*$      &               &          $*$      &               &  \num{5.257e-01}  &  \num{ 0.30}  \\
    800  &  \num{4.375e-01}  &  \num{ 0.26}  &          $*$      &               &          $*$      &               &  \num{4.376e-01}  &  \num{ 0.26}  \\
   1600  &  \num{3.573e-01}  &  \num{ 0.29}  &          $*$      &               &          $*$      &               &  \num{3.573e-01}  &  \num{ 0.29}  \\
   3200  &  \num{2.855e-01}  &  \num{ 0.32}  &          $*$      &               &          $*$      &               &  \num{2.855e-01}  &  \num{ 0.32}  \\
   6400  &  \num{2.328e-01}  &  \num{ 0.29}  &          $*$      &               &          $*$      &               &  \num{2.329e-01}  &  \num{ 0.29}  \\
  12800  &  \num{1.880e-01}  &  \num{ 0.31}  &          $*$      &               &          $*$      &               &  \num{1.881e-01}  &  \num{ 0.31}  \\
  \bottomrule
  \toprule
  & \multicolumn{2}{c}{$\rho,v,T$ (2) + LLF} 
  & \multicolumn{2}{c}{$\rho,V,T$ (rev) + LLF}
  & \multicolumn{2}{c}{LLF} 
  & \multicolumn{2}{c}{Suliciu} 
  \\
  $N$
    & $\norm{\mathrm{err_\rho}}_M$ & EOC
    & $\norm{\mathrm{err_\rho}}_M$ & EOC
    & $\norm{\mathrm{err_\rho}}_M$ & EOC
    & $\norm{\mathrm{err_\rho}}_M$ & EOC
  \\
  \midrule
    100  &  \num{7.245e-01}  &               &  \num{7.279e-01}  &               &  \num{7.349e-01}  &               &  \num{7.003e-01}  &               \\
    200  &  \num{6.443e-01}  &  \num{ 0.17}  &  \num{6.462e-01}  &  \num{ 0.17}  &  \num{6.519e-01}  &  \num{ 0.17}  &  \num{6.159e-01}  &  \num{ 0.19}  \\
    400  &  \num{5.254e-01}  &  \num{ 0.29}  &  \num{5.261e-01}  &  \num{ 0.30}  &  \num{5.302e-01}  &  \num{ 0.30}  &  \num{4.923e-01}  &  \num{ 0.32}  \\
    800  &  \num{4.375e-01}  &  \num{ 0.26}  &  \num{4.379e-01}  &  \num{ 0.26}  &  \num{4.407e-01}  &  \num{ 0.27}  &  \num{4.026e-01}  &  \num{ 0.29}  \\
   1600  &  \num{3.572e-01}  &  \num{ 0.29}  &  \num{3.575e-01}  &  \num{ 0.29}  &  \num{3.590e-01}  &  \num{ 0.30}  &  \num{3.240e-01}  &  \num{ 0.31}  \\
   3200  &  \num{2.855e-01}  &  \num{ 0.32}  &  \num{2.857e-01}  &  \num{ 0.32}  &  \num{2.869e-01}  &  \num{ 0.32}  &  \num{2.540e-01}  &  \num{ 0.35}  \\
   6400  &  \num{2.328e-01}  &  \num{ 0.29}  &  \num{2.330e-01}  &  \num{ 0.29}  &  \num{2.338e-01}  &  \num{ 0.30}  &  \num{2.038e-01}  &  \num{ 0.32}  \\
  12800  &  \num{1.880e-01}  &  \num{ 0.31}  &  \num{1.882e-01}  &  \num{ 0.31}  &  \num{1.890e-01}  &  \num{ 0.31}  &  \num{1.604e-01}  &  \num{ 0.35}  \\
  \bottomrule
\end{tabular}
\end{center}
\end{table}

\subsection{Left Half of the Blast Wave Problem of Derigs, Winters, Gassner and Walch}
\label{sec:DWGWleft-FV}

In this section, the left half of the blast wave problem of
\citet[Section 6]{derigs2016novelAveraging} is considered. In primitive variables,
it is given by
\begin{equation}
\label{eq:DWGWleft}
\begin{aligned}
  \rho_0(x) = 1
  ,\qquad
  v_0(x) = 10
  ,\qquad
  p_0(x) =
  \begin{cases}
    1, & x < -\frac{1}{10}, \\
    10^{-6}, & \text{ else}.
  \end{cases}
\end{aligned}
\end{equation}
Here, $\gamma = \frac{5}{3}$ is used.
The solution is computed on the domain $[-1, 1]$ from $t = 0$ until $t = \num{5.e-2}$
using the finite volume methods of section \ref{sec:sod-FV}.

The results are shown in Table \ref{tab:DWGWleft-FV}.
Designed as a test case to crash the flux of \citet{roe2006affordable}, the fluxes
containing pressure influence in the density flux are unstable. However, there
is nearly no variance across the other fluxes that remain stable, since the problem
needs a very high resolution to capture the solution.

Contrary to the results of \citet{derigs2016novelAveraging} for the MHD equations,
the simple LLF dissipation is enough to stabilise the solution for the Euler
equations in this case and their specially designed dissipation operator does
not show any improvement over the LLF dissipation.

\begin{table}[!htp]
\small
\caption{Errors and experimental order of convergence (EOC) for varying number
         of elements $N$ for the left half \eqref{eq:DWGWleft} of 
         the blast wave problem of \citet{derigs2016novelAveraging} using several
         numerical fluxes.}
\label{tab:DWGWleft-FV}
\sisetup{
  output-exponent-marker=\text{e},
  round-mode=places,
  round-precision=2
}
\begin{center}
\begin{tabular}{r | rr|rr|rr|rr}
  \toprule
  & \multicolumn{2}{c}{Ch + SD DWGW} 
  & \multicolumn{2}{c}{Ch + MD DWGW} 
  & \multicolumn{2}{c}{Ch + HD DWGW} 
  & \multicolumn{2}{c}{Ch + LLF} 
  \\
  $N$
    & $\norm{\mathrm{err_\rho}}_M$ & EOC
    & $\norm{\mathrm{err_\rho}}_M$ & EOC
    & $\norm{\mathrm{err_\rho}}_M$ & EOC
    & $\norm{\mathrm{err_\rho}}_M$ & EOC
  \\
  \midrule
    100  &  \num{4.324e-01}  &               &  \num{4.326e-01}  &               &  \num{4.325e-01}  &               &  \num{4.326e-01}  &               \\
    200  &  \num{3.131e-01}  &  \num{ 0.47}  &  \num{3.126e-01}  &  \num{ 0.47}  &  \num{3.127e-01}  &  \num{ 0.47}  &  \num{3.132e-01}  &  \num{ 0.47}  \\
    400  &  \num{3.082e-01}  &  \num{ 0.02}  &  \num{3.074e-01}  &  \num{ 0.02}  &  \num{3.076e-01}  &  \num{ 0.02}  &  \num{3.083e-01}  &  \num{ 0.02}  \\
    800  &  \num{2.640e-01}  &  \num{ 0.22}  &  \num{2.629e-01}  &  \num{ 0.23}  &  \num{2.631e-01}  &  \num{ 0.23}  &  \num{2.641e-01}  &  \num{ 0.22}  \\
   1600  &  \num{2.834e-01}  &  \num{-0.10}  &  \num{2.822e-01}  &  \num{-0.10}  &  \num{2.824e-01}  &  \num{-0.10}  &  \num{2.836e-01}  &  \num{-0.10}  \\
   3200  &  \num{2.596e-01}  &  \num{ 0.13}  &  \num{2.583e-01}  &  \num{ 0.13}  &  \num{2.584e-01}  &  \num{ 0.13}  &  \num{2.597e-01}  &  \num{ 0.13}  \\
   6400  &  \num{2.240e-01}  &  \num{ 0.21}  &  \num{2.227e-01}  &  \num{ 0.21}  &  \num{2.227e-01}  &  \num{ 0.21}  &  \num{2.240e-01}  &  \num{ 0.21}  \\
  12800  &  \num{1.781e-01}  &  \num{ 0.33}  &  \num{1.766e-01}  &  \num{ 0.33}  &  \num{1.766e-01}  &  \num{ 0.33}  &  \num{1.782e-01}  &  \num{ 0.33}  \\
  \bottomrule
  \toprule
  & \multicolumn{2}{c}{$\rho,v,\beta$ (2) + LLF}
  & \multicolumn{2}{c}{$\rho,v,\frac{1}{p}$ + LLF}
  & \multicolumn{2}{c}{$\rho,v,p$ + LLF}
  & \multicolumn{2}{c}{$\rho,v,T$ (1) + LLF}
  \\
  $N$
    & $\norm{\mathrm{err_\rho}}_M$ & EOC
    & $\norm{\mathrm{err_\rho}}_M$ & EOC
    & $\norm{\mathrm{err_\rho}}_M$ & EOC
    & $\norm{\mathrm{err_\rho}}_M$ & EOC
  \\
  \midrule
    100  &  \num{4.325e-01}  &               &          $*$      &               &          $*$      &               &  \num{4.326e-01}  &               \\
    200  &  \num{3.131e-01}  &  \num{ 0.47}  &          $*$      &               &          $*$      &               &  \num{3.132e-01}  &  \num{ 0.47}  \\
    400  &  \num{3.082e-01}  &  \num{ 0.02}  &          $*$      &               &          $*$      &               &  \num{3.083e-01}  &  \num{ 0.02}  \\
    800  &  \num{2.640e-01}  &  \num{ 0.22}  &          $*$      &               &          $*$      &               &  \num{2.641e-01}  &  \num{ 0.22}  \\
   1600  &  \num{2.834e-01}  &  \num{-0.10}  &          $*$      &               &          $*$      &               &  \num{2.836e-01}  &  \num{-0.10}  \\
   3200  &  \num{2.596e-01}  &  \num{ 0.13}  &          $*$      &               &          $*$      &               &  \num{2.597e-01}  &  \num{ 0.13}  \\
   6400  &  \num{2.239e-01}  &  \num{ 0.21}  &          $*$      &               &          $*$      &               &  \num{2.240e-01}  &  \num{ 0.21}  \\
  12800  &  \num{1.781e-01}  &  \num{ 0.33}  &          $*$      &               &          $*$      &               &  \num{1.782e-01}  &  \num{ 0.33}  \\
  \bottomrule
  \toprule
  & \multicolumn{2}{c}{$\rho,v,T$ (2) + LLF} 
  & \multicolumn{2}{c}{$\rho,V,T$ (rev) + LLF}
  & \multicolumn{2}{c}{LLF} 
  & \multicolumn{2}{c}{Suliciu} 
  \\
  $N$
    & $\norm{\mathrm{err_\rho}}_M$ & EOC
    & $\norm{\mathrm{err_\rho}}_M$ & EOC
    & $\norm{\mathrm{err_\rho}}_M$ & EOC
    & $\norm{\mathrm{err_\rho}}_M$ & EOC
  \\
  \midrule
    100  &  \num{4.325e-01}  &               &  \num{4.326e-01}  &               &  \num{4.328e-01}  &               &  \num{4.331e-01}  &               \\
    200  &  \num{3.131e-01}  &  \num{ 0.47}  &  \num{3.132e-01}  &  \num{ 0.47}  &  \num{3.134e-01}  &  \num{ 0.47}  &  \num{3.130e-01}  &  \num{ 0.47}  \\
    400  &  \num{3.082e-01}  &  \num{ 0.02}  &  \num{3.083e-01}  &  \num{ 0.02}  &  \num{3.086e-01}  &  \num{ 0.02}  &  \num{3.079e-01}  &  \num{ 0.02}  \\
    800  &  \num{2.640e-01}  &  \num{ 0.22}  &  \num{2.641e-01}  &  \num{ 0.22}  &  \num{2.644e-01}  &  \num{ 0.22}  &  \num{2.634e-01}  &  \num{ 0.23}  \\
   1600  &  \num{2.834e-01}  &  \num{-0.10}  &  \num{2.836e-01}  &  \num{-0.10}  &  \num{2.839e-01}  &  \num{-0.10}  &  \num{2.827e-01}  &  \num{-0.10}  \\
   3200  &  \num{2.596e-01}  &  \num{ 0.13}  &  \num{2.597e-01}  &  \num{ 0.13}  &  \num{2.600e-01}  &  \num{ 0.13}  &  \num{2.586e-01}  &  \num{ 0.13}  \\
   6400  &  \num{2.239e-01}  &  \num{ 0.21}  &  \num{2.240e-01}  &  \num{ 0.21}  &  \num{2.241e-01}  &  \num{ 0.21}  &  \num{2.228e-01}  &  \num{ 0.22}  \\
  12800  &  \num{1.781e-01}  &  \num{ 0.33}  &  \num{1.782e-01}  &  \num{ 0.33}  &  \num{1.782e-01}  &  \num{ 0.33}  &  \num{1.766e-01}  &  \num{ 0.33}  \\
  \bottomrule
\end{tabular}
\end{center}
\end{table}

\subsection{Right Half of the Blast Wave Problem of Derigs, Winters, Gassner and Walch}
\label{sec:DWGWright-FV}

In this section, the right half of the blast wave problem of
\citet[Section 6]{derigs2016novelAveraging} is considered. In primitive variables,
it is given by
\begin{equation}
\label{eq:DWGWright}
\begin{aligned}
  \rho_0(x) = 1
  ,\qquad
  v_0(x) = 10
  ,\qquad
  p_0(x) =
  \begin{cases}
    10^{-6}, & x < \frac{1}{10}, \\
    1, & \text{ else}.
  \end{cases}
\end{aligned}
\end{equation}
Here, $\gamma = \frac{5}{3}$ is used.
The solution is computed on the domain $[-1, 1]$ from $t = 0$ until $t = \num{5.e-2}$
using again the same finite volume methods as before.

The results are shown in Table \ref{tab:DWGWright-FV}. The results are similar
to the left half of this problem in section \ref{sec:DWGWleft-FV}: The pressure
influence in the density flux results in unstable schemes while all other methods
yield similar errors and are stable.

\begin{table}[!htp]
\small
\caption{Errors and experimental order of convergence (EOC) for varying number
         of elements $N$ for the right half \eqref{eq:DWGWright} of 
         the blast wave problem of \citet{derigs2016novelAveraging} using several
         numerical fluxes.}
\label{tab:DWGWright-FV}
\sisetup{
  output-exponent-marker=\text{e},
  round-mode=places,
  round-precision=2
}
\begin{center}
\begin{tabular}{r | rr|rr|rr|rr}
  \toprule
  & \multicolumn{2}{c}{Ch + SD DWGW} 
  & \multicolumn{2}{c}{Ch + MD DWGW} 
  & \multicolumn{2}{c}{Ch + HD DWGW} 
  & \multicolumn{2}{c}{Ch + LLF} 
  \\
  $N$
    & $\norm{\mathrm{err_\rho}}_M$ & EOC
    & $\norm{\mathrm{err_\rho}}_M$ & EOC
    & $\norm{\mathrm{err_\rho}}_M$ & EOC
    & $\norm{\mathrm{err_\rho}}_M$ & EOC
  \\
  \midrule
    100  &  \num{4.332e-01}  &               &  \num{4.323e-01}  &               &  \num{4.326e-01}  &               &  \num{4.332e-01}  &               \\
    200  &  \num{3.134e-01}  &  \num{ 0.47}  &  \num{3.128e-01}  &  \num{ 0.47}  &  \num{3.131e-01}  &  \num{ 0.47}  &  \num{3.134e-01}  &  \num{ 0.47}  \\
    400  &  \num{3.085e-01}  &  \num{ 0.02}  &  \num{3.073e-01}  &  \num{ 0.03}  &  \num{3.077e-01}  &  \num{ 0.02}  &  \num{3.085e-01}  &  \num{ 0.02}  \\
    800  &  \num{2.640e-01}  &  \num{ 0.22}  &  \num{2.625e-01}  &  \num{ 0.23}  &  \num{2.628e-01}  &  \num{ 0.23}  &  \num{2.640e-01}  &  \num{ 0.22}  \\
   1600  &  \num{2.829e-01}  &  \num{-0.10}  &  \num{2.798e-01}  &  \num{-0.09}  &  \num{2.804e-01}  &  \num{-0.09}  &  \num{2.830e-01}  &  \num{-0.10}  \\
   3200  &  \num{2.584e-01}  &  \num{ 0.13}  &  \num{2.541e-01}  &  \num{ 0.14}  &  \num{2.548e-01}  &  \num{ 0.14}  &  \num{2.584e-01}  &  \num{ 0.13}  \\
   6400  &  \num{2.227e-01}  &  \num{ 0.21}  &  \num{2.176e-01}  &  \num{ 0.22}  &  \num{2.182e-01}  &  \num{ 0.22}  &  \num{2.227e-01}  &  \num{ 0.21}  \\
  12800  &  \num{1.787e-01}  &  \num{ 0.32}  &  \num{1.730e-01}  &  \num{ 0.33}  &  \num{1.735e-01}  &  \num{ 0.33}  &  \num{1.787e-01}  &  \num{ 0.32}  \\
  \bottomrule
  \toprule
  & \multicolumn{2}{c}{$\rho,v,\beta$ (2) + LLF}
  & \multicolumn{2}{c}{$\rho,v,\frac{1}{p}$ + LLF}
  & \multicolumn{2}{c}{$\rho,v,p$ + LLF}
  & \multicolumn{2}{c}{$\rho,v,T$ (1) + LLF}
  \\
  $N$
    & $\norm{\mathrm{err_\rho}}_M$ & EOC
    & $\norm{\mathrm{err_\rho}}_M$ & EOC
    & $\norm{\mathrm{err_\rho}}_M$ & EOC
    & $\norm{\mathrm{err_\rho}}_M$ & EOC
  \\
  \midrule
    100  &  \num{4.332e-01}  &               &          $*$      &               &          $*$      &               &  \num{4.332e-01}  &               \\
    200  &  \num{3.135e-01}  &  \num{ 0.47}  &          $*$      &               &          $*$      &               &  \num{3.134e-01}  &  \num{ 0.47}  \\
    400  &  \num{3.086e-01}  &  \num{ 0.02}  &          $*$      &               &          $*$      &               &  \num{3.085e-01}  &  \num{ 0.02}  \\
    800  &  \num{2.641e-01}  &  \num{ 0.22}  &          $*$      &               &          $*$      &               &  \num{2.640e-01}  &  \num{ 0.22}  \\
   1600  &  \num{2.831e-01}  &  \num{-0.10}  &          $*$      &               &          $*$      &               &  \num{2.830e-01}  &  \num{-0.10}  \\
   3200  &  \num{2.586e-01}  &  \num{ 0.13}  &          $*$      &               &          $*$      &               &  \num{2.584e-01}  &  \num{ 0.13}  \\
   6400  &  \num{2.229e-01}  &  \num{ 0.21}  &          $*$      &               &          $*$      &               &  \num{2.227e-01}  &  \num{ 0.21}  \\
  12800  &  \num{1.788e-01}  &  \num{ 0.32}  &          $*$      &               &          $*$      &               &  \num{1.787e-01}  &  \num{ 0.32}  \\
  \bottomrule
  \toprule
  & \multicolumn{2}{c}{$\rho,v,T$ (2) + LLF} 
  & \multicolumn{2}{c}{$\rho,V,T$ (rev) + LLF}
  & \multicolumn{2}{c}{LLF} 
  & \multicolumn{2}{c}{Suliciu} 
  \\
  $N$
    & $\norm{\mathrm{err_\rho}}_M$ & EOC
    & $\norm{\mathrm{err_\rho}}_M$ & EOC
    & $\norm{\mathrm{err_\rho}}_M$ & EOC
    & $\norm{\mathrm{err_\rho}}_M$ & EOC
  \\
  \midrule
    100  &  \num{4.332e-01}  &               &  \num{4.333e-01}  &               &  \num{4.333e-01}  &               &  \num{4.323e-01}  &               \\
    200  &  \num{3.135e-01}  &  \num{ 0.47}  &  \num{3.135e-01}  &  \num{ 0.47}  &  \num{3.133e-01}  &  \num{ 0.47}  &  \num{3.127e-01}  &  \num{ 0.47}  \\
    400  &  \num{3.086e-01}  &  \num{ 0.02}  &  \num{3.086e-01}  &  \num{ 0.02}  &  \num{3.084e-01}  &  \num{ 0.02}  &  \num{3.072e-01}  &  \num{ 0.03}  \\
    800  &  \num{2.641e-01}  &  \num{ 0.22}  &  \num{2.640e-01}  &  \num{ 0.22}  &  \num{2.638e-01}  &  \num{ 0.23}  &  \num{2.623e-01}  &  \num{ 0.23}  \\
   1600  &  \num{2.831e-01}  &  \num{-0.10}  &  \num{2.830e-01}  &  \num{-0.10}  &  \num{2.828e-01}  &  \num{-0.10}  &  \num{2.796e-01}  &  \num{-0.09}  \\
   3200  &  \num{2.586e-01}  &  \num{ 0.13}  &  \num{2.584e-01}  &  \num{ 0.13}  &  \num{2.583e-01}  &  \num{ 0.13}  &  \num{2.540e-01}  &  \num{ 0.14}  \\
   6400  &  \num{2.229e-01}  &  \num{ 0.21}  &  \num{2.228e-01}  &  \num{ 0.21}  &  \num{2.227e-01}  &  \num{ 0.21}  &  \num{2.176e-01}  &  \num{ 0.22}  \\
  12800  &  \num{1.788e-01}  &  \num{ 0.32}  &  \num{1.787e-01}  &  \num{ 0.32}  &  \num{1.788e-01}  &  \num{ 0.32}  &  \num{1.730e-01}  &  \num{ 0.33}  \\
  \bottomrule
\end{tabular}
\end{center}
\end{table}

\subsection{Another Blast Wave Problem}
\label{sec:criticalExplosion-FV}

In this section, another blast wave problem is considered. In primitive variables,
it is given by
\begin{equation}
\label{eq:criticalExplosion}
\begin{aligned}
  \rho_0(x) = \frac{1}{10}
  ,\qquad
  v_0(x) = 10
  ,\qquad
  p_0(x) =
  \begin{cases}
    10^{-12}, & x < -\frac{1}{2}, \\
    10^{-3}, & \text{ else}.
  \end{cases}
\end{aligned}
\end{equation}
Here, $\gamma = \frac{5}{3}$ is used.
The solution is computed on the domain $[-1, 1]$ from $t = 0$ until $t = 0.1$
using the same set of FV methods as in the previous test cases.

To the author's knowledge, this test problem has not been used before, and is
designed to show the importance of positivity preserving for the pressure.
As can be seen in the results shown in Table \ref{tab:criticalExplosion-FV},
the new scalar and matrix dissipation operators of \citet{derigs2016novelAveraging,
winters2016uniquely} are not stable for this problem. Indeed, they result in
negative pressures. However, the simple LLF dissipation that has been reported
to be less stable than these dissipation operators by \citet{derigs2016novelAveraging}
for the MHD equations remains stable in this test case.

Of course, the fluxes containing an influence of the pressure in the density flux
are unstable. The remaining fluxes (with LLF dissipation and Suliciu) are all
stable and result in the same error (up to two digits of precision).

As another example demonstrating the positivity preserving issue for the pressure,
explicit Euler FV steps \eqref{eq:FV} using the entropy conservative flux
\eqref{eq:Chandrashekar-rho-v-beta-EC-KEP} of \citet{chandrashekar2013kinetic}
with scalar and matrix dissipation operators by \citet{derigs2016novelAveraging,
winters2016uniquely} as well as LLF dissipation, respectively, have been performed
with the states
\begin{equation}
\label{eq:pos-pres-test}
  \begin{pmatrix}
    \rho \\ v \\ p
  \end{pmatrix}_{i-1}
  = 
  \begin{pmatrix}
    \rho \\ v \\ p
  \end{pmatrix}_{i}
  =
  \begin{pmatrix}
    \num{0.1} \\ \num{10} \\ \num{1.e-12}
  \end{pmatrix},
  \qquad
  \begin{pmatrix}
    \rho \\ v \\ p
  \end{pmatrix}_{i+1}
  =
  \begin{pmatrix}
    \num{10} \\ \num{10} \\ \num{1.e-6}
  \end{pmatrix}.
\end{equation}
As can be seen in Figure \ref{fig:chandrashekar_DWGW}, the pressure becomes
negative for both the scalar [$\frac{\Delta t}{\Delta x} \lessapprox \num{0.3e-12}$]
and the matrix dissipation [$\frac{\Delta t}{\Delta x} \lessapprox \num{0.1e-12}$]
operator, while the characteristic speeds on the left ($i-1, i$) and right ($i+1$)
hand side are $v = 10$, $c_{i} = \sqrt{ \gamma \frac{p_i}{\rho_i} } = \sqrt{10^{-11} \gamma}$,
and $c_{i+1} = \sqrt{10^{-13} \gamma}$. Thus, there does not seem to be a reasonable
CFL condition for these two fluxes and for this initial condition.
Contrary, the LLF dissipation operator results in a positive pressure.

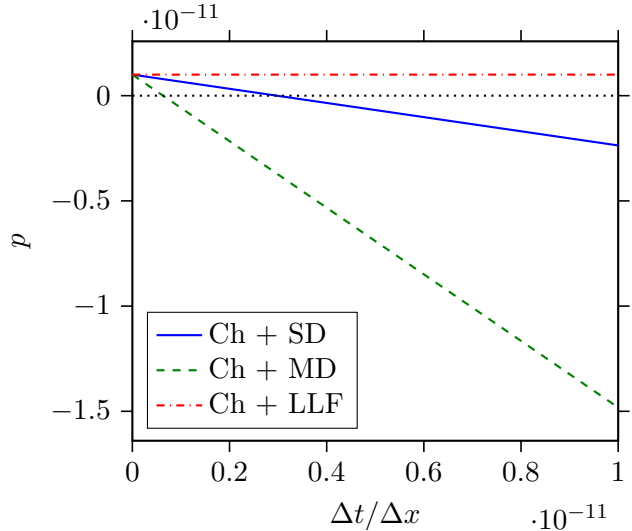
\begin{figure}[!htp]
\centering
\ifx\useTikzForPlotting\undefined
  \includegraphics{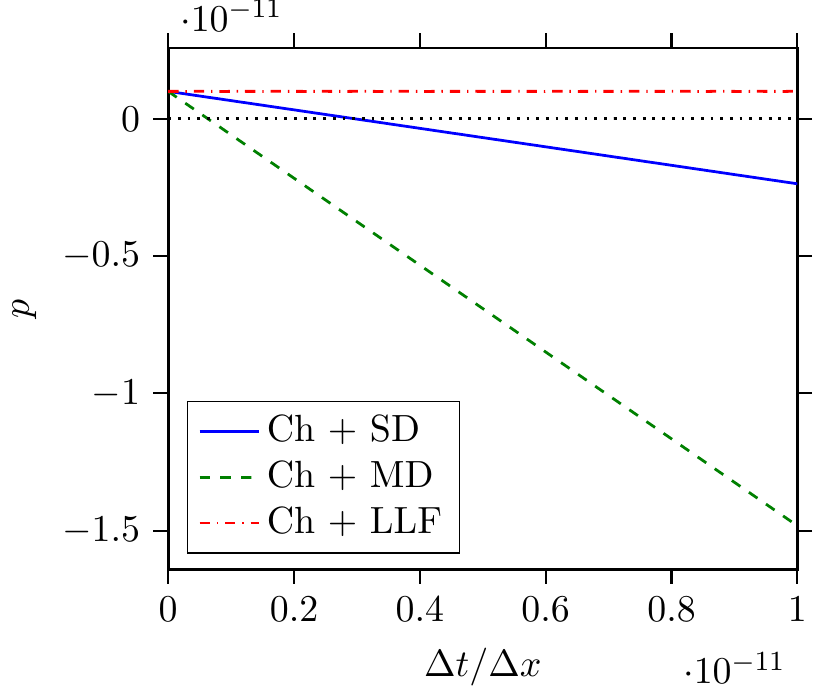}
\else
  \tikzsetnextfilename{report_chandrashekar_DWGW}
  \begin{tikzpicture}
    \definecolor{plt_blue}{HTML}{0000FF}
    \definecolor{plt_green}{HTML}{008000}
    \definecolor{plt_red}{HTML}{FF0000}
    \begin{axis}[
      xlabel={$\Delta t / \Delta x$},
      ylabel={$p$},
      xmin=0,
      xmax=1.e-11,
      legend cell align=left,
      legend pos=south west,
      width=0.5\textwidth,
      axis line style=thick,
      every tick/.style={semithick},
      tick align=outside,
    ]
      \addplot[smooth, plt_blue, solid, thick] table[x index=0, y index=1] {figures/chandrashekar_DWGW.dat};
      \addlegendentry{Ch + SD}
      
      \addplot[smooth, plt_green, dashed, thick] table[x index=0, y index=2] {figures/chandrashekar_DWGW.dat};
      \addlegendentry{Ch + MD}
      
      \addplot[smooth, plt_red, dashdotted, thick] table[x index=0, y index=3] {figures/chandrashekar_DWGW.dat};
      \addlegendentry{Ch + LLF}
      
      \addplot[dotted, domain=0:1.e-11, thick] {0.};
      
    \end{axis}
  \end{tikzpicture}
\fi
  \caption{Pressure depending on the step size ration $\frac{\Delta t}{\Delta x}$
           for one FV step \eqref{eq:FV} using a scalar (solid), matrix (dashed),
           and LLF (dash-dotted) dissipation operator for the initial
           condition \eqref{eq:pos-pres-test}.}
  \label{fig:chandrashekar_DWGW}
\end{figure}

\begin{table}[!htp]
\small
\caption{Errors and experimental order of convergence (EOC) for varying number
         of elements $N$ for the blast wave \eqref{eq:criticalExplosion}
         using several numerical fluxes.}
\label{tab:criticalExplosion-FV}
\sisetup{
  output-exponent-marker=\text{e},
  round-mode=places,
  round-precision=2
}
\begin{center}
\begin{tabular}{r | rr|rr|rr|rr}
  \toprule
  & \multicolumn{2}{c}{Ch + SD DWGW} 
  & \multicolumn{2}{c}{Ch + MD DWGW} 
  & \multicolumn{2}{c}{Ch + HD DWGW} 
  & \multicolumn{2}{c}{Ch + LLF} 
  \\
  $N$
    & $\norm{\mathrm{err_\rho}}_M$ & EOC
    & $\norm{\mathrm{err_\rho}}_M$ & EOC
    & $\norm{\mathrm{err_\rho}}_M$ & EOC
    & $\norm{\mathrm{err_\rho}}_M$ & EOC
  \\
  \midrule
    100  &          $*$      &               &          $*$      &               &          $*$      &               &  \num{1.732e-03}  &               \\
    200  &          $*$      &               &          $*$      &               &          $*$      &               &  \num{5.815e-03}  &  \num{-1.75}  \\
    400  &          $*$      &               &          $*$      &               &          $*$      &               &  \num{4.657e-03}  &  \num{ 0.32}  \\
    800  &          $*$      &               &          $*$      &               &          $*$      &               &  \num{2.543e-02}  &  \num{-2.45}  \\
   1600  &          $*$      &               &          $*$      &               &          $*$      &               &  \num{1.834e-02}  &  \num{ 0.47}  \\
   3200  &          $*$      &               &          $*$      &               &          $*$      &               &  \num{1.832e-02}  &  \num{ 0.00}  \\
   6400  &          $*$      &               &          $*$      &               &          $*$      &               &  \num{1.823e-02}  &  \num{ 0.01}  \\
  12800  &          $*$      &               &          $*$      &               &          $*$      &               &  \num{1.778e-02}  &  \num{ 0.04}  \\
  \bottomrule
  \toprule
  & \multicolumn{2}{c}{$\rho,v,\beta$ (2) + LLF}
  & \multicolumn{2}{c}{$\rho,v,\frac{1}{p}$ + LLF}
  & \multicolumn{2}{c}{$\rho,v,p$ + LLF}
  & \multicolumn{2}{c}{$\rho,v,T$ (1) + LLF}
  \\
  $N$
    & $\norm{\mathrm{err_\rho}}_M$ & EOC
    & $\norm{\mathrm{err_\rho}}_M$ & EOC
    & $\norm{\mathrm{err_\rho}}_M$ & EOC
    & $\norm{\mathrm{err_\rho}}_M$ & EOC
  \\
  \midrule
    100  &  \num{1.732e-03}  &               &          $*$      &               &          $*$      &               &  \num{1.732e-03}  &               \\
    200  &  \num{5.815e-03}  &  \num{-1.75}  &          $*$      &               &          $*$      &               &  \num{5.815e-03}  &  \num{-1.75}  \\
    400  &  \num{4.657e-03}  &  \num{ 0.32}  &          $*$      &               &          $*$      &               &  \num{4.657e-03}  &  \num{ 0.32}  \\
    800  &  \num{2.543e-02}  &  \num{-2.45}  &          $*$      &               &          $*$      &               &  \num{2.543e-02}  &  \num{-2.45}  \\
   1600  &  \num{1.834e-02}  &  \num{ 0.47}  &          $*$      &               &          $*$      &               &  \num{1.834e-02}  &  \num{ 0.47}  \\
   3200  &  \num{1.832e-02}  &  \num{ 0.00}  &          $*$      &               &          $*$      &               &  \num{1.832e-02}  &  \num{ 0.00}  \\
   6400  &  \num{1.823e-02}  &  \num{ 0.01}  &          $*$      &               &          $*$      &               &  \num{1.823e-02}  &  \num{ 0.01}  \\
  12800  &  \num{1.778e-02}  &  \num{ 0.04}  &          $*$      &               &          $*$      &               &  \num{1.778e-02}  &  \num{ 0.04}  \\
  \bottomrule
  \toprule
  & \multicolumn{2}{c}{$\rho,v,T$ (2) + LLF} 
  & \multicolumn{2}{c}{$\rho,V,T$ (rev) + LLF}
  & \multicolumn{2}{c}{LLF} 
  & \multicolumn{2}{c}{Suliciu} 
  \\
  $N$
    & $\norm{\mathrm{err_\rho}}_M$ & EOC
    & $\norm{\mathrm{err_\rho}}_M$ & EOC
    & $\norm{\mathrm{err_\rho}}_M$ & EOC
    & $\norm{\mathrm{err_\rho}}_M$ & EOC
  \\
  \midrule
    100  &  \num{1.732e-03}  &               &  \num{1.732e-03}  &               &  \num{1.732e-03}  &               &  \num{1.733e-03}  &               \\
    200  &  \num{5.815e-03}  &  \num{-1.75}  &  \num{5.815e-03}  &  \num{-1.75}  &  \num{5.815e-03}  &  \num{-1.75}  &  \num{5.818e-03}  &  \num{-1.75}  \\
    400  &  \num{4.657e-03}  &  \num{ 0.32}  &  \num{4.657e-03}  &  \num{ 0.32}  &  \num{4.657e-03}  &  \num{ 0.32}  &  \num{4.660e-03}  &  \num{ 0.32}  \\
    800  &  \num{2.543e-02}  &  \num{-2.45}  &  \num{2.543e-02}  &  \num{-2.45}  &  \num{2.543e-02}  &  \num{-2.45}  &  \num{2.543e-02}  &  \num{-2.45}  \\
   1600  &  \num{1.834e-02}  &  \num{ 0.47}  &  \num{1.834e-02}  &  \num{ 0.47}  &  \num{1.834e-02}  &  \num{ 0.47}  &  \num{1.834e-02}  &  \num{ 0.47}  \\
   3200  &  \num{1.832e-02}  &  \num{ 0.00}  &  \num{1.832e-02}  &  \num{ 0.00}  &  \num{1.832e-02}  &  \num{ 0.00}  &  \num{1.831e-02}  &  \num{ 0.00}  \\
   6400  &  \num{1.823e-02}  &  \num{ 0.01}  &  \num{1.823e-02}  &  \num{ 0.01}  &  \num{1.823e-02}  &  \num{ 0.01}  &  \num{1.823e-02}  &  \num{ 0.01}  \\
  12800  &  \num{1.778e-02}  &  \num{ 0.04}  &  \num{1.778e-02}  &  \num{ 0.04}  &  \num{1.777e-02}  &  \num{ 0.04}  &  \num{1.776e-02}  &  \num{ 0.04}  \\
  \bottomrule
\end{tabular}
\end{center}
\end{table}

\subsection{Summary of the Numerical Results}
\label{subsec:summary-numerical-tests}

There are three main results of these numerical tests. Firstly, none of the entropy
conservative fluxes not including an influence of the pressure in the density flux 
seems to be clearly superior to the others.

Secondly, the entropy conservative volume fluxes result in schemes that are more
robust for discontinuous solutions than the schemes using the other fluxes.
Nevertheless, coupling entropy conservative volume fluxes with dissipative surface
fluxes is not sufficient for strong shocks. Thus, these results should be considered
carefully, since no additional shock capturing mechanisms --- which will be needed
in practice --- have been used.

Finally, enhancing the entropy conservative fluxes not using the pressure in the
density flux by a local Lax-Friedrichs type dissipation $-\frac{\lambda}{2} \jump{u}$
is more robust concerning positivity of the pressure than dissipation operators 
that have been developed for the MHD equations and transferred to the Euler 
equations.

\section{Summary and Conclusions}
\label{sec:summary}

After formulating a general procedure to develop affordable entropy conservative
fluxes, several new numerical fluxes for the Euler equations have been developed 
in sections~\ref{sec:fluxes} and \ref{sec:reversed-fluxes} and compared with
existing ones in two kinds of application.

Firstly, the entropy conservative fluxes can be used as building blocks of entropy
stable high-order schemes using the flux differencing form of Fisher and Carpenter
\cite{fisher2013high}.
In section~\ref{sec:flux-diff}, the high order of accuracy of the flux differencing
form has been proven for consistent and symmetric numerical fluxes, extending the
known theory of \cite{fisher2013high}. Moreover, entropy conservation and stability
has been investigated in a framework of generalised SBP operators applicable to
multiple dimensions and simplex elements. This last extension may be possible, but
to the author's knowledge, there are no SBP operators on simplices in general
fulfilling the conditions used there. Although these may exist, they will probably
require more nodes per element and could therefore be less efficient.

Moreover, numerical tests have been performed using the flux differencing form 
and several different volume fluxes. There does not seem to be any clearly superior candidate
outperforming the other ones in all cases. Whereas for smooth solutions some not
entropy conservative volume fluxes performed better than their entropy conservative
counterparts, this is different for the considered discontinuous solutions. 
Here, the entropy conservative volume fluxes yielded schemes that were more stable,
i.e. that did not crash (due to negative density or pressure or other reasons).
Nevertheless, coupling entropy conservative volume fluxes with dissipative surface
fluxes is not sufficient for strong shocks. Thus, these results should be considered
carefully, since no additional shock capturing mechanisms --- which will be needed
in practice --- have been used.

Secondly, entropy conservative numerical fluxes can be used as surface fluxes in
flux differencing form / discontinuous Galerkin / finite volume methods. There, 
they should be enhanced by additional dissipation operators. 
In section~\ref{sec:surface-fluxes}, positivity preservation has been investigated.
It has been proven that most of the entropy conservative fluxes preserve
non-negativity of the density, if they are enhanced with local Lax-Friedrichs
type dissipation operators.

Moreover, the (scalar, matrix, and hybrid) dissipation operators of Derigs et al. \cite{derigs2016novelAveraging,
winters2016uniquely} have been tested and compared with a simple local Lax-Friedrichs 
dissipation $-\frac{\lambda}{2} \jump{u}$ for the entropy conservative fluxes as
well as with the classical LLF flux and the Suliciu relaxation solver of 
Bouchut~\cite{bouchut2004nonlinear}.
In some problems, the hybrid and matrix dissipation operators yield similar
results regarding stability and accuracy as the Suliciu solver, but they are
less stable in general, as has been demonstrated in section~\ref{sec:criticalExplosion-FV}.
Therefore, the LLF dissipation $-\frac{\lambda}{2} \jump{u}$ seems to be advantageous 
compared to the scalar dissipation operator regarding stability, contrary to the
results of \cite{derigs2016novelAveraging} for the MHD equations, where specifically
tuned dissipation operators were more stable than the LLF dissipation.

However, investigating performance of numerical fluxes, the costs have to be considered.
Here, the implementation has not been optimised for every flux in detail, but
the Suliciu relaxation solver is the second cheapest one after the LLF flux.
The fluxes relying on an entropy conservative baseline flux are significantly
more expansive. Thus, the Suliciu relaxation solver of Bouchut~\cite{bouchut2004nonlinear}
seems to be the best one in this comparison.

There are many open problems. Firstly, the positivity of the pressure using the
LLF dissipation has been observed in all test cases but no analytical proof has
been conducted yet. Another possibility is the addition of dissipation for the
variables $\rho, \rho v, \rho s$ followed by a conversion to the usual conserved 
variables as described by Bouchut~\cite[Section 2.4.6]{bouchut2004nonlinear}.

Moreover, it has still to be investigated thoroughly in what regard the entropy
conservative fluxes as ingredients in the flux differencing framework have
advantages compared to the split forms tested by Gassner et al. \cite{gassner2016split}.
Additionally, it is still unclear, whether there are some superior entropy
conservative fluxes or cheaper ones.

Furthermore, the implications of (semidiscretely) entropy stable schemes have to
be investigated. To the authors' knowledge, there are no general convergence
results about high-order schemes for nonlinear systems of conservation laws in 
several space dimensions. Entropy stability, implying $L_2$ bounds if strict 
positivity of density and pressure are ensured, alone does not suffice to prove
convergence, since such $L_2$ bounds cannot prevent oscillations can in general.

\appendix

\section*{Acknowledgements}
The author would like to thank the anonymous reviewers for their helpful comments.

\printbibliography

\end{document}